\documentclass[bsl]{asl}

\usepackage{epsfig}
\usepackage{latexsym}
\usepackage{amssymb}
\usepackage[all]{xy}
\usepackage{graphicx}

\title[Fifty years of the spectrum problem]{
Fifty years of the spectrum problem:\\
Survey and new results
}

\author{A. Durand}
\revauthor{Durand, Arnaud}

\address{\'Equipe de Logique,
Universit\'e Paris 7, France}

\email{durand@logique.jussieu.fr}

\author{N. D. Jones}
\revauthor{Jones, Neil D.}

\address{DIKU, University of Copenhagen,
Denmark}

\email{neil@diku.dk}

\author{J. A. Makowsky}
\revauthor{Makowsky, Johann A.}

\address{Faculty of Computer Science, Technion, Haifa, Israel}

\email{janos@cs.technion.ac.il}

\author{M. More}
\revauthor{More, Malika}

\address{Univ Clermont 1, LAIC, France}

\email{more@laic.u-clermont1.fr}

\newif\ifskip
\skiptrue

\newtheorem{theorem}{Theorem}[section]
\newtheorem{proposition}[theorem]{Proposition}
\newtheorem{lemma}[theorem]{Lemma}
\newtheorem{corollary}[theorem]{Corollary}

\newtheorem{question}[theorem]{Question}
\newtheorem{openquestion}{Open Question}

\theoremstyle{definition}
\newtheorem{definition}[theorem]{Definition}

\newtheorem{conjecture}[openquestion]{Conjecture}
\newtheorem{obs}[theorem]{Observation}
\newtheorem{remark}[theorem]{Remark}

\newtheorem{example}[theorem]{Example}

\newcommand{\CMSOL}[1]{\textsc{CMSOL}({#1})}
\newcommand{\Cmsol}{\textsc{CMSOL}}

\newcommand{\Fol}{\textsc{FOL}}

\newcommand{\raus}[1]{}

\newcommand{\NE}{\mathrm{NE}}
\newcommand{\coNE}{\mathrm{coNE}}
\newcommand{\E}{\mathrm{E}}
\newcommand{\coE}{\mathrm{coE}}

\renewcommand{\L}{\mathrm{L}}
\newcommand{\NL}{\mathrm{NL}}
\newcommand{\coNL}{\mathrm{coNL}}
\newcommand{\NP}{\mathrm{NP}}
\newcommand{\coNP}{\mathrm{coNP}}
\newcommand{\classP}{\mathrm{P}}

\newcommand{\UN}{\mathrm{UN}}

\newcommand{\LINSPACE}{\mathrm{LINSPACE}}
\newcommand{\NLINSPACE}{\mathrm{NLINSPACE}}
\newcommand{\coNLINSPACE}{\mathrm{coNLINSPACE}}
\newcommand{\LTH}{\mathrm{LTH}}
\newcommand{\PH}{\mathrm{PH}}

\newcommand{\DTime}[1]{\mathrm{DTIME}({#1})}
\newcommand{\NTime}[1]{\mathrm{NTIME}({#1})}
\newcommand{\DSpace}[1]{\mathrm{DSPACE}({#1})}
\newcommand{\NSpace}[1]{\mathrm{NSPACE}({#1})}

\newcommand{\PSPACE}{\mathrm{PSPACE}}

\newcommand{\NRam}[1]{\mathrm{NRAM}({#1})}

\newcommand{\SPEC}{\textsc{Spec}}
\newcommand{\coSPEC}{\mathrm{co}\textsc{Spec}}
\newcommand{\HOSPEC}[1]{\textsc{ho}\!-\!\textsc{spec}_{#1}}
\newcommand{\spec}[1]{\textsc{spec}({#1})}
\newcommand{\MSOL}[1]{\textsc{MSOL}({#1})}
\newcommand{\Msol}{\textsc{MSOL}}
\newcommand{\SOL}[1]{\textsc{SO}({#1})}
\newcommand{\Sol}{\textsc{SO}}
\newcommand{\EMSOL}[1]{\textsc{EMSOL}({#1})}

\newcommand{\FO}{\textsc{FO}}

\newcommand{\RUD}{\textsc{Rud}}
\newcommand{\SRUD}{\textsc{Srud}}

\newcommand{\harp}{\!\upharpoonright\!}

\newcommand{\N}{\mathbb{N}}
\newcommand{\Q}{\mathbb{Q}}

\newcommand{\FSpectra}[2]{{\textsc{f-spec}}_{#1}^{#2}}
\newcommand{\RSpectra}[2]{{\textsc{spec}}_{#1}^{#2}}

\newcommand{\NN}{\mathbb{N}} 

\def\iff{\Longleftrightarrow}
\def\N{{\mathbb{N}}}

\def\Z{{\mathbb{Z}}}

\def\cC{{\cal C}}

\def\cA{\mathcal{A}}
\def\cB{\mathcal{B}}
\def\cC{\mathcal{C}}

\def\bP{{\mathbf P}}
\def\N{{\mathbb N}}
\def\bPSpace{{\mathbf{PSpace}}}
\def\bPH{{\mathbf{PH}}}
\def\bNP{{\mathbf{NP}}}

\newcommand{\comment}[1]{\marginpar{\small\em {#1}}}

\newenvironment{renumerate}{\begin{enumerate}}{\end{enumerate}}
\renewcommand{\theenumi}{\roman{enumi}}
\renewcommand{\labelenumi}{(\roman{enumi})}
\renewcommand{\labelenumii}{(\roman{enumi}.\alph{enumii})}

\begin{document}
\begin{abstract}
In 1952, Heinrich Scholz published a
question in the Journal of Symbolic Logic
asking for a characterization of spectra, i.e., sets of natural numbers
that are the cardinalities of finite models of first order sentences.
G\"unter Asser asked whether the complement of a spectrum is always a spectrum.
These innocent questions turned out to be seminal for the development
of finite model theory and descriptive complexity.
In this paper we survey developments over the last 50-odd years
pertaining to the spectrum problem.
Our presentation follows conceptual developments rather than 
the chronological order.  
Originally a number theoretic problem, it has been approached in terms
of recursion theory, resource bounded complexity theory,  
classification by complexity of the defining sentences,
and finally in terms of structural graph theory.
Although Scholz' question was answered in various ways, Asser's question remains open.
One appendix 
paraphrases the contents of several early and not easily accesible papers
by G. Asser, A. Mostowski, J. Bennett and S. Mo.
Another appendix 
contains a compendium of questions and conjectures which remain open.
\end{abstract}
 
\maketitle
\small
\begin{center}
To be submitted to the Bulletin of Symbolic Logic.
\smallskip  

\today \ \  (version 13.2)
\end{center}


\footnotesize
\tableofcontents
\normalsize
\newpage
\noindent
\begin{flushright}
\small
\begin{tabular}{l}
{\sc Sganarelle}: Ah! Monsieur, c'est un {\bf spectre}: \\
je le reconnais au marcher. \\
{\sc Dom Juan}: Spectre, fant\^ome, ou diable, \\
je veux voir ce que c'est.\\
{\tiny
J.B. Poquelin, dit 
Moli\`ere,
Dom Juan, Acte V, sc\`ene V
}
\end{tabular}
\end{flushright}
\section{Introduction}
\label{intro1}
At the Annual Symposium of the European Association of Computer Science Logic, CSL'05,
held in Oxford in 2005,
Arnaud Durand, Etienne Grandjean and Malika More 
organized a special workshop dedicated to the spectrum problem.
The workshop speakers and the title of their talks where
\begin{itemize}
\item
Annie Chateau (UQAM, Montreal)
\\
The Ultra-Weak Ash Conjecture is Equivalent to the Spectrum Conjecture, and Some Relative Results
\item
Mor Doron (Hebrew University, Jerusalem).
\\
Weakly Decomposable Classes and Their Spectra (joint work with S. Shelah),
\item
Aaron Hunter (Simon Fraser University, Burnaby).
\\
Closure Results for First-Order Spectra: The Model Theoretic Approach
\item
Neil Immerman (University of Massachusetts, Amherst)
\\
Recent Progress in Descriptive Complexity
\item
Neil Jones (University of Copenhagen, Copenhagen)
\\
Some remarks on the spectrum problem
\item
Johann A. Makowsky (Technion--Israel Institute of Technology, Haifa)
\\
50 years of the spectrum problem
\end{itemize}
The organizers and speakers then decided to use the occasion to expand
the survey talk given by J.A. Makowsky into the present survey paper,
rather than publish the talks.

\subsubsection*{Acknowledgements}
We would like to thank 
A. Chateau,
M. Doron,
A. Esbelin,  
R. Fagin, 
E. Fischer,
E. Grandjean,
A. Hunter,
N. Immerman and
S. Shelah
for their fruitful comments which helped our preparation of this survey.
\ 
\\
\ 
\newpage
\section{The Emergence of the Spectrum Problem}
\label{se:intro}
\ 
\\
\ 
\subsection{Scholz's problem}

In 1952, H. Scholz published an innocent
question in the Journal of Symbolic Logic 
\cite{ar:Scholz52}:
\begin{quote}
\small
{\bf 1.}
Ein ungel\"ostes Problem in der symbolischen Logik.
$IK$ sei der Pr\"adikatenkalk\"ul der ersten Stufe mit der
Identit\"at. In $IK$ ist ein Postulatensystem $BP$ f\"ur
die Boole'sche Algebra mit einer einzigen zweistelligen
Pr\"adikatenvariablen formalisierbar.
$\theta$ sei die Konjunktion der Postulate von $PB$.
Dann ist $\theta$ f\"ur endliche $m$ $m$-zahlig erf\"ullbar
genau dann, wennn es ein $n > 0$ gibt, 
soda{\ss} $m=2^n$.

Hieraus ergibt sich das folgende Problem. $H$ sei ein Ausdruck
des $IK$. Unter dem {\em Spectrum von $H$} soll die Menge der
nat\"urlichen Zahlen verstanden sein, f\"ur welche $H$
erf\"ullbar ist. $M$ sei eine beliebige Menge von nat\"urlichen 
Zahlen. Gesucht ist eine hinreichende [hinrerichende] und
notwendige Bedingung daf\"ur, 
da{\ss} es ein $H$ gibt, 
soda{\ss}
$M$ das Spectrum von $H$ ist.
\\
{\em (Received September 19, 1951)}.
\end{quote}
In English:
\begin{quote}
\small
{\bf 1.}
An unsolved problem in symbolic logic.
Let
$IK$ [Identit\"atskalk\"ul]
be the first order predicate calculus 
with identity.
In $IK$ one can formalize an axiom system 
$BP$ [Boole'sche Postulate] 
for Boolean algebras
with only one binary relation variable.
Let $\theta$  be the conjunction of the axioms of $BP$.
Then $\theta$ is satisfiable in a finite domain of $m$ elements
if and only if there is an $n > 0$ such that $m=2^n$.

From this results the following problem.
Let $H$ be an expression of $IK$. 
We call the set of natural numbers, for which $H$ is satisfiable,
the {\em spectrum of $H$}.
Let $M$ be an arbitrary set of natural numbers.
We look for a sufficient and necessary condition
that ensures that there exists an $H$, such that $M$
is the spectrum of $H$.
\\
{\em (Received September 19, 1951)}.
\end{quote}
This question inaugurated a new column of {\em Problems}
to be published in the Journal of Symbolic Logic and
edited by L. Henkin.
Other questions published in the same issue 
were authored by G. Kreisel and L. Henkin.
They deal with a question about 
interpretations of non-finitist proofs dealing with
recursive ordinals and
the no-counter-example interpretation (Kreisel),
the provability of formulas asserting the provability or independence
of provability assertions (Henkin),
and the question whether the ordering principle is equivalent
to the axiom of choice (Henkin).
All in all 9 problems were published, the last in 1956.

The context in which Scholz's question was formulated
is given by the various completeness and incompleteness
results for First Order Logic that were the main concern
of logicians of the period.
An easy consequence of 
G\"odel's classical completeness theorem of 1929
states that
validity of first order sentences in all (finite and infinite)
structures is recursively enumerable,
whereas Church's  and Turing's classical theorems state 
that it is not recursive.
In contrast to this, the following was shown in 1950  by
B. Trakhtenbrot. 
\begin{theorem}[Trakhtenbrot 50\cite{ar:Trakhtenbrot50}]
Validity of first order sentences in all finite structures
(f-validity) is not recursively enumerable, and hence
satisfiability of first order sentences in some finite structure
(f-satisfiability) is not 
decidable, although it is recursively enumerable.
\end{theorem}

\begin{quote}
\small
\_\hrulefill

\noindent
Heinrich Scholz, a German philosopher, 
was born 17. December 1884  in Berlin and died
30. December 1956 in M\"unster.
He was a student of Adolf von Harnack.
He studied in Berlin and Erlangen philosophy and theology
and got his habilitation in 1910 in Berin for the subjects
philosophy of religion and systematic theology.
He received his Ph.D. in 1913 for his thesis
{\em
Schleiermacher and Goethe. 
A contribution to the history of German thought}.
In 1917 he was appointed full professor for philosophy of religion
in Breslau (Wroclaw, today Poland).
In 1919 he moved to Kiel, and from 1928 on he taught in M\"unster. 
From 1924 till 1928 he studied exact sciences and logic
and formed in M\"unster a center for mathematical logic and
foundational studies,
later to be known as {\em the school of M\"unster}.
His chair became in 1936 the first chair for {\em mathematical logic
and the foundations of exact sciences}.
His seminar underwent several administrative metamorphoses
that culminated in 1950  in the creation of the {\em
Institute for mathematical logic and the foundations
of exact sciences}, which he led until his untimely death.
Among his pupils and collaborators we find
W. Ackermann, F. Bachmann, G. Hasenj\"ager, H. Hermes, K. Schr\"oter
and H. Schweitzer.
He was also among the founders
of the German society bearing the same name (DVMLG).
H. Scholz was a Platonist, and he considered mathematical logic
as the foundation of epistemology.
He is credited for his discovery of Frege's estate,
and for making Frege's writing accessible to a wider readership.
Together with his pupil Hasenj\"ager he authored
the monograph {\em Grundz\"uge der Mathematischen Logik},
 published posthumously in 1961.

\noindent
\_\hrulefill
\end{quote}

Thus,
H. Scholz really asked whether one could prove anything
meaningful about f-satisfiablity besides its undecidability.

\ 
\subsection{Basic facts and questions}
In our notation to be used throughout the paper,
H. Scholz introduced the following:

Let $\tau$ be a vocabulary, 
i.e., set of relation and function symbols.
Let $\phi$ be a sentence in some 
logic with equality over a vocabulary $\tau$. 
Unless otherwise stated the logic will be first order logic $FOL(\tau)$.
Sometimes we shall also discuss second order logic,
or a fragment thereof.

\begin{definition}
The {\em spectrum} $\spec{\phi}$ of  $\phi$ 
is the set of finite cardinalities
(viewed as a subset of $\N$),
in which $\phi$ has a model.

We denote by $\SPEC$ the set of spectra of first-order sentences, i.e.,
$$\SPEC = \{\spec{\phi}\ |\ \phi \mbox{\ is a first-order formula} \}$$
\end{definition}

We shall use $S, S_i$ to denote spectra.
For the definition of spectra it does not
matter whether we use function symbols or not.
So, unless otherwise stated, vocabularies will be without
function symbols. However, we shall allow function symbols
when dealing with sentences of special forms.

Clearly, $\spec{\phi} = \emptyset$ if and only if
$\phi$ is not f-satisfiable.
By definition of satisfiability $0$ is never part of a spectrum.
Very often a spectrum is finite, cofinite or even of the form 
$\mathbb{N}^+ = \mathbb{N} - \{0\}$.

\begin{question}
\label{q1}
Is it decidable whether, for a given $\phi$, 
$\spec{\phi}= \mathbb{N}^+$?
\end{question}

As H. Scholz noted, 
(the set of)
powers of $2$ form a spectrum, because they are the cardinalities
of finite Boolean algebras.
Similarly, powers of primes form a spectrum, because they
are the cardinalities of finite fields.
For $a,b \in \mathbb{N}^+$ 
there are many ways to construct a sentence $\phi$ with 
$\spec{\phi}= a +b\mathbb{N}$,
one of which
consists in using one unary function symbol.
With a moment of reflection, one sees that
spectra have the following closure properties.

\begin{proposition}
\label{pr:easyclosures}
Let $S_1$ and $S_2$ be spectra. 
\begin{renumerate}
\item
Then $S_1 \cup S_2, S_1 \cap S_2$
are also spectra. 
\item
Let 
$S_1 + S_2 = 
\{ m +n : m \in S_1, n \in S_2 \}$. Then $S_1 + S_2$ is a spectrum.
\item
Let 
$S_1 \star S_2 = 
\{ m \cdot n : m \in S_1, n \in S_2 \}$. Then $S_1 \star S_2$ is a spectrum.
\end{renumerate}
\end{proposition}

In the spirit of Question \ref{q1}
we can also ask:

\begin{question}
\label{q2}
Which of the following sets
are recursive?
The set of sentences $\phi$ such that
\begin{renumerate}
\item
$\spec{\phi}$ is finite, cofinite.
\item
$\spec{\phi}$ is ultimately periodic.
\item
$\spec{\phi}$ is,
for given $a,b \in \N$ of the form $a+ b\N$.
\item
$\spec{\phi}=S$  for a given set $S \subseteq \N^+$.
\end{renumerate}
\end{question}
We shall answer Questions \ref{q1} and \ref{q2} 
in Section~\ref{prescribed spectra}.

\ 
\subsection{Immediate responses to H. Scholz's problem}
The first to publish a paper in response to H. Scholz's problem
was G. Asser \cite{ar:Asser55}. A. Robinson's review
\cite{misc:Robinson-mr}
 summarizes it as
follows:
\begin{quote}
\small
($\ldots$)
The present paper is concerned with the characterisation of all 
representable sets [=spectra]. A rather intricate necessary and sufficient 
condition is stated for arithmetical function $X(n)$ to be the 
characteristic function of a representable set. 
The condition shows that such a function is elementary in 
the sense of Kalmar.
($\ldots$)
On the other hand, the author establishes that there exist 
non-representable sets whose characteristic function is elementary. 
Examples of representable sets (some of which are by no means obvious) 
are given without proof and the author suggests that further 
research in this field is desirable.
\end{quote}
Asser also noted that his characterization did not
establish whether the complement of a spectrum is a spectrum.

About the same time, A. Mostowski \cite{ar:Mostowski56}
also considered the problem. H. Curry \cite{misc:Curry-mr}
summarizes Mostowski's paper as follows:
\begin{quote}
\small
($\ldots$)
The author proves that for each function $f(n)$ of a class $K$ of 
functions, which is like the class of primitive functions except 
that at each step all functions are truncated above at $n$,
there is a formula $H$ that has a model in a set of $n+1$ 
individuals if and only if $f(n)=0$. 
From this he deduces positive solutions to Scholz's problem in a 
number of special cases. 
\end{quote}
It is usually considered that A. Mostowski really proved 

\begin{theorem}
\label{th:mostowski}
All sets of natural numbers, 
whose characteristic functions are in the second level
of the Grzegorzcyk Hierarchy ${\mathcal E}^2$, are first order spectra.
\end{theorem}
The detailed definitions and contents of this
theorem will be discussed in Section \ref{section recursion}.

In the last 50 years a steady stream of 
papers appeared dealing with spectra of first order
and higher order logics. The problem seems not too important at first
sight. However, some of these papers had considerable 
impact on what is now called Finite Model Theory and 
Descriptive Complexity Theory.

\begin{openquestion}[Scholz's Problem]
\label{scholz-problem}
Characterize the sets of natural numbers
that are first order spectra.
\end{openquestion}
Scholz's Problem, as stated, is rather vague. He asks for a characterization
of a family of subsets of the natural numbers without specifying,
what kind of an answer he had in mind. The answer could be in terms of
number theory, recursion theory, it could be algebraic, or in terms
of something still to be developed.
We shall see in the sequel many solutions to Scholz's Problem,
but we consider it still open, because further answers are still possible.

The same question can be asked for any logic,
in particular second order logic $\Sol$,
or fragments thereof, like monadic second order logic $\Msol$,
fixed point logic, etc., as discussed in \cite{bk:EbbinghausF95,bk:Libkin04}.

\begin{openquestion}[Asser's Problem]
\label{open-assers-problem}
Is the complement of a first order spectrum
a first order spectrum?
\end{openquestion}
Here the answer should be yes or no.

The corresponding problem for $\Sol$ has a trivial solution.
Let  $\phi \in \SOL{\tau}$ with 
$\tau= \{R_1, \ldots, R_k\}$. An integer $n$ is in $\spec{\phi}$ iff 
\[
n \models \exists R_1 \exists R_2 \ldots \exists R_k \phi.
\]
Then, the complement of $\spec{\phi}$ is easily seen 
to be the spectrum of the $\Sol$ sentence 
$\neg (\exists R_1 \ldots \exists R_k \phi)$. 
In passing, note that every $\Sol$-spectrum is a 
$\SOL{\tau}$ over a language $\tau$ containing equality only.

However, for fixed  fragments of $\Sol$, Asser's Problem remains open.
In particular
\begin{openquestion}
\label{asser-msol}
Is the complement of a spectrum of an $\Msol$-sentence
again
a spectrum of an $\Msol$-sentence?
\end{openquestion}

\ 
\subsection{Approaches and themes}
In this survey we shall describe the various solutions
and attempts to
solve Scholz's  and Asser's problems, and the developments these
attempts triggered.
We shall
emphasize more the various ways the questions were approached,
and focus less on the historical order of the papers. 

There are several discernible themes:
\begin{description}
\item[Recursion Theory]
The early authors H. Asser and A. Mostowski
approached the question in the language
of the theory of recursive functions
i.e. they looked for characterization of spectra in terms of
recursion schemes, or hierarchies of recursive functions.
Most prominently in terms of Kalmar's elementary functions,
the Grzegorczyk hierarchy and hierarchies of arithmetical
predicates, in particular {\em rudimentary relations}. 
This line of thought culminates in 1962
in the thesis of J. Bennett \cite{phd:Bennett62}\footnote{
It seems that some of Bennet's unpublished results were rediscovered
independently in China in the late 1980ties by Shaokui Mo \cite{ar:Mo91}.
We shall discuss his work in Section \ref{se:Mo}.
}.
Although G. Asser already characterized first order spectra
in such terms, his characterization was not considered
satisfactory even by himself, 
because it did not use standard terms and was not useful in 
proving that a given set of integers is (or not) a spectrum.
We shall discuss the recursion theoretic
approach in detail in Section \ref{recursion}.
\item[Complexity Theory]
In the 1970s,
D. R\"odding and H. Schwichtenberg
of the M\"unster school
\cite{ar:RoeddingS72} gave a sufficient but not
necessary condition: any set of integers recognizable by a deterministic linear
space-bounded Turing machine is a first-order spectrum. 
(This is also a consequence of results of
Bennett and Ritchie
 \cite{phd:Bennett62,ar:Ritchie63},
 obtained before the emergence of complexity theory.)
Further, R\"odding and Schwichtenberg showed that 
sets of integers recognisable 
using larger space bounds are higher order spectra.
C.\ Christen developed this line further
\cite{phd:Christen74,proc:Christen76},
independently obtaining a number of the following results.

At the same time the spectrum problem
gained renewed interest  in the USA.
In 1972 A. Selman and N. Jones found
an exact solution to Scholz's original question
\cite{proc:JonesS72}: a set of integers is a 
first order spectrum if and only if it is
 recognizable by a non-deterministic Turing machine
in time $O(2^{c \cdot n})$.

This result was independently also obtained by R. Fagin in his thesis
\cite{phd:Fagin73},
which contains an abundance of further results.
Most importantly, R. Fagin studies {\em generalized spectra},
which are the projective classes of Tarski, restricted
to finite structures, and really laid the foundations
for Finite Model Theory and Descriptive Complexity,
as can be seen in the monographs 
\cite{bk:Immerman99,bk:EbbinghausF95,bk:Libkin04}.
We shall discuss the complexity theoretic
approach in detail in Section \ref{complexity}.
\item[Images and preimages of spectra]
From Proposition \ref{pr:easyclosures} it follows that,
if $S$ is a first order spectrum and $p$ is a polynomial
with positive coefficents, then $p(S)=\{p(m): m \in S\}$
is also a spectrum. In J. Bennett's thesis
it is essentially proved that there is a first order spectrum $S$
and an integer $k$
such that $\{n : 2^{n^k} \in S\}$ is not a spectrum.
It is natural to ask what happens to a spectrum
under images and preimages of number theoretic functions.
The general line of this type of results states that certain
images or preimages of spectra of specific forms
of sentences are or are not spectra of other specific forms
of sentences.
%
\item[Spectra of syntactically restricted sentences]
Already in a paper by L. L\"owenheim  from 1915
\cite{ar:Loewenheim15}
it is noted that, what later will be called the spectrum
of a sentence in monadic second order logic ($\Msol$)
with unary relation symbols only, is finite or cofinite.
The set of even numbers is the spectrum 
of an $\Msol$ sentence 
with one binary relation symbol, and it is ultimately periodic.
Further, {\em every}  ultimately periodic set of positive integers
is a spectrum of a first order $\Msol$ sentence with one unary function
symbol. 
\raus{If we allow one binary relation symbol in first order logic,
one can see that the even numbers form a spectrum that is
ultimately periodic.
On the other hand, every ultimately periodic set of positive integers
is a spectrum of a first order sentence of one unary function
symbol.}
Over the last fifty years various papers were written
relating restrictions on the use of relation and function symbols,
or other syntactic restrictions,
to special forms of spectra.
R. Fagin, in his thesis, poses the following problem

\begin{openquestion}[Fagin's Problem for binary relations]
\label{Fagin-binary}
Is every first order
spectrum the spectrum of a first order sentence of
one binary relation symbol?
\end{openquestion}
The question is even open, if 
restricted to
any fixed vocabulary that contains at least one
binary relation symbol or two unary function symbols.

Much of this line of research is motivated by attempts to
solve Fagin's problem. 
%
\item[Transfer theorems]
Another way of studying spectra
is given by the following result, again from Fagin's thesis:
If $S$ is a spectrum of a
purely relational sentence where all the predicate symbols have arity
bounded by $k$, then $S^k =\{ m^k : m \in S\}$
is a spectrum of a sentence with one binary relation symbol only,
or even a spectrum on simple graphs.
One can view this an approach combining the
study of images and preimages of spectra with 
either syntactically or semantically restricted spectra.
Over the years quite a few results along this line were published.
We shall discuss the last three approaches under
the common theme of restrictions on vocabularies 
in detail in Section \ref{vocabularies}.
\item[Spectra of semantically restricted classes]
R. Fagin shows that Asser's problem has a positive answer
if and only if it has a positive answer if restricted to
the class  of simple graphs.
Similarly, in order to understand Fagin's problem better, one
could consider restricted graph classes $K$, 
and study first order spectra restricted to
graphs in $K$. One may think of graphs of
bounded degree, planar graphs, trees,
graphs of tree-width at most $k$, etc.

\begin{openquestion}[Fagin's Problem for simple graphs]
\label{Fagin-simple}
Is every first order
spectrum the spectrum of a first order sentence 
over simple graphs?
\end{openquestion}

\begin{openquestion}
\label{Fagin-planar}
Is every first order
spectrum the spectrum of a first order sentence 
over planar graphs?
\end{openquestion}
For restrictions to graph classes of bounded tree-width,
the answer is negative. 
The reason for this is that spectra of graphs of bounded tree-width
are ultimately periodic.
In fact, this holds for a much wider class of spectra.
E. Fischer and J.A. Makowsky, \cite{ar:FischerM04},
have analyzed 
under what conditions
$\Msol$-spectra are ultimately periodic.
We shall discuss their results in detail in Section \ref{semantic}.

This line of thought has not been extensively  explored,
this may well be a fruitful avenue for 
studying spectra in the future.
\end{description}
In the sequel of this survey we shall summarize what is known
about spectra along these themes.
Various solutions to
Scholz's Problem  were  offered in the literature,
varying with the tastes of the times, but there may be still
more to come.
Asser's and Fagin's Problems are still open.
Both problems are intimately related to our understanding
of definability hierarchies in Descriptive Complexity Theory.
They may well serve as benchmarks of our understanding.

\section{Understanding Spectra: counting functions and number theory}
\label{se:counting}
In this section we formulate various ways to test our understanding
of spectra. It will turn out that there still many questions
we do not know how to answer.

\ 
\subsection{Representation of spectra and counting functions}
Spectra are sets of positive natural numbers.
These sets can be represented in various ways. We shall use
the following:

\begin{definition}
Let $M \subseteq \N^+$, and let
$m_1, m_2, \ldots  $ an enumeration of $M$ ordered
by the size of its elements.
\begin{enumerate}
\item
$\chi_M(n)$ is the characteristic function of $M$, i.e.,
$$
\chi_M(n) =
\begin{cases}
1 & \mbox{ if } n \in M \\
0 & \mbox{ else }.
\end{cases}
$$
\item
$\eta_M(n)$ is the enumeration function of $M$, i.e.,
$$
\eta_M(n) =
\begin{cases}
m_n & \mbox{ if it exists }\\
0 & \mbox{ else }.
\end{cases}
$$
\item
$\gamma_M(n)$ is the counting function of $M$, i.e.,
$\gamma_M(n)$ is the number of elements in $M$
that are strictly smaller than $n$.
\item
A {\em gap} of $M$ is a pair of integers $g_1, g_2$ such that
$g_1, g_2 \in M$ but for each $n$ with $g_1 < n < g_2$ we have that
$n \not\in M$.
Now let $\delta_M(n)$ be the length of the $n$th gap of $M$.
Clearly, $\delta_M(n)= \eta_M(n+1)-\eta_M(n)$.
\end{enumerate}
\end{definition}

Obvious questions are of the following type:
\begin{openquestion}
\label{counting}
Which strictly increasing sequences of positive integers,
are enumerating functions of spectra?
For instance, how fast can they grow?
\end{openquestion}
\begin{openquestion}
\label{gaps}
If $M$ is a spectrum how can $\delta_M(n)$ behave?
\end{openquestion}
Coding runs of Turing machines one can easily obtain
the following.
\begin{proposition}
\label{pr:gaps}
For every recursive monotonically increasing function $f$
there is a first order formula $\phi$ such that $\delta_{\phi}(n)=f(n)$.
\end{proposition}

Various other partial answers to these questions will appear throughout our
narrative.

\subsection{Prime numbers}
An obvious question is whether the primes form a spectrum.
If one gets more ambitious one can ask for special sets of primes
such as Fermat primes (of the form $2^{2^n}+1$), 
Mersenne primes (of the form $2^p-1$ with $p$ a prime), or
the set of primes $p$ such that $p+2$ is also a prime (twin primes).
Even if we do not know whether such a set is finite,
which is the case for twin primes,
it may still be possible to prove that it is a spectrum.
The answer to all these question is yes, 
because all these sets are easily proved to be rudimentary, see Section \ref{recursion}.

In the sense of the above definitions
we have $\chi_{primes}$ is the charactersitic function
of the set of primes,
$\eta_{primes}(n)= p_n$,  and $\gamma_{primes}(n)$
is the counting function of the primes, usually denoted by $\pi(n)$.
$\delta_{primes}(n)$ is usually denoted by $d_n$.
All these functions related to primes are
subject to intensive study in the literature, see eg. \cite{bk:Ribenboim89}.
As we have said that the primes form a first order spectrum,
all the features of these functions observed on primes
do occur on spectra.

For instance,
$\pi(n)$
is approximated by the integral logarithm $li(n)$,
and it was shown by J.E. Littlewood in 1914, cf. \cite{bk:Ribenboim89}
that $\pi(n) - li(n)$  changes sign infinitely many often.
For logical aspects of Littlewood's theorem, see
\cite{ar:Kreisel52}.

Let us define 
\begin{gather}
\pi^+ =\{ n : \pi(n) - li(n) > 0 \}
\notag \\
\pi^- =\{ n : \pi(n) - li(n) \geq 0 \}
\notag
\end{gather}
A less obvious question concerning spectra and primes is
\begin{openquestion}
\label{littlewood}
Are the sets
$\pi^+$ and $\pi^-$ spectra?
\end{openquestion}

\ 
\subsection{Density functions}
Many combinatorial functions are defined by 
linear or polynomial recurrence relations. 
Among them we have the powers of $2$, factorials, 
the Fibonacci numbers, Bernoulli numbers,
Lucas numbers, Stirling numbers and many more,
cf. \cite{bk:GrahamKP89}.
\begin{question}
Are the sets of values of these combinatorial functions
first order spectra? 
\end{question}
The answer will be yes in all of these cases.
We shall sketch a proof in Section \ref{subsection bennett}
that is based on the existence of such recurrence relations.

But these functions also allow combinatorial interpretations
as counting functions: The powers of $2$ count subsets,
the factorials count linear orderings, the Stirling numbers
are related to counting equivalence relations.
We shall see below that
in these three examples
the combinatorial definitions allow
us to give alternative proofs that these sets of
numbers are first order spectra.

The spectrum of a sentence $\phi$ witnesses the existence
of models of $\phi$ of corresponding cardinalities.
Instead, one could also ask for the number
of ways the set $\{0,1, \ldots, n-1\}=[n]$
can be made into a model of  $\phi$.
Alternatively one could count models up to isomorphisms
or up to some other equivalence relation.

Combinatorial counting functions come in different flavours;
\begin{definition}
\label{def-counting}
Let $\mathcal{C}$ be a class of finite $\tau$-structures.
With $\mathcal{C}$ we associate the following counting functions:
\begin{renumerate}
\item
$f_{\mathcal{C}}(n)$  is the number of ways one can interpret the
relation symbols of $\tau$ on the universe $[n]$ such that 
the resulting structure is in $\mathcal{C}$.
This corresponds to counting labeled structures.
\item
Let $Str(\tau)(n)$ denote the number of labeled $\tau$-structures
of size $n$. We put
$$
{\mathit prob}_{\mathcal{C}}(n) = 
\frac{f_{\mathcal{C}}(n)}{Str(\tau)(n)}
$$
which can be interpreted as the probability that
a labeled $\tau$-structure of size $n$
is in $\mathcal{C}$.
\item
$f_{\mathcal{C}}^{iso}(n)$ is the number of non-isomorphic
models in $\mathcal{C}$ of size $n$.
\item
For an equivalence relation $E$ on $\mathcal{C}$ we denote by
$f_{\mathcal{C}}^{E}(n)$ the number of non-E--equivalent
models in $\mathcal{C}$ of size $n$.
\item
\label{Ash-f}
If $E$ is the $k$-equivalence from Ehrenfeucht-Fra\"{\i}ss\'e
games, $f_{\mathcal{C}}^E(n)$ is denoted by
$N_{\mathcal{C},k}(n)$, and is called an Ash-function,
cf. \cite{ar:Ash94} and Section \ref{Ash}.
\item
If $\mathcal{C}$ consists of all the finite models
of a sentence $\phi$ we write 
$f_{\phi}^E(n)$  instead of
$f_{\mathcal{C}}^E(n)$. 
Similarly for ${\mathit prob}_{\phi}(n)$.
\end{renumerate}
\end{definition}

Counting labeled and non-labeled structures has a rich literature,
cf. \cite{bk:HararyP73,bk:Wilf90}.
Note that
counting non-labeled non-isomorphic structures is in general
much harder than the labeled case. 
The first connection between counting labeled structures
and logic is the celebrated 0-1 Law for first order logic:

\begin{theorem}[0-1 Laws]
\label{01 laws}
For every first order sentence $\phi$ over a purely relational
vocabulary $\tau$ we have:
\begin{renumerate}
\item
(Y. Glebskii, D. Kogan, M. Liogonki and V. Talanov \cite{ar:GlebskijKLT69}; 
R. Fagin \cite{phd:Fagin73})
$$
\lim_{n \rightarrow \infty} {\mathit prob}_{\phi}(n) =
\begin{cases}
0 & \\
1 &
\end{cases}
$$
and the limit always exists.
\item
(E. Grandjean \cite{ar:Grandjean83a})
Furthermore, the set of sentences $\phi$ such that
$\lim_{n \rightarrow \infty} {\mathit prob}_{\phi}(n) =1$
is decidable, and in fact $\bPSpace$-complete. 
\end{renumerate}
\end{theorem}

What we are interested in here, is the relationship of 
such counting functions to spectra.
Our example of powers of $2$ shows that
$$
2^n 
= f_{\phi}(n)
= \eta_{\psi}(n) 
$$
where $\phi$ is an always-true first order sentence with one unary relation symbol,
and $\psi$ is the conjunction of the axioms of Boolean algebras.
Similarly,
$$
n!
= f_{\phi_{LIN}}(n)
= \eta_{\psi}(n) 
$$
where $\phi_{LIN}$ are the axioms of linear orders,
and $\psi$ describes the following situation:
\begin{renumerate}
\item
$P$ is a unary relation and $R$ is an linear order on $P$.
\item
$E$ is a ternary relation that is a bijection between
the universe (first argument $x$) and all the linear
orderings on $P$ (remaining two arguments $y,z$).
\item
First 
we say that there is an $x$ that corresponds to $R$;
and that for $x \neq x'$ the orderings are different.
This says that $E$ is injective.
To ensure that we get all the orderings on $P$
we say that for every ordering and every transposition
of two elements in this ordering, there is a corresponding
ordering.
\end{renumerate}
Hence, the size of the model of $\psi$ is the number of
linear orderings on $P$.

Clearly, if 
$f_{\phi}$
is not strictly increasing, there is no $\psi$
with $f_{\phi}(n) = \eta_{\psi}(n)$.
For instance,
for $\phi$ which says that some function is a bijection of 
a part of the universe to its complement, we have
$$
f_{\phi}(n) = 
\begin{cases}
{2m \choose m} \cdot m! & \mbox{ if } n =2m
\\
0 & \mbox{ else }
\end{cases}
$$
\begin{openquestion}
\label{lcounting}
Let $\phi$ a first order sentence, and 
$f_{\phi}$
be the associated
labeled counting function that is monotonically increasing.
Is there a first order sentence $\psi$ such that
for all $n$
$$
f_{\phi}(n) = \eta_{\psi}(n)
$$
\end{openquestion}
The converse question seems more complicated. For instance,
as we have noted before, the primes $p_n$
are of the form $\eta_{\psi}$  for some first order $\psi$,
but we are not aware of any labeled counting function
that will produce the primes.

R. Fagin \cite{proc:Fagin74} calls
$\phi$ {\em categorical} if $f_{\phi}^{iso}(n) \leq 1$
for every $n$. For instance, $\phi_{LIN}$ is categorical.
The counting function up to isomorphisms 
can be bounded by any finite number $m$, 
using disjunctions of different categorical sentences.
So it may be less promising to study for which first order
sentences
$\phi$  there is a $\psi$ such that
$f_{\phi}^{iso}(n) = \eta_{\psi}(n)$, or vice versa.

Surprisingly enough, C. Ash \cite{ar:Ash94}
has found a connection between Asser's Problem
and the behaviour of the Ash functions defined in 
Definition \ref{def-counting}(\ref{Ash-f}).
We shall discuss this in Section \ref{Ash}.

\ 
\subsection{Sentences with prescribed spectra}~\label{prescribed spectra}
In the light of Theorem~\ref{01 laws}
we note that if ${\mathit prob}_{\phi}(n)$
tends to $1$ then $spec(\phi)$ is cofinite.
Obviously, the converse does not hold,
because there are categorical sentences with
models in all finite cardinalities.

Trakhtenbrot's Theorem says that it is undecidable whether
a spectrum is empty, and Grandjean's Theorem says that
it is decidable, whether a sentence  is almost always true,
i.e.  ${\mathit prob}_{\phi}(n)$ tends to $1$.
As a partial answer to Questions \ref{q1} and \ref{q2} we have:
\begin{proposition}
Let $\phi$ be a first order sentence.
The following are undecidable:
\begin{renumerate}
\item
$\spec{\phi}$ is finite, cofinite.
\item
$\spec{\phi}$ is ultimately periodic.
\item
$\spec{\phi}$ is,
for given $a,b \in \N$ of the form $a+ b\N$.
\item
$\spec{\phi}=S$  for a given set $S \subseteq \N^+$.
\end{renumerate}
\end{proposition}
\begin{proof}[Sketch of Proof]
\raus{
\marginpar{
Malika check the proof!

We also distinguished two case
in some margin notes of version 11.

A: Fix $X \subset P(N)$
\\
is $spec(\phi) \in X$?

B: Fix $S \subseteq N$.
Given $\phi$ is $spec(\phi)=S$?
}
}

Let $\varphi\in\FO$. 
We describe the construction of $\FO$-sentences $\psi_1$, $\psi_2$, $\psi_3$, $\psi_4$ and $\psi_5$ such that $\spec{\varphi}=\emptyset$ if and only if:
\begin{itemize}
\item[-] $\spec{\psi_1}$ is finite.
\item[-] $\spec{\psi_2}=\N^+$.
\item[-] More generally, $\spec{\psi_3}=\{f(i)\ |\ i\in\N^+\}$ for a given function $f$ such that $f(i)\geq i$ for all $i$ and the graph $n=f(i)$ seen as a binary relation is rudimentary (see Section \ref{rudimentary} for a precise definition).
\item[-] $\spec{\psi_4}$ is cofinite.
\item[-] $\spec{\psi_5}$ is ultimately periodic.
\end{itemize}
Since the problem of emptiness of spectra is undecidable, the announced result follows.

Let $0$ and $\max$ be two constant symbols, let $\leq$ be a binary predicate symbol and let $+$ and $\times$ be two ternary predicate symbols. Let $Arithm(0,\max,\leq,+,\times)$ denote a first-order sentence axiomatizing the usual arithmetic predicates. Our sentences $\psi_i$ ($i=1,\ldots,3$) consist of the conjunction of $Arithm(0,\max,\leq,+,\times)$ with a specific part $\psi'_i$ that we will describe below. 
We shall use the fact that the $Bit$ predicate\footnote{$Bit(a,b)$ is true iff the bit of rank $b$ of $a$ is $1$.} is definable from $+$ and $\times$ in finite structures, as well as the ternary relation $a=b^c$.  
For simplicity, we will assume w.l.o.g. that the signature of $\varphi$ consists of a binary relation $R$ only. Let $n$ and $y$ be new variable symbols. Let $\varphi'(n,y)$ be the formula obtained from $\varphi$ by replacing every quantification $\forall x$ by $\forall x<n$ and $\exists x$ by $\exists x<n$, and every atomic formula $R(x,x')$ by $Bit(y,x+nx')$. The idea is that a graph $R$ on a set of $n$ elements seen as $\{0,\ldots,n-1\}$ is encoded by the number $y<2^{n^2}$ written in binary with a $1$ in position $a+bn$ if and only if $R(a,b)$ holds. Hence for all $n\in\N^+$, we have $\exists y<2^{n^2}\varphi'(n,y)$ if and only if $\varphi$ has a model with $n$ elements.

\begin{itemize}
\item[-] Let $\psi'_1\equiv \exists m,n,y<\max (\max=3^m\times2^{n^2}\wedge y<2^{n^2}\wedge \varphi'(n,y))$.

It is easy to verify that if $\spec{\varphi}=\emptyset$, then $\spec{\psi_1}$ is also empty (hence finite), and conversely, if $\spec{\varphi}\neq\emptyset$, then $\spec{\psi_1}$ contains all the integers of the form $3^m\times2^{n^2}$ for some $m\in\N^+$ and $n\in\spec{\varphi}$, i.e. $\spec{\psi_1}$ is infinite.

\item[-] Let $\psi'_2\equiv(\forall n<\max\ \max\neq2^{n^2})\vee (\exists n<\max (\max=2^{n^2}\wedge \forall y<2^{n^2} \neg\varphi'(n,y)))$.

If $\spec{\varphi}=\emptyset$, then for all $n$ and $y$, the condition $\exists y<2^{n^2}\varphi'(n,y)$ is false, hence $\spec{\psi_2}=\N^+$. Conversely, if $\spec{\varphi}\neq\emptyset$, then the integers of the form $2^{n^2}$ with $n\in\spec{\varphi}$ are not in $\spec{\psi_2}$.

\item[-] Since the binary relation $y=f(x)$ is rudimentary, it is definable from $+$ and $\times$ in finite structures.
Let $\psi'_3\equiv \exists i<\max(\max=f(i)\wedge((\forall n<i\ i\neq2^{n^2})\vee (\exists n<i (i=2^{n^2}\wedge \forall y<2^{n^2} \neg\varphi'(n,y)))))$.

The verification that $\spec{\varphi}=\emptyset$ if and only if $\spec{\psi_3}=\{f(i)\ |\ i\in\N^+\}$ is similar to the previous case.

\item[-] Let $\psi'_4\equiv\forall n<\max(2^{n^2}\leq\max\longrightarrow\forall y<2^{n^2}\neg\varphi'(n,y))$.

If $\spec{\varphi}=\emptyset$, then for all $n$ and $y$, the condition $\exists y<2^{n^2}\varphi'(n,y)$ is false, hence $\spec{\psi_4}=\N^+$ (hence is cofinite). Conversely, if $\spec{\varphi}\neq\emptyset$, then the integers greater than $2^{n^2}$ with $n\in\spec{\varphi}$ are not in $\spec{\psi_2}$, which is not cofinite.

\item[-] Let $\psi'_5\equiv \exists n,m<\max (\max=2^{n^2}\times m^2\wedge \exists y<2^{n^2}\varphi'(n,y) \wedge \forall n'<n\forall y'<2^{n^2} \neg\varphi'(n',y')))$.

If $\spec{\varphi}=\emptyset$, then $\spec{\psi_5}=\emptyset$ (hence is ultimately periodic). Conversely, observe that $\spec{\psi_5}=\{n_0\times m^2\ |\ m\in\N+\}$, where $n_0=\inf(\spec{\varphi})$. Hence $\spec{\psi_5}$ is not ultimately periodic.
\end{itemize}
\end{proof}

\ 
\subsection{Real numbers and spectra}

Let $\chi_{\phi}(n)$ be the characteristic function of
the spectrum of
a first order sentence $\phi$.
We can associate with $\phi$ and $a \in \Z$ the real number
$r_{\phi}= a+ \sum_n \chi_{\phi}(n) 2^{-n}$.
\begin{definition}
A real number is {\em first order spectral} 
if it is of the form $r_{\phi}$ for some $a \in \Z$ and some first order
sentence $\phi$.
\end{definition}
As we have noted ultimately periodic sets of 
natural numbers are first order spectra, and correspond to
rational numbers. Also every ultimately periodic spectrum
can be realized by a formula with one
function symbol only. We have

\begin{proposition}
Every rational number $q$ is first order spectral
using a formula with one function symbol only.
\end{proposition}

\begin{question}
Do the the first order spectral reals form a field?
\end{question}

E. Specker, \cite{ar:Specker49},
proved that 
there is a real $x$ primitive recursive in
base $2$ such
that $3x$, $x + \frac{1}{3}$, $x^2$ are not primitive recursive 
in base $2$. A modern treatment can be found in \cite{ar:ChenSuZheng2007}. 
H. Friedman \cite{misc:Friedman06} mentioned on an internet discussion
site that primitive recursive can be replaced in Specker's Theorem
by much lower complexity within the Grzegotczyk Hierarchy.
J. Miller kindly provided us, \cite{email:miller}, 
with the more precise statement
\begin{theorem}[E. Specker, 1949 and H. Friedman 2003]
\label{thm:friedman}
\ \\
There is a real $x$ which is primitive recursive in base $2$
(and can even be taken
to be in $\mathcal{E}^2$ of the Grzegorczyk Hierarchy),  
such that $3x$, $x + \frac{1}{3}$, $x^2$ are not primitive recursive 
in base $2$.
\end{theorem}
\begin{proof}[Sketch of proof\footnote{due to J. Miller}]
This can be proved by exploiting the fact that none of these functions
$3x, x^2, x+1/3$ are continuous as functions on binary expansions of
reals. They can take a number $x$ that is not a binary rational to one
that is.

Let us focus on $3x$. So, for example, if $x = 0.0101010101 \ldots$, 
then $3x = 1$. We can exploit this as follows. 
Say we have built the binary
expansion of $x$ up to position $n-1$ and it looks like $0.b$ 
(where $b$ is a finite string) and that we want to 
diagonalize against the $i$th primitive
recursive function $p_i$. 
Compute $p_i(n)$, step by step. 
As long as it
does not converge, keep building $x$ to look like $0.b0010101010101 \ldots$.
If $p_i(n)$ converges at stage $s$, then use position $n+2s$ 
or position $n+2s+1$
to spring the delayed trap. If $p_i(n)=0$, then let 
$x = 0.b0010101 \ldots 01011$. If $p_i(n)=1$, then let 
$x = 0.b0010101 \ldots 0100$.
Either way, $p_i(n)$ does not correctly compute the $n$th bit of $3x$. 
Note that we spread out the unbounded search, so that each bit is computed
by a bounded (primitive recursive) procedure.

In this way we can diagonalize against $3x, x^2$ and $x+1/3$ being
primitive recursive while making $x$ primitive recursive.
To make $x$ to be in $\mathcal{E}^2$ one uses the fact
that $\mathcal{E}^2$ is the same as computable in linear space
\cite{ar:Ritchie63}.

In the argument above,
when we are trying to figure out bit $t$ of $x$, we compute 
$p_i(n)$ (for some
$i$ and $n$ determined earlier in the construction of $x$) for $t$
steps and
if it does not halt we output the default bit (alternating between 
$0$ and $1$), so it can be made in linear time.
Actually, in the case of $x^a$2, one has to work a little harder to
determine the default bit, but this can definitely be done in linear
space (and polynomial time).
\end{proof}

We shall see in Section \ref{recursion}, 
Theorem \ref{th:most56},
that Theorem \ref{thm:friedman} covers all the spectral reals,
therefore the spectral reals do not form a field. More precisely
we have the following Corollary:

\begin{corollary}
The spectral reals are not closed under addition nor
under multiplication. 
Furthermore, they are closed under the operation $1-x$
iff the complement of a spectrum is a spectrum.
\end{corollary}

We now turn the question of algebraicity and transcendence of spectral reals.
Clearly, every first order spectral real
is a recursive real
in the sense of A. Turing \cite{ar:Turing36}.
Using Liouville's  Theorem\footnote{
Liouville's Theorem states, in simplified form, that
a real of the form $r =\sum_n 2^{-f(n)}$  where $f(n) \geq n!$
is transcendental.
}, we can see that many
transcendental reals are first order spectral.

\begin{openquestion}
\label{algebraic}
Are there any irrational algebraic reals which are spectral?
\end{openquestion}

One way of analyzing irrational numbers is by counting
the number of $1$s in their binary representation.
For a real $r \in (0,1)$ let $\gamma_r(n)$ be the
number of $1$s among its first $n$ digits.
If $r=r_{\phi}$ is spectral we have $\gamma_r(n)=\gamma_{phi}(n)$.

In the sequel we follow closely and quote from M. Waldschmidt \cite{ar:Waldschmidt08}.

\begin{theorem}[Bailey, Borwein, Crandall, and Pomerance, 2004, \cite{ar:BBCP04}] 
Let $r$ be a real algebraic number of degree $d \geq 2$. 
Then there is a positive number $C_{r,d}$, which depends
only on $r$, such that 
$\gamma_r(n) \geq C_{r,d}n^{\frac{1}{d}}$.
\end{theorem}

In other words, if a spectral number $r_{\phi}$ 
is algebraic of degree $d \geq 2$, then 
$\gamma_{\phi}(n) \geq C_{{\phi},d}n^{\frac{1}{d}}$, 
for some positive number $C_{{\phi},d}$.

To get more information about irrational numbers $r$ we have to look
at the binary string complexity of $r \in (0,1)$.
We consider $r$ as an infinite binary word.
\begin{definition}[Binary string complexity]
The {\em binary string complexity of $r$} is the function $p_r(m)$ which counts, for each $m$ 
the number of distinct binary words $w$ of length $m$ occuring in $r$. 
Hence we have $1 \leq  p_r(m) \leq 2^m$, and the function
$p_r(m)$ is non-decreasing.
\end{definition}

\begin{conjecture}[E. Borel 1950, \cite{ar:Borel50}]
The binary string complexity of 
an irrational algebraic number $r$ should be $p_r(m) = 2^m$.
\end{conjecture}

\begin{definition}
We call a real number $r \in (0,1)$ {\em automatic} if the $n$-th bit of its binary expansion
can be generated by a finite automaton from the binary representation of $n$.
\end{definition}
Clearly, the binary string complexity of an automatic real is $O(m)$.

\begin{openquestion}
\label{automatic}
Is every automatic real a spectral real?
\end{openquestion}

In 1968 A. Cobham, \cite{ar:Cobham68} conjectured that automatic numbers
are transcendental. This was proven in 2007 by B. Adamczewski and Y. Bugeaud, 
\cite{ar:AdamczewskiBugeaud07}.
They actually proved a stronger theorem.

\begin{theorem}[B. Adamczewski and Y. Bugeaud, 2007]
\label{th:AB}
The binary string complexity $p_r(m)$ of 
a real irrational algebraic number $r$ satisfies
$$
\liminf_{m \rightarrow \infty} \frac{p_r(m)}{m} = + \infty
$$
\end{theorem}
Borel's Conjecture would imply that
the binary string complexity $p_r(m)$ of 
a real irrational algebraic number $r$ satisfies
$$
\liminf_{m \rightarrow \infty} \frac{p_r(m)}{2^m} = 1
$$
\begin{openquestion}
\label{spectral-lim}
Does the binary string complexity $p_r(m)$ of 
a spectral real $r$ satisfy
$$
\liminf_{m \rightarrow \infty} \frac{p_r(m)}{2^m} < 1
$$
or even
$$
\liminf_{m \rightarrow \infty} \frac{p_r(m)}{2^m} = 0?
$$
\end{openquestion}

From Theorem \ref{th:AB}
one gets that the Fibonacci numbers, which will be shown to form a spectrum
in Corollary \ref{cor rud-in-spec} of Section \ref{subsection recursion},
give us a transcendental spectral number. More generally, we get the following:
\begin{proposition}
\label{pr:gap-transcendental}
Let $r_{\phi}$ be a spectral real such that
the gap function
$\delta_{\phi}(n)$ is monotonically increasing
and grows exponentially.
Then $p_{\phi}(n) = O(n)$.
Therefore,
$r_{\phi}$ is transcendental.
\end{proposition}

The analysis of computable reals in binary or $b$-adic presentation
is tricky because of the behaviour of the carry, cf. \cite{ar:ChenSuZheng2007a}.
Let $\mathcal{F}$ be a class of functions $f: \N \rightarrow \N$.
\begin{definition}
\item
A real number $\alpha$ is called {\em $\mathcal{F}$-Cauchy computable} if there are functions
$f,g,h \in \mathcal{F}$ such that for
$$
r_n = \frac{f(n)-g(n)}{h(n)+1}
$$
we have that for all $n \in \N$
$$
\mid r_n -\alpha \mid \leq \frac{1}{n+1}.
$$
\item
A real number $\alpha \in [0,1]$ in $b$-adic presentation is called {\em $\mathcal{F}$-computable} 
if there is $f: \N \rightarrow \{0, \ldots , b-1\}$ such that
$$
\alpha = \sum_{n \in \N} f(n) b^{-n}
$$
\end{definition}
Note that it is not clear at all how to define spectral Cauchy reals.
If $\mathcal{F}$ contains the function $2^{n}$ then the $\mathcal{F}$ computable
reals in $b$-adic presentation are also $\mathcal{F}$-Cauchy computable.
In particular, this is true for $\mathcal{F}= \mathcal{E}^i$ and $i \geq 3$.
\begin{openquestion}
\label{e2-reals}
Are the $b$-adic $\mathcal{E}^2$-computable reals
$\mathcal{E}^2$-Cauchy computable? 
\end{openquestion}

Recently, $\mathcal{E}^2$-Cauchy computable reals have received quite a bit of attention,
cf. \cite{ar:Skordev2002,ar:Skordev2008}.
The following summarizes what is known.

\begin{proposition}[D. Skordev]
\begin{renumerate}
\item
The $\mathcal{E}^2$-Cauchy computable reals form a real closed field.
\item
The transcendental numbers $e$ and $\pi$, and the Euler constant
$\gamma$ and the Liouville number $\sum_{n \in \N} 10^a{-n!}$ are
$\mathcal{E}^2$-Cauchy computable. 
\item
There are
$\mathcal{E}^3$-Cauchy computable reals which are not
$\mathcal{E}^2$-Cauchy computable. 
\end{renumerate}
\end{proposition}
Let $\mathcal{F}_{low}$ be the smallest class of functions in $\mathcal{E}^2$
which contains the constant functions, projections, successor, modified difference,
and which is closed under composition and bounded summation.
A real $\alpha$ is {\em low} if $\alpha$ is $\mathcal{F}_{low}$-Cauchy computable.
The low reals also form a real closed field.
In \cite{ar:TentZiegler2009} low reals are studied and some very deep theorems about
low transcendental numbers are
obtained, the discussion of which would take too much space.
\begin{openquestion}
\label{low}
Is the inclusion $\mathcal{F}_{low} \subseteq \mathcal{E}^2$ proper?
\end{openquestion}

\ 
\section{Approach I: Recursion Theory}~\label{section recursion}
\label{recursion}
\ 
\\
\ 

This approach has generated all in all four papers (namely \cite{ar:Asser55} 
by G. Asser in 1955, \cite{ar:Mostowski56} by A. Mostowski in 1956, \cite{ar:Ritchie63} by
R. Ritchie, and \cite{ar:Mo91} 
by S. Mo in 1991) and two Ph.D. dissertations, namely
\cite{phd:Ritchie60} by R. Ritchie in 1960 and
\cite{phd:Bennett62} by 
J. Bennett in 1962. 
These works share the common feature of being hardly available for many readers 
on various grounds: 
Asser's and Mostowski's papers are difficult to read because they are more 
than fifty years old and Asser's paper is in German. 
Bennett's thesis, cited in many papers, is almost equally old and in 
addition has remained unpublished. 
Finally, Mo's paper, though more recent, is in Chinese. 
This is the reason why we propose in Section \ref{section technical} 
a detailed review of these references, 
including several sketches of proofs in modern language.
In the present section, after some background material, 
we present a synthetic survey of the recursive approach of the spectrum problem.

\ 
\subsection{Grzegorczyk's Hierarchy}
For a detailed presentation of the material in this subsection, see eg. \cite{bk:Rose84}. 
A. Grzegorczyk's seminal paper \cite{ar:Grzegorczyk53} about 
classification of primitive recursive functions was published in 1953, one year 
after Scholz's question, and two years before Asser's paper. Hence, 
Grzegorczyk's Hierarchy was not the standard way to consider primitive 
recursive functions  in the mid-fifties. And actually, G. Asser and A. Mostowski deal 
with recursive aspects of spectra, but not explicitly with Grzegorczyk's 
classes, though it is the usual framework in which their results are presented. 
It is only in J. Bennett's thesis in 1962 and especially in S. Mo's paper in 
1991 that one finds an explicit study of spectra in terms of Grzegorczyk's 
classes.
\medskip

In the sequel a function is always intended to be a function from some 
${\mathbb N}^k$ to $\mathbb N$ (total, unless otherwise specified).
\begin{definition}[Elementary functions]
The class $\mathcal E$ of elementary functions is the smallest class of functions 
containing the zero, successor, projections, addition, multiplication and 
modified subtraction functions and which is closed under 
composition and bounded sum and product 
(i.e. $f(n,\vec{x})=\sum_{i=0}^ng(i,\vec{x})$ and 
$f(n,\vec{x})=\prod_{i=0}^ng(i,\vec{x})$, with previously defined $g$). 
We denote by ${\mathcal E}_{\star}$ the elementary relations, 
i.e. the class of relations whose characteristic functions are elementary.
\end{definition}

The class $\mathcal E$ was introduced by Kalm\'ar \cite{ar:Kalmar43} 
and Csillag \cite{ar:Csillag47} in the forties, and contains most usual 
number-theoretic functions. 
It also corresponds to Grzegorczyk's class ${\mathcal E}^3$, that we define below.

\begin{definition}[Primitive recursion]
Let $f,g,h$ be functions. 
We say that $f$ is defined from $g$ and $h$ by primitive recursion when it obeys a schema: 
$\left\{
\begin{array}{lll}
f(0,\vec{x})&=&g(\vec{x})\\f(n+1,\vec{x})&=&h(n,\vec{x},f(n,\vec{x}))
\end{array}
\right.$.
 
The class of primitive recursive functions, denoted by 
$\mathcal PR$, is the smallest class of functions containing the zero function, 
the successor function, the projection functions, and which is closed under composition 
and primitive recursion. 
\end{definition}

For instance, elementary functions are primitive recursive. 
The following binary function $Ack$, known as Ackermann's function, 
is provably not primitive recursive, 
whereas all unary specialised functions $Ack_x:y\mapsto Ack(x,y)$ are primitive recursive: 

$\left\{\begin{array}{l}Ack(0,y)=y+1\\
Ack(x+1,0)=Ack(x,1)\\
Ack(x+1,y+1)=Ack(x,Ack(x+1,y))\end{array}\right.$
\medskip

In order to introduce Grzegorczyk's hierarchy, 
we need a weaker version of primitive recursion, 
in which the newly defined functions have to be bounded by some previously defined function.

\begin{definition}[Bounded recursion]
Let $f,g,h,j$ be functions. 
We say that $f$ is defined from $g$, $h$ and $j$ by bounded recursion when it obeys a schema: 
$\left\{
\begin{array}{lll}
f(0,\vec{x})&=&g(\vec{x})\\f(n+1,\vec{x})&=&h(n,\vec{x},f(n,\vec{x}))\\f(n,\vec{x})&\leq&j(n,\vec{x})
\end{array}
\right.$
\end{definition}

Let $f_n$ ($n=0,1,2,\ldots$) be the following sequence of primitive recursive functions : 
\begin{itemize}
\item[-] $f_0(x,y)=y+1$,
\item[-] $f_1(x,y)=x+y$,
\item[-] $f_2(x,y)=(x+1)\cdot(y+1)$, 
\item[-] and  for $k\geq 0$ 
$\left\{
\begin{array}{ll}f_{k+3}(0,y)&=f_{k+2}(y+1,y+1)\\ f_{k+3}(x+1,y)&=f_{k+3}(x,f_{k+3}(x,y))
\end{array}
\right.$
\end{itemize}

Roughly speaking, the important feature is that the functions $f_n$ are 
more and more rapidly growing. 
Several other similar sequences of increasingly growing functions can be used 
to define Grzegorczyk's classes.

\begin{definition}[Grzegorczyk's hierarchy]
The Grzegorczyk's class ${\mathcal E}^n$ is the smallest class of functions 
containing the zero function, the projections functions and $f_n$ and 
which is closed under composition and bounded recursion. 
The associated classes of relations ${\mathcal E}_{\star}^n$ are defined 
as the class of relations on integers with a characteristic function in ${\mathcal E}^n$.
\end{definition}

Note that, for sake of simplicity, we use the same notation for a class of relations 
of various arities (eg. ${\mathcal E}_{\star}^3$) 
and the class of unary relations (i.e. sets) it contains. 
Which one is intended will always be clear from the context.
\medskip

The main features of Grzegorczyk's classes were studied by A. Grzegorczyk 
in \cite{ar:Grzegorczyk53} and by R. Ritchie in \cite{ar:Ritchie63}.

\begin{theorem}[A. Grzegorczyk (1953)]\ 
\begin{itemize}
\item[-] The functional hierarchy is strict for $n\geq 0$, i.e. we have ${\mathcal E}^n\subsetneq{\mathcal E}^{n+1}$.
\item[-] The relational hierarchy is strict for $n\geq 3$, i.e. we have ${\mathcal E}^n_{\star}\subsetneq{\mathcal E}^{n+1}_{\star}$.
\item [-] For the initial levels of the relational hierarchy, we have ${\mathcal E}_{\star}^0\subseteq{\mathcal E}_{\star}^1\subseteq{\mathcal E}_{\star}^2\subseteq{\mathcal E}_{\star}^3$.
\item[-] The Kalm\'ar-Csillag class of elementary functions $\mathcal E$ is equal to ${\mathcal E}^3$.
\item[-] Finally, the full hierarchy corresponds to primitive recursion, i.e. ${\mathcal PR}=\bigcup_{n=0}^{+\infty}{\mathcal E}^n$.
\end{itemize}
\end{theorem}

\begin{theorem}[R. Ritchie (1963)]
We have ${\mathcal E}^2_{\star}\neq{\mathcal E}^3_{\star}$ \cite{ar:Ritchie63}.
\end{theorem}

Note that the possible separation of the relational classes 
${\mathcal E}_{\star}^0,{\mathcal E}_{\star}^1,{\mathcal E}_{\star}^2$ 
is still an open question.

\begin{openquestion}
\label{e0e1e2}
Are the inclusions in
$$
{\mathcal E}_{\star}^0 \subseteq {\mathcal E}_{\star}^1 \subseteq {\mathcal E}_{\star}^2
$$ 
proper?
\end{openquestion}
An important point is that the functional hierarchy deals with the rate at 
which functions may grow: 
intuitively, functions in the low level of the hierarchy grow very slowly, 
while functions higher up in the hierarchy grow very rapidly. 
However, this feature does not hold for the relational hierarchy, 
because characteristic functions do not grow at all (they are $0-1$ valued). 
For instance, the ternary relations 
$z=x+y$, $z=x\times y$, $z=x^y$ as well as $z=Ack(x,y)$ 
all belong to ${\mathcal E}_{\star}^0$, 
whereas the corresponding functions provably do not lie in ${\mathcal E}^0$.

\ 
\subsection{Rudimentary relations and strictly rudimentary relations}\label{rudimentary}

In addition to primitive recursive classes of relations, 
we also introduce two new classes of relations with an arithmetical flavour, 
namely the rudimentary and strictly rudimentary relations. 
These classes were originally introduced by R. Smullyan \cite{bk:Smullyan61}, 
and a major reference about rudimentary relations and subclasses 
is J. Bennett's thesis \cite{phd:Bennett62}.

\begin{definition}[Rudimentary relations]
Denote by $\RUD$ the smallest class of relations over integers 
containing the graphs of addition and multiplication (seen as ternary relations) 
and closed 
under 
Boolean operations ($\neg$, $\wedge$, $\vee$)
and bounded quantifications 
($\forall x<y  \ldots$ and 
$\exists x<y  \ldots$).
\end{definition}

In spite of its very restricted definition, the class $\RUD$ is surprisingly 
robust (eg. it has several equivalent definitions in the fields of 
computational complexity, recursion theory, formal languages etc.) and large. 
For instance, the following formula defines the set of prime numbers: 
$$x>1 \ \wedge \ \forall y<x \ \forall z<x \ \neg (\ x=y.z \ )$$ 
More (sometimes VERY) 
sophisticated formulas prove that the ternary relation $z=x^y$ is rudimentary 
(Bennett \cite{phd:Bennett62}), as well as the graph $z=Ack(x,y)$ of 
Ackermann's function (Calude \cite{ar:Calude87}), which is not primitive 
recursive (as a function), or the four-ary relation $x^y\equiv z\ [mod\ t]$ 
(Hesse, Allender, Barrington \cite{ar:HesseAB02}). Actually, we are not aware 
of a natural number theoretic relation which is provably not rudimentary.

The following is easy to see.
\begin{proposition}
$\RUD \subseteq 
{\mathcal E}^0_{\star} \subseteq
{\mathcal E}^1_{\star} \subseteq
{\mathcal E}^2_{\star}$. 
\end{proposition}

However, the equality remains an open question (and would 
imply $\RUD = {\mathcal E}^2_{\star}$ as well, 
since the closure of ${\mathcal E}^0_{\star}$ 
by polynomial substitution is ${\mathcal E}^2_{\star}$ whereas
$\RUD$ is closed under polynomial substitution). 

\begin{openquestion}
\label{rud=e0}
Are the inclusions in
$\RUD \subseteq 
{\mathcal E}^0_{\star} \subseteq
{\mathcal E}^1_{\star} \subseteq
{\mathcal E}^2_{\star} 
$ proper?
\end{openquestion}
\medskip

It remains to introduce the strictly rudimentary relations.
Let us consider the dyadic representation of integers, 
i.e. $n\in{\mathbb N}$ is represented by a word in $\{1,2\}^*$. 
Compared to binary notation, dyadic notation avoids the problem 
of leading $0$s and yields a bijection between integers and words. 
When integers are seen as words, it is natural to consider 
subword quantifications instead of ordinary bounded quantification. 
We say that $w=w_1\ldots w_k$ is a subword of $v=v_1\ldots v_p$ and 
we denote $w\harp  v$ when there exists $1\leq i\leq p$ 
such that $w_1=v_i,\ldots,w_k=v_{i+k-1}$. 
Of course, if $x\harp  y$, then $x\leq y$.

\begin{definition}[Strictly rudimentary relations]
Denote by $\SRUD$ the smallest class of relations over integers 
containing the graphs of dyadic concatenation (seen as a ternary relation) 
and closed 
under 
Boolean operations ($\neg$, $\wedge$, $\vee$)
and subword quantifications 
($\forall x\harp  y  \ldots$ and 
$\exists x\harp  y  \ldots$).
\end{definition}

There are only few examples of strictly rudimentary relations, 
e.g. $x$ begins (or ends or is a part of) $y$ (as dyadic words), 
$x=y$, the dyadic representation of $x$ is a single symbol, 
the dyadic representation of $x$ contains only one type of symbol. 
On the other hand, several relations are provably not strictly rudimentary 
such as $x\leq y$, $x=y+1$, $x$ and $y$ have the same dyadic length 
and the dyadic representation of $x$ is of the form $1^n2^n$ for some $n$ 
(V. Nepomnjascii 1978, see \cite{ar:Nepomnjascii78}).

Note that rudimentary relations were originally (and equivalently) 
defined by Smullyan \cite{bk:Smullyan61} using dyadic concatenation 
as a basis relation instead of addition and multiplication. 
Clearly, we have $\SRUD\subsetneq \RUD$.

\ 
\subsection{Recursive and arithmetical characterizations of spectra}
\label{subsection recursion}
In the fifties and sixties, following the tastes of their time, 
logicians aim at characterizing $\SPEC$ via recursion and arithmetics. 
Typically, they wished to obtain the characteristic functions 
of spectra as the $0$-$1$-valued functions in a class defined 
by closure of a certain set of simple functions under certain operators 
(such as composition or various recursion schemas). 
From this point of view, their results are not totally satisfactory 
because they are either partial, or somehow cumbersome or unnatural.
However, these studies show that the class of spectra is very broad, 
and that most classical arithmetical sets are spectra. 
\medskip

The class $\SPEC$ is set within Grzegorczyk's hierarchy (by G. Asser in \cite{ar:Asser55} and A. Mostowski in \cite{ar:Mostowski56}), from which we can deduce that all rudimentary sets are spectra. 

\begin{theorem}[G. Asser (1955)]\label{theorem asser}
$\SPEC\subsetneq {\mathcal E}_{\star}^3$
\end{theorem}

Asser's theorem is based on a rather complicated and artificial 
arithmetical characterization of spectra 
(see Subsection \ref{subsection asser}). 
In particular, Asser's construction is of no help in 
proving that a particular set is (or not) spectrum.

Though he actually uses a slightly different class 
(see Subsection \ref{subsection mostowski}), 
the following result is usually attributed to Mostowski:

\begin{theorem}[A. Mostowski (1956)]
\label{th:most56}
${\mathcal E}_{\star}^2\subseteq\SPEC$
\end{theorem}

Note that equality in Mostowski's theorem remains an open question.
\begin{openquestion}
\label{e2=spec}
Is the inclusion in
${\mathcal E}_{\star}^2\subseteq\SPEC$
proper?
\end{openquestion}
The following corollary is not stated by A. Mostowski, 
but can be found in J. Bennett's thesis. 
It is worth noting because one of the most fruitful ways in proving 
that various arithmetical sets are spectra is to prove that they are actually rudimentary.

\begin{corollary}\label{cor rud-in-spec}
$\RUD\subseteq\SPEC$
\end{corollary}

For instance, any set defined by a linear or polynomial recurrence condition, 
such as the Fibonacci numbers (i.e. those numbers 
appearing in the sequence defined by $u_0=u_1=1$ and $u_{n+2}=u_n+u_{n+1}$), is rudimentary
(see \cite{ar:EsbelinM98}). 
From Corollary \ref{cor rud-in-spec}, we deduce that such sets are spectra, as 
announced in Section \ref{se:counting}. 
Similarly, using the fact that the set of prime numbers 
is rudimentary and the exponentiation has a rudimentary graph, 
one proves that the sets of Fermat primes (of the form $2^{2^n}+1$), 
Mersenne primes (of the form $2^p-1$ with $p$ a prime), or
twin primes ($p$ prime such that $p+2$ is also a prime) are rudimentary (hence also spectra).

Note that the question of whether the inclusion in Corollary \ref{cor rud-in-spec} 
is 
proper
is still open. 

\begin{openquestion}
\label{rud=spec}
Do we have $\RUD=\SPEC$?
\end{openquestion}

This problem is further investigated in Subsubsection \ref{subsubsection rud voc}.
\medskip

An arithmetic characterization of $\SPEC$ in terms of 
strictly rudimentary relations is also given, 
among many other results 
(see Subsection \ref{subsection bennett}), by J. Bennett in his thesis.

\begin{theorem}[J. Bennett (1962)]
\label{th spec srud}
\ \\
A set $S\subseteq{\mathbb N}$ is in $\SPEC$ iff it can be defined 
by a formula of the form $\exists y\!\leq\! 2^{x^j} R(x,y)$ 
for some $j\geq 1$, where $R$ is in $\SRUD$.
i.e.,
$$S =
\{ x \in \N \ | \ \exists y\!\leq\! 2^{x^j} R(x,y)\}$$
for some   $R\in \SRUD\hbox{ and }j\geq 1$.
\end{theorem}

J. Bennett also characterizes spectra of higher order logics 
and shows that the union of spectra of various orders equals 
the class of elementary relations ${\mathcal E}_{\star}^3$. 

The characterization of spectra stated in Theorem \ref{th spec srud} 
is rather simple and elegant. 
However, once again, it is not really useful in proving that 
a given set is a spectrum, now because $\SRUD$ is very restrictive. 
A somehow similar characterization of $\SPEC$ using $\RUD$ instead of 
$\SRUD$ would have been more powerful - but, one gets this way second-order spectra.
\medskip

Finally, let us note a late paper on the recursive aspect of spectra, 
namely \cite{ar:Mo91}, due to the Chinese logician Mo Shaokui in 1991 
(see Subsection \ref{se:Mo}). 
The solution to Scholz's problem proposed there is of the same type as 
Bennett's characterization. 
However, the only bibliographic references in Mo's paper are 
Scholz's question \cite{ar:Scholz52} and Grzegorczyk's paper \cite{ar:Grzegorczyk53}, 
so that it can be considered completely independent 
from all other contributions about spectra. 
Section \ref{section technical} summarises this paper's results.


\ 
\section{Approach II: Complexity Theory}
\label{complexity}

The spectrum problem, formulated in the early 1950s, predates complexity theory 
since the notions of time or space bounded Turing machines first emerged in the 
1960s (see \cite{ar:Hartmanis65,ar:Kuroda64}). However, the first results about 
complexity of spectra appeared very soon (see 
Subsection~\ref{earlycomplexity}), and computational complexity 
characterisations of spectra were found, in at least three independent early 
contexts (see Subsection~\ref{crucialcomplexity}). Later on, several 
refinements and developments of these seminal results have been published (see 
Subsections \ref{descriptivecomplexity} and ~\ref{latecomplexity}).

Turing machines and other standard models of computation operate on words, not 
on numbers. Let $L \subseteq \Sigma^*$  be a set of finite words over a fixed 
finite alphabet  $\Sigma$. Without loss of generality we assume $\Sigma = 
\{0,1\}$, and that input words have no leading zeros.

The archetypical  task,  given a language $L$, is  to study the complexity of 
deciding membership in $L$ of a word $x$ as a function of the length $|x|$, 
i.e., asymptotic growth rate of the time, space or other computational 
resources needed to decide whether $x \in L$.

\subsection{Complexity and spectra.} 
In this section, for a fixed sentence 
$\phi$, the set of natural numbers $\spec{\phi}$ is seen as the set of positive 
instances of a decision problem (given a number $n$, is there a model of $\phi$ 
with $n$ elements?).

When dealing with computational complexity, we convert spectra (sets of natural 
numbers) into languages (over alphabet $\{0,1\}$). The spectrum problem can 
thus be rephrased as: What is the computational complexity of the decision 
problems for spectra?

\subsubsection*{Complexity classes}
\label{complexityclasses}

Denote by $\NTime{f(n)}$ (resp.\ $\DTime{f(n)}$) the class of binary languages 
accepted in time $O(f(n))$ by some non-deter\-mi\-ni\-stic (resp. 
deterministic) Turing machine, where $n$ is the length of the input. Similarly, 
let us denote by $\DSpace{f(n)}$ the class of languages accepted in space 
$O(f(n))$ by some deterministic Turing machine. Some well-known complexity 
classes which concern us here are: 
$$ 
\L = \DSpace{\log n} \subseteq
\NL = \NSpace{\log n}\ \ 
$$
$$ 
\LINSPACE = \DSpace{n} \subseteq
\NLINSPACE = \NSpace{n}\ \ 
$$
$$
\classP=\displaystyle\bigcup_{c\geq 1}\DTime{n^{c}} \subseteq
\NP=\displaystyle\bigcup_{c\geq 1}\NTime{n^{c}}\ \ 
$$ 
$$\E=\displaystyle\bigcup_{c\geq 1}\DTime{2^{c\cdot n}}\subseteq  
\NE=\displaystyle\bigcup_{c\geq 1}\NTime{2^{c\cdot n}} $$ 
Finally, if 
$\mathrm{C}$ denotes a complexity class, we denote its complement class, i.e. 
the class of binary languages $L$ such that $\Sigma^*-L\in\mathrm{C}$, by 
$\mathrm{coC}$.

Of course, the perennial open questions are:
\begin{openquestion}
\label{p=np}
\begin{renumerate}
\item
Are any of the inclusions 
\\
$\L \subseteq \NL$, $\LINSPACE \subseteq \NLINSPACE$,
$\classP \subseteq \NP$ and $\E \subseteq \NE$ proper?
\item
Do any of the equalities  $\NP = \coNP$ and $\NE = \coNE$ hold?
\end{renumerate}
\end{openquestion}
Surprisingly, the following was shown independently
by N. Immermann and R. Szelepc\'zenyi in 1982, cf. \cite{bk:Immerman99}:
\begin{theorem}[Immermann, Szelepc\'zenyi 1982]\ \\
\label{l=nl}
$\NL = \coNL$ and $\NLINSPACE = \coNLINSPACE$.
\end{theorem}
In Section \ref{vocabularies} we shall also make use of the
polynomial time hierarchy $\PH$ and its linear analogue $\LTH$.

The class $\RUD$ lies between $\L$ and $\LINSPACE$,
and must be different from one of them.
\begin{proposition}
\ 
\begin{renumerate}
\item(\rm Nepomnjascii 1970, \cite{ar:Nepomniaschy70})
$\L \subseteq \RUD$
\item(\rm Wrathall 1978, \cite{ar:Wrathall78})
$\RUD = \LTH$
\item(\rm Myhill 1960, \cite{rep:Myhill60})
$\LTH  \subseteq \LINSPACE$
\end{renumerate}
\end{proposition}

\begin{openquestion}
\label{lth}
\ \\
Are the inclusions
$\L \subseteq \RUD = \LTH  \subseteq \LINSPACE$ proper?
\end{openquestion}

\subsubsection*{Number representations by binary or unary words}

It is natural to use binary notation for natural numbers (an alternative 
without leading zeros is Smull\-yan's {\em dyadic} notation 
\cite{bk:Smullyan61}). The shortest binary  length and dyadic length of the 
natural number $n$ are very close to $\lceil\log_2n\rceil$, whereas its unary 
length is of course $n$, and we have $n=2^{\log_2n}$. Consequently, the same 
(mathematical) computation that is performed by some Turing machine in time eg. 
$O(2^{c\cdot |n|})$ when $|n|$ is the binary length of the natural number 
input, is also performed (by a slightly different Turing machine) in time 
$O(n^c)$ when $n$ is the (unary length of the) natural number input.

Unary notation (also called {\em tally} notation, i.e. the number $n$ is 
represented by the word $1\ldots 1$ composed of $n$ ones)  also has its fans, 
for reasons explained in the description of Fagin's work. Most results in this 
section may be stated in either notation, but for sake of simplicity, and 
unless explicitly stated otherwise, we use binary notation. The length of a 
binary or unary word $x$ is written $|x|$.

Recall that $\SPEC$ denotes the set of spectra of first-order sentences, i.e., 
$$\SPEC = \{\spec{\phi}\ |\ \phi \mbox{\ is a first-order sentence} \}$$

\subsection{Spectra, formal languages, and complexity theory}

Formal language  theory was much studied in the early 1960s, cf. \cite{bk:Harrison78}, 
in particular the Chomsky hierarchy. While the regular and context-free language classes were 
well-understood, several questions remained open for  larger classes.
We need here the following:

\newpage
\begin{theorem}
\label{ritchie-kuroda}\ 
\begin{renumerate}
\item
{\rm (Ritchie 1963, \cite{ar:Ritchie63})}
${\mathcal E}^2_{\star}  = \LINSPACE $ 
\item
{\rm (Kuroda 1964, \cite{ar:Kuroda64})}
A language $L$ is {\em context sensitive} iff \\$L \in\NLINSPACE$. 
\end{renumerate}
\end{theorem}

For our discussion one should remember that at that time
it was then (as now) unknown whether $\LINSPACE = \NLINSPACE$
and also unknown whether $\NLINSPACE$ was closed under complementation. 
The latter was only resolved positively
more than 20 years later, see Theorem \ref{l=nl}.

These open questions showed a tantalising similarity to Scholz' and Asser's 
questions. If we identify characteristic functions with sets, then Bennett's 
$1962$ thesis 
combined with Asser, Mostowski and Ritchie's results, 
yield 
$$ 
\LINSPACE \subseteq \SPEC \subseteq {\mathcal E}^3_{\star} 
\mbox{\ and\ }
\LINSPACE \subseteq \NLINSPACE \subseteq {\mathcal E}^3_{\star} .
$$ 
This led to 
a conjecture $\SPEC  \stackrel{?}{=}  \NLINSPACE $, that spectra might be 
coextensive to the context sensitive languages. 
The analogy fails, though, since more 
than $n$ ``bits of storage'' are needed to store an $n$-element model $\mathcal 
M$ of a  sentence $\phi$.

\subsection{An early paper}\label{earlycomplexity}

One of the first papers explicitly relating spectra to bounded resource machine 
models of computation is \cite{ar:RoeddingS72} (in German), due to R\"odding 
and Schwichtenberg from M\"unster in 1972. This switch from recursion theory to 
complexity theory had been prepared ten years before by Bennett and Ritchie,
and R\"odding and Schwichtenberg made a  step further. The model of computation 
they use is not Turing machines, but register machines. As Bennett does, 
R\"odding and Schwichtenberg not only consider spectra of first-order 
sentences, but also higher order spectra, namely spectra of sentences using 
$i$-th order variables. Let us denote by  $\HOSPEC{i}$ the class of spectra of 
sentences using $i$-th order variables. Let us define the following sequence of 
functions : let $exp_0(n) =n$, and $exp_{i+1}(n) = 2^{exp_i(n)}$. Along with 
other results in the field of recursion theory, R\"odding and Schwichtenberg 
prove the following theorem.

\begin{theorem}[R\"odding and Schwichtenberg $1972$ \cite{ar:RoeddingS72}] For 
all $i\in\N$, we have $\DSpace{exp_{i}(n)}\subseteq \HOSPEC{i+1}$.
\end{theorem}

In particular, taking $i=0$, first-order spectra are thereby shown to contain 
$\DSpace{n}$.
Let us finally remark that R\"odding and Schwichtenberg did not consider 
non-deterministic complexity classes.

\subsection{First-order spectra and non-deterministic exponential time}
\label{crucialcomplexity}

Scholz's original question (see \cite{ar:Scholz52}) was finally answered after 
twenty years, when Jones and Selman related first-order spectra to 
non-deter\-mi\-ni\-stic time bounded Turing machines. Their result was first 
published in a conference version in 1972 (see \cite{proc:JonesS72}), and the 
journal version appeared in 1974 (see \cite{ar:JonesS74}).  The following 
theorem holds.

\begin{theorem}[Jones and Selman $1972$ \cite{proc:JonesS72}] $\SPEC = \NE$.
\end{theorem}

\noindent This leads to a complexity theory counterpart of Asser's question:

\begin{corollary} $\SPEC=\coSPEC$ if and only if $\NE = \coNE$.
\end{corollary}

They note that this does not answer Asser's question, but it shows the link 
with a wide range of closure under complement questions, in complexity theory. 
Presently, we know that many of them are very difficult questions.

\begin{proof}[Proof ideas] To see that $\SPEC \subseteq \NE$, consider 
$\spec{\phi}\in \SPEC$. Since the sentence $\phi$ is fixed, satisfaction ${\mathcal 
M} \models \phi$ can be decided in time that is at most polynomial in the size 
of model ${\mathcal M} $, where the polynomial's degree depends on the quantifier 
nesting depth in $\phi$. A simple guess-and-verify algorithm is: given number 
$x$, non-deterministically guess an $x$-element model $\mathcal M$, then decide 
whether ${\mathcal M} \models \phi$ is true. Time and space $2^{c\cdot n}$ 
suffice to store an $x$-element model and check ${\mathcal M} \models \phi$, 
where constant $c$  is independent of $x$ and $n$ is the length of $x$'s binary 
notation. Thus the algorithm works in non-deterministic exponential time (as a 
function of input length $n$).

To show $\SPEC \supseteq \NE$, let $Z$ be a nondeterministic time-bounded 
Turing machine that runs in time ${2^{c\cdot n}}$ on inputs of length $n$. Here 
the input is a word $x$ of length $n$. We think of $x$ as  a binary-coded 
natural number. Computation $C$  can be coded as a word $C = {\it config}_0 
{\it config}_1 \ldots {\it config}_{2^{c\cdot n}}$ where ${\it config}_0$ 
contains the Turing machine's input, and each $ {\it config}_t$ encodes the 
tape contents and control point at its $t$-th computational step.

Now we have to
find a first-order sentence $\phi$ such that:
\begin{enumerate}
\item For every input-accepting $Z$ computation,  ${\mathcal M} \models \phi$  
for some model ${\mathcal M}$  of cardinality $x$
\item  If $Z$ has no computation that accepts its input, then $\phi$  has no 
model   of cardinality $x$
\end{enumerate}
Each ${\it config}_t$ length is at most $2^{c\cdot n}$, so $C$ has length bound 
$2^{c'\cdot n} = x^{c'}$ for $c'$ independent of $n$. A   model ${\mathcal M}$  
of cardinality $x$ contains, for each $k$-ary predicate symbol $P$ of sentence 
$\phi$, a relation $\overline{P}\subseteq \{0,1,\ldots,x-1\}^k$. Thus a model 
can in principle ``have enough bits'' 
$x^k = 2^{k\log x} = 2^{O(|x|)}$
to encode all the symbols of computation 
$C$.

The remaining task is to actually construct $\phi$ so it has a model ${\mathcal 
M}$  of cardinality $x$ if and only if  $Z$ has a well-formed computation $C$ 
that accepts input $x$. In effect, the task is to use predicate logic to check 
that ¤C¤ is well-formed and accepts $x$. The technical details are omitted from 
this survey paper; some approaches may be seen in
\cite{ar:JonesS74,phd:Fagin73,phd:Christen74}
\end{proof}

\subsection{Relationship to the question $\ \classP \stackrel{?}{=} \NP $ }
\label{relationshippequalsnp}

Let $\UN = \{ L | L \subseteq 1^* \}$ be the set of {\em tally languages} (each 
is a set of unary words over the one-letter alphabet $\{1\}$) and let $\NP_1 = 
\NP \cap \UN $. Since there is a natural identification between $\NP_1$ and 
$\NE$, we can deduce that if $\classP=\NP$, then $\NP=\coNP$ and $\NE = \coNE$, 
i.e. the complement of a spectrum is a spectrum. Of course, it also holds that 
if there is a spectrum whose complement is not a spectrum, i.e. if $\NE \neq 
\coNE$, then $\NP\neq \coNP$ and $\classP\neq\NP$. The converse implication 
remains open. \smallskip

\subsection{Independent solutions to Scholz's problem }

The characterization of spectra via non-deterministic complexity classes was 
independently found also by Christen on the one hand and Fagin on the other 
hand during their PhD studies.

Claude Christen's thesis\footnote{Claude Christen, born 1943, joined the faculty of CS 
at the University of Montreal in 1976 and died there, a full professor, 
prematurely, April 10, 1994.} \cite{phd:Christen74} (1974, ETH 
Z\"urich, E. Specker) remains unpublished, and only a small part was published 
in German \cite{proc:Christen76}. Christen discovered all his results 
independently, and only in the late stage of his work his attention was drawn 
to Bennett's work \cite{phd:Bennett62} and the paper of Jones and Selman 
\cite{proc:JonesS72}. It turned out that most of his independently found 
results were already in print or published by Fagin after completion of 
Christen's thesis.

Ronald Fagin's thesis (1973, UC Berkeley, R. Vaught) is treasure of results 
introducing {\em projective classes of finite structures}, which he called
{\em generalized spectra} (see Subsection 
\ref{descriptivecomplexity}) that had wide impact on what is now called 
descriptive complexity and finite model theory. Most of our knowledge about 
spectra till about 1985  and, to some extent far beyond that, is contained in 
the published papers (see \cite{proc:Fagin74,ar:Fagin75a,ar:Fagin75b}) 
emanating from Fagin's thesis \cite{phd:Fagin73}. In this survey, these papers 
are pervasive. Right now, let us begin with reviewing what is said in 
\cite{proc:Fagin74} about the consequences of the complexity characterization 
of spectra {\em per se}.

Recall that $\E=\bigcup_{c\geq 1}\DTime{2^{c\cdot n}}\subseteq\NE$ and let us examine 
the closure under complementation problem. Since $\E=\coE$, it is clear that 
if a first-order spectrum is in $\E$, then its complement is also a 
first-order spectrum. Of course, the question $\E \stackrel{?}{=} \NE$ is 
still open. Fagin notes that $\E$ contains the spectra of categorical 
sentences, i.e. sentences that have at most one model for every cardinality. 
Thus, one obtains a model theoretic question closely related to Asser's question. 

\begin{openquestion}
\label{asser-cat}
Is every 
spectrum the spectrum of a categorical sentence ?
\end{openquestion}

Besides reviewing many natural sets of numbers that are spectra, Fagin also 
proves by a complexity argument the existence of a spectrum $S$ such that 
$\{n\, \mid 2^n\in S\}$ is {\em not} a spectrum (see also~\cite{Hunter04} for a 
recent proof by diagonalization).

\subsection{Generalized first-order spectra and $\NP$: Fagin's result}
\label{descriptivecomplexity}

Let us spend some time on what is called {\em generalized first-order spectra} 
by Fagin in his $1974$ paper \cite{proc:Fagin74}, and is nowadays more usually 
refered to as (classes of finite structures definable in) {\em existential 
second-order logic}. Our main goal is to clarify the differences and the 
connections with ordinary first-order spectra.

In this subsection, we are no longer interested  in the {\em size} of the 
finite models of some given sentence, but in the models themselves. Hence, let 
$\sigma$ and $\tau$ be two disjoint vocabularies, and let $\phi$ be a 
first-order $\sigma\cup\tau$-sentence. The {\em generalized spectrum} of $\phi$ 
is the class of finite $\tau$-structures that can be expanded to models of 
$\phi$. In other words, it is the class of finite models of the existential 
second-order sentence $\exists\sigma\phi$ with vocabulary $\tau$. The vocabulary 
$\tau$ is usually refered to as the {\em built-in} vocabulary, whereas $\sigma$ 
is often called the {\em extra} vocabulary. Note that generalized spectra are 
finite counterparts to Tarski's {\em projective classes} (see 
\cite{ar:Tarski54}). Fagin's theorem states the equivalence between generalized 
spectra and classes of finite structures accepted in $NP$.

\begin{theorem} [Fagin 1974 \cite{proc:Fagin74}] Let $\tau$ be a non-empty
vocabulary. A class of finite $\tau$-structures $K$ is a generalized 
spectrum iff $K \in \NP$.
\end{theorem}

If the built-in vocabulary $\tau$ is empty, then the $\tau$-structures are 
merely sets. From a computational point of view, it is natural to see such 
empty structures as unary representations of natural numbers. From a logical 
point of view, one obtains ordinary spectra. Hence, Fagin rephrases Jones and 
Selman's complexity characterization of first-order spectra as follows:

\begin{proposition}
A set $S,$ if regarded as a set of unary words,  is a first-order spectrum
if and only if $S\in\NP_1$.
\end{proposition}

Concerning the complement problem for generalized spectra, in view of Fagin's 
theorem, it is not surprising that the general case remains open. However, the 
following is known. If $\sigma$ consists of unary predicates only, it is called 
unary. It has been proved in several occasions that unary generalized spectra 
are not closed under complement (see Fagin 1975 \cite{ar:Fagin75b}, Hajek 1975 
\cite{proc:Hajek75}, Ajtai and Fagin 1990 \cite{ar:AjtaiF90}). For instance, it 
is shown in \cite{ar:Fagin75b} by a game argument that the set of connected 
simple graphs is not a unary generalized spectrum. In contrast, it is easy to 
design a monadic existential second-order sentence defining the class of 
non-connected simple graphs. \smallskip

Since our survey deals with spectra and not with descriptive complexity as a 
whole, we will not say any more on this subject. However, let us note that {\em 
descriptive complexity} \cite{bk:Immerman99} emerged as a specific field of research out of Fagin's 
paper about generalized spectra.

\subsection{Further results and refinements}\label{latecomplexity}
 
During the late 1970s and the 1990s, several results were published that generalize 
Jones and Selman's result to higher order spectra on the one hand, and that 
refine this result, in order to obtain correspondences between certain 
complexity classes and the spectra of certain types of sentences.

In 1977, Lov\'asz and G\'acs \cite{ar:LovaszGacs77}, 
it is shown essentially that there are generalized first order spectra such that
their complement cannot be expressed with a smaller number of variables.
To do this they introduced first order reductions, which became a very important tool
in finite model theory and descriptive complexity. 
In fact they were the first to show the existence of decision problems
which are $\bNP$-complete with respect to first order reductions.

First order reductions were used in (un)decidability results early on,
\cite{bk:TarskiMR53}, and more explicitely in \cite{pr:Rabin65}.
For a systematic survey, see \cite{ar:HensonCompton,ar:MakowskyTARSKI}.
In the context of generalized spectra they were rediscovered independently
also by Immerman in \cite{ar:IMM1}, Vardi in \cite{pr:vardi82} and  Dahlhaus in \cite{phd:dahlhaus}.
First order reductions are of very low complexity, essentially they are uniform 
$\mathbf{AC}^0$ transductions. The first use of low complexity reduction techniques
seems to be Jones \cite{ar:Jones75} who termed them log-rudimentary.
Allender and Gore \cite{ar:AllenderGore91} showed them equivalent to
uniform $\mathbf{AC}^0$ reductions.

\begin{openquestion}
\label{compl-spec}
Is there a universal (complete) spectrum $S_0$ and a suitable notion of
reduction such that every spectrum $S$ is reducible to $S_0$?
\end{openquestion}
Note that this question has two flavors, one where we look at spectra
in terms of sets of natural numbers, and one where we look at sets  of (cardinalities of)
finite models and their defining sentences.
First order reductions may be appropriate
in the latter case.

In 1982, Lynch \cite{ar:Lynch82a} relates the computation time needed to 
decide property on set of integers to the maximal arity of symbols required in 
the sentence to define this property. Below, we refine the definition of classes of 
spectra in order to take into account some specific syntactic restrictions on 
sentences.

\smallskip

\begin{definition} Let $d\in\mathbb{N}^*$, the following classes are defined as 
follows.
\begin{renumerate}
  \item $\RSpectra{d}{}$ (resp.  $\FSpectra{d}{}$) is the class of spectra of 
  first-order sentences over arbitrary  predicate  (resp. predicate and 
  function) symbols of arity at most $d$.
  \item $\RSpectra{}{}(k\forall)$ (resp. $\RSpectra{d}{}(k\forall)$) is the 
  class of spectra of prenex first-order sentences involving at most $k$ 
  variables that are all universally quantified (resp. and involving predicate 
  symbols of arity at most $d$). The classes $\FSpectra{}{}(k\forall)$ and 
  $\FSpectra{d}{}(k\forall)$ are analogously defined.
  \item  Finally, let $\RSpectra{d}{}(+)$ be the class of spectra of 
  first-order sentences over the language containing one ternary relation 
  interpreted as the addition relation over finite segments of $\mathbb{N}$ and 
  predicate symbols of arity at most $d$
\end{renumerate}
\end{definition}

Some inclusions between these classes are easy to obtain: more resources (in 
terms of arity or number of variables) means more expressive power. Hence, for 
example,  for all $i,j\in\N$ such that $i<j$:

\[ \RSpectra{i}{}\subseteq \RSpectra{j}{}, \  \FSpectra{i}{}\subseteq 
\FSpectra{j}{}  \mbox{ and }  \RSpectra{i}{}\subseteq \FSpectra{i}{}. \]

The following results proves a first relationship between time complexity of 
computation  and ``syntactic'' complexity of definition.

\medskip 

\begin{proposition}[Lynch 1982~\cite{ar:Lynch82a}]
\label{prop Lynch 82}
\ \\
For all $d\geq 1$, $\NTime{2^{d\cdot n}}\subseteq \RSpectra{d}{}(+)$.
\end{proposition} 

The converse of this result remains open. It refines the complexity 
characterization of first-order spectra and has many further developments 
that we present in Section \ref{vocabularies}. From a technical point of view, 
note that Lynch works with so-called ``word-models''. Namely, a binary word 
$w$ with length $n$ is seen as a structure with universe $\{0,\ldots,n-1\}$, 
equipped with some arithmetics (eg. successor predicate or addition predicate) 
and with a unary predicate that indicates the positions of the digits $1$ of 
$w$. The methods developed in this paper are re-used later on by several 
authors.

\begin{openquestion}
\label{lynch}
Is the inclusion
$d\geq 1$, $\NTime{2^{d\cdot n}}\subseteq \RSpectra{d}{}(+)$ proper?
\end{openquestion}

Finally, Lynch explains that, from an attentive reading of Fagin's 
proof, one can only deduce that if some language $L$ is in $\NTime{2^{d\cdot 
n}}$, then $L$ is in $\RSpectra{2d}{}$ i.e. is a the spectrum of a first-order 
sentence involving predicates of arity at most $2d$. Even though Lynch's result 
is not an exact characterization, but only an inclusion, it has been very 
influential to other researchers.

In a series of papers published between $1983$ and $1990$ (see 
\cite{ar:Grandjean83a,proc:Grandjean83b,ar:Grandjean84,ar:Grandjean85,proc:Grandjean87,ar:Grandjean90b}), 
Grandjean proposes two fruitful ideas. The first one is to use RAM machines as 
a natural model of computation for general logical structures instead of Turing 
machines, which are best fitted for languages (or word structures). The second 
idea is to remark that the time complexity seems closely related to the 
syntactical form of the sentences (and more specifically in this case with the 
number of universally quantified variables). Let $\NRam{f(n)}$ be the class of 
binary languages accepted in time $O(f(n))$ by a non-deterministic RAM (with 
successor), where $n$ is the length of the input.

\medskip

\begin{theorem}[Grandjean $1983$
\cite{proc:Grandjean83b,ar:Grandjean85,proc:Grandjean87,ar:Grandjean90b}]
\label{THE Grandjean}
\ \\
For all $d\geq 1$, we have  $\NRam{2^{d\cdot 
n}}=\FSpectra{}{}(d\forall)=\FSpectra{d}{}(d\forall)$.
\end{theorem}      

The case $d=1$ i.e. the case of sentences with one universally 
quantified variable, is more involved: it requires to encode arithmetic 
predicates such as linear order or addition that appear intrinsically in the 
characterization of computation by sentences with one variable. It is developed 
in the papers \cite{proc:Grandjean87,ar:Grandjean90b}). In passing, this 
implies that the presence of the addition relation  is not mandatory in 
Proposition~\ref{prop Lynch 82} provided (unary) functions are allowed in the 
language.

An interesting corollary of the latter characterization is that  when the number 
of (universally quantified) variables is fixed, restricting the language to 
contain function or relation symbols of arity bounded by $d$ only does not 
weaken the expressive power of sentences and define the same class of spectra. 
In other words, the following holds.

\medskip

\begin{corollary}[Grandjean $1983$~\cite{proc:Grandjean83b,ar:Grandjean85,proc:Grandjean87,ar:Grandjean90b}]~\label{COR Grandjean}
For all $d\geq 1$, it holds that 
  $\FSpectra{}{}(d\forall)=\FSpectra{d}{}(d\forall)$.
\end{corollary}      

The original proof of this result relies on  complexity arguments based on the
characterization of $\FSpectra{}{}(d\forall)$. We give here a purely logical
proof~\footnote{We thank \'Etienne Grandjean for kindly giving us this proof.}.

\begin{proof}[Proof of Corollary~\ref{COR Grandjean}]
For simplicity of notation, we give the proof in the case $d=1$.
Let $\varphi\equiv \forall t \Psi$ where $\Psi$ is quantifier-free and whose
vocabulary is composed of function symbols of various arities. Let
$\textsf{Term}(\Psi)$
be the set of terms and subterms of $\Psi$. The first idea is to associate with
each element $\tau$ of $\textsf{Term}(\Psi)$ a new unary function $f_{\tau}$.
The definition of $f_{\tau}$ is as follows:

\begin{renumerate}
\item if $\tau=t$ or $\tau$ is a constant symbol, then $f_{\tau}(t)=\tau$,
\item if $\tau=f(\tau_1(t),\ldots,\tau_k(t))$ for some function symbol 
$f$ of
arity $k$, then $f_{\tau}(t)=f(f_{\tau_1}(t),\ldots,f_{\tau_k}(t))$.
\end{renumerate}

One obtains a new sentence $\Psi'$ instead of $\Psi$ by replacing each term 
$\tau\in\textsf{Term}(\Psi)$ by $f_{\tau}(t)$ in conjunction with the definition  
of each function symbol $f_{\tau}$. Let us explain the 
transformation on some example.

Let $\Psi$ be the following very simple sentence with $f$ of arity $2$ and $g$ of
arity $1$:

\[
E(\underbrace{f(\overbrace{f(t,\underbrace{g(t)}_{\tau_1})}^{\tau_2},t)}_{\tau_3},t)
\]

Then, $\Psi'$ corresponds to:

\[
\begin{array}{l}
E(f_{\tau_3}(t),f_t(t)) \wedge f_{\tau_3}(t)=f(f_{\tau_2}(t),f_{t}(t)) \\
\wedge f_{\tau_2}(t)=f(f_t(t),f_{\tau_1}(t)) \wedge
f_{\tau_1}(t)=g(f_t(t))\wedge f_t(t)=t 
\end{array}
\]

It is easily seen that the only non unary symbols (here $f$) appear (if they do at all)
 only as an outermost symbol in atomic formula. Let now $f(\overline{\sigma_1}(t))$,
 \ldots, $f(\overline{\sigma_h}(t))$ be the list of terms in $\Psi'$ involving $f$. The
 idea is now to replace in $\Psi'$ each $f(\overline{\sigma_i}(t))$ by some new  term
 $F_i(t)$ where $F_i$ is of arity one (let's call  this new sentence $\Psi''$) and to
 write down
 the relations between each pair $F_i(t)$ and $F_{j}(t)$ for $i,j\leq h$. 
 This provides a new sentence $\varphi'=\forall t\forall t'\ (\Psi''\wedge
 \Delta)$ where
 
 \[
 \Delta\equiv  \bigwedge_{i,j\leq h} 
 (\overline{\sigma}_i(t)=\overline{\sigma}_j(t')\rightarrow F_i(t)=F_j(t')) \]
 
  The above method shows, when the number of variables is $d=1$ how to replace
  $h$-ary functions by unary functions. However, in order to control the
  definition of the $F_is$ we introduce one additional quantified variable. 
To get rid of this additional variable one can proceed as follows. First, the vocabulary is enriched with a binary predicate $<$ interpreted as a linear order on the domain, and $h$ unary functions $\overline{N}_j$ for $j\leq h$. Let $\Delta'$ be the following sentence. 

\[
\begin{array}{l}
\Delta'\equiv \\
\forall (i,t) \exists (j,x) \ \overline{N}(j,x)=(\overline{\sigma}_i(t),i,t)\wedge\\
\forall (j,x)\neq (h,max) \exists (j',x') \ (j',x')=(j,x)+1 \wedge \overline{N}(j,x)<\overline{N}(j',x')\wedge\\
\forall (j,x)\neq (h,max) \exists (j',x') \exists (i,t) \exists (i',t')\\
\qquad (j',x')=(j,x)+1 \wedge \overline{N}(j,x)=(\overline{\sigma}_i(t),i,t) \wedge \overline{N}(j',x')=(\overline{\sigma}_{i'}(t'),i',t')\\
\qquad \wedge (\overline{\sigma}_i(t) = \overline{\sigma}_{i'}(t')\rightarrow F_i(t)=F_{i'}(t')). 
\end{array}
\]

\noindent where $\forall (i,t)$ stands 
for $\bigwedge_{1\leq i \leq h}\forall t$ and $\exists (j,x)$ for 
$\bigvee_{1\leq i \leq h}\exists x$; $N(j,x)$ stands for $N_j(x)$. Similarly, $(j,x)+1$ represents the 
successor of pair $(j,x)$ in the lexicographic ordering of pairs $(j,x)$, 
$j\in\{1,\ldots,h\}$ and $x\in D$. The above sentence expresses the fact that the function 
$\overline{N}$ (in fact the union of functions $\overline{N}_j$, $j\leq h$) is 
an increasing bijection from the set $\{1,\ldots,h\}\times D$ to the set 
$\{(\overline{\sigma}_i(t),i,t) | i\leq h, t\in D\}$. This sentence plays the 
same role as the sentence $\Delta$ but this time tuples 
$(\overline{\sigma}_i(t),i,t)$ with the same first component 
$\overline{\sigma}_i(t)$ are contiguous in the numbering $\overline{N}$.
Using a result of Grandjean \cite{ar:Grandjean90b}, one can replace the linear ordering $<$ by additional unary functions.
\end{proof}

To be complete, one should also mention the earlier (and weaker) result obtained
by Pudl\'ak~\cite{ar:Pudlak75} by purely logical argument at that time.

\begin{proposition}[Pudl\'ak 75~\cite{ar:Pudlak75}] 
\ \\
$\FSpectra{}{}(d\forall)\subseteq \FSpectra{d}{}(2d\forall)$
for all $d\geq 1$.
\end{proposition}

\raus{
\begin{proof}[Proof ideas] We sketch the proof on an example. Suppose $d=2$ and
$\varphi$ is the following sentence involving arity at most $3$ predicate and
function symbols.

\[ \varphi\equiv \forall x_1\forall x_2 \ R(x_2,x_1,x_2) \wedge 
R(g(x_1,f(x_1),x_2),g(x_1,x_1,x_2),x_2) \]

For every occurrences of $R$ and $g$ (function $f$ may be left unchanged), new
symbols $R_1,R_2,\ldots$ and $g_1$ are introduced and $\varphi$ is transformed
into a new sentence $\phi$ where $R(x_2,x_1,x_2)$ is replaced by
$R_1(x_1,x_2)$ and $R(g(x_1,f(x_1),x_2),g(x_1,x_1,x_2),x_2)$ by $R_2(x_1,x_2)$. In
conjunction with $\phi$ the following axioms using $2d$ variables are stated which
relates the truth value of the new added symbols :

\[
\begin{array}{l}
\forall x_1\forall x_2\forall y_1\forall y_2 \\
\ \ x_1=y_1\wedge f(x_1)=y_1\wedge x_2=y_2 \rightarrow
(g_1(x_1,x_2)=g_2(y_1,y_2))\\
\ \ x_2=g_1(y_1,y_2) \wedge x_1=g_2(y_1,y_2) \wedge x_2=y_2 \rightarrow
(R_1(x_1,x_2)\leftrightarrow R_2(y_1,y_2)) 
\end{array}
\]
 
 \end{proof}
 }
In the next section, we will examine more closely the expressive power of 
spectra on restricted vocabulary. The results of this section show that a tight 
connection exists between nondeterministic complexity classes and classes of 
spectra defined by limiting the number of universally quantified variables in 
sentences. A natural question is whether such a connection exists when the 
language itself is limited. In particular

\begin{openquestion}
\label{fspec}
Is there a characterization as a complexity class of the classes 
$\FSpectra{d}{}$ for all $d\geq 1$ ?
\end{openquestion}

This question has also some connections with problems addressed in 
Section~\ref{SEC unary and arity hierarchies}.

\raus{

\subsection{First order reductions
} 
\comment{TBC : jam}

First order reductions (translations schemes) were introduced first in {\bf 
Tarski, Mostowski, Robinson 1953} \\
and further developed by {\bf Rabin 1965}.

{\bf Lovasz, Gacs 1977} used them to show that certain
generalized spectra are $\bNP$-complete via first order reductions.

{\bf Dahlhaus  1983 \cite{proc:Dahlhaus83}} showed that SAT, CLIQUE and 
\\
DHC (directed hamiltonian cycle) are such
generalized spectra.

{\bf Immerman 1982} 
and
{\bf Vardi 1982} 
used the same technique to characterize the
generalized spectra in  $\bP$.

}

\section{Approach III: Restricted vocabularies}
\label{vocabularies}
\subsection{Spectra for monadic predicates}

Maybe the simplest way to restrict vocabularies is by limiting the arity of the 
symbols. In that direction, the smallest restricted class of spectra that can 
be studied is that of sentences involving only relation symbols of arity one 
(so-called monadic in the literature). In this case, the following can be 
proved:

\begin{proposition}[L\"owenheim 1915~\cite{ar:Loewenheim15}, Fagin 1975 ~\cite{ar:Fagin75b}]
\label{RES MSO finite}
Let $\tau$ be a vocabulary consisting of unary relation symbols only
and $\phi \in \MSOL{\tau}$.
Then the spectrum of $\phi$ is finite or co-finite.
\end{proposition}

\begin{proof} 
Use quantifier elimination or Ehrenfeucht-Fra\"{\i}ss\'e games.
\end{proof}

Remark that the even numbers are a spectrum of the following  sentence
with one unary function:

\[
\forall x \ f(x)\neq x \wedge f^2(x)=x.
\]

Hence, the most trivial extension of monadic relational vocabulary already provide a spectrum which is neither finite nor co-finite. Then, a natural question is whether the converse of Proposition~\ref{RES MSO finite} is true or not. The following observation can be made by remarking that one can express the cardinality of a finite domain set by an existential first-order formula.

\begin{obs}
Every finite or co-finite set $X \subseteq \N$
is a first-order spectrum for a sentence with equality only (i.e., no relation or function symbols).
\end{obs}

This contrast  with the fact that  every $\Sol$-spectrum is also an
$\Sol$-spectrum  over equality only.
This allows to conclude.

\begin{proposition}
If $\tau$ consists of a finite (possibly empty) set
of unary relation symbols, the $\MSOL{\tau}$-spectra
are exactly all finite and cofinite subsets of $\N$.
\end{proposition}
\ 

\subsection{Spectra for one unary function}
\label{SEC one unary function}

As remarked above, one unary function is enough to define nontrivial spectra. 
It turns out, however, that a complete characterization of spectra for 
one unary function (with additional unary relations) is possible.

\begin{definition}
A set $X \subseteq \N$ is {\em ultimately periodic} if there
are $a, p \in \N$ such that for each $n \geq a$ we have that $n \in X$ iff
$n+p \in X$.
\end{definition}

The set of even numbers is ultimately periodic with $a=p=2$. Again, one may observe 

\begin{obs}
Every ultimately periodic set $X \subseteq \N$
is a first order spectrum for a sentence with one
unary function and equality only (this is already true if the function is restricted to be a permutation).
\end{obs}

Surprisingly, ultimately periodic sets are precisely the spectra of sentences with one unary function
 \cite{proc:DurandFL97,proc:GurevichS03}.

\begin{theorem}[Durand, Fagin, Loescher 1997, 
Gurevich, Shelah 2003]\label{th:gur-she}
Let $\phi$ be a sentence of $\MSOL{\tau}$ where $\tau$
consists of

\begin{itemize}
\item[-] finitely many unary relation symbols, 

\item[-] one unary function and equality only.
\end{itemize}

Then $\spec{\phi}$ is ultimately periodic,
\end{theorem} 

\begin{proof} The proof of~\cite{proc:DurandFL97} uses Ehrenfeucht-Fra\"{\i}ss\'e game argument and is  restricted to the $\EMSOL{\tau}$ case. The generalization of~\cite{proc:GurevichS03} is an application of the Feferman-Vaught-Shelah decomposition method.
\end{proof}

There exists alternative ways to characterize ultimately periodic sets. Among others, they can  also be seen as sets of integers definable in Presburger Arithmetic. Also, since ultimately periodic sets are  closed under complementation, one have:

\begin{corollary}
Spectra involving a single unary function symbol are closed under complement.
\end{corollary}
\  

\subsection{Beyond one unary function and transfer theorems}
\ \\ 
There exist several ways to extend a vocabulary with one unary 
function: one may choose to add one (or several) new unary function(s) or, in 
the opposite direction, one may consider vocabularies with only one unrestricted binary relation 
symbol. It turns out that, up to what will be called "transfer theorems" in 
the sequel, both kinds of extension lead to very expressive formulas.

Before going further, notations about classes of spectra need to be refined to 
take into account the number of symbols of distinct arities. Again, we will 
distinguish in the sequel {\em whether function symbols are allowed or not}.
We write $\textsc{f-spec}$ with various indices when function symbols are allowed,
and $\textsc{spec}$ when function symbols are not allowed.

\begin{definition}
A set $S$ of integers is in  $\FSpectra{i}{\alpha,\beta}$ if there exists a 
first-order sentence $\phi$ such that $S=\spec{\phi}$ and the vocabulary of 
$\varphi$ contains  only

\begin{itemize}
\item[-] $\alpha$ function symbols of arity $i$ and 

\item[-] at most $\beta$ function  symbols   of arity less than $i$, or 
relation symbols of arity less or equal to $i$.
\end{itemize}

Said another way: a set of integers is in $\RSpectra{i}{\alpha,\beta}$ if it can be 
defined by a first order formula over the language of $\alpha$ 
relation symbols of arity $i$ and  $\beta$ relation symbols of arity less than $i$.
\end{definition}

When the number of symbols of arity less than $i$ is not restricted, the 
respective class of spectra are denoted by $\RSpectra{i}{\beta,\omega}$ and 
$\FSpectra{i}{\beta,\omega}$. For example, the class of first order spectra 
over one unary function and an arbitrary number of monadic relation 
and constant symbols (studied in Section~\ref{SEC one unary function}) is denoted 
$\FSpectra{1}{1,\omega}$. Similarly $\RSpectra{i}{\alpha}$ and 
$\FSpectra{i}{\alpha}$ are abbreviations for $\RSpectra{i}{\alpha,0}$ and 
$\FSpectra{i}{\alpha,0}$ Finally, in this setting $\RSpectra{i}{}$ abbreviates
 for $\RSpectra{i}{\omega,\omega}$ (the same holds for  $\FSpectra{i}{}$).

Let us examine what are the relations between these different classes of 
spectra. Recall that for all $i,j\in\N$ such that $i<j$:

\[ \RSpectra{i}{}\subseteq \RSpectra{j}{}, \  \FSpectra{i}{}\subseteq 
\FSpectra{j}{}  \mbox{ and }  \RSpectra{i}{}\subseteq \FSpectra{i}{}. \]

The following inclusions are also easy to see. For all $i,\beta\in\N$,

\[
\RSpectra{i}{\alpha,\omega}\subseteq \RSpectra{i}{\alpha+1} \mbox{ and }  \FSpectra{i}{\alpha,\omega}\subseteq \FSpectra{i}{\alpha+1}.
\]

The relationships between spectra of $i$-ary functions and spectra of $i+1$-ary 
relations can be made more precise. In~\cite{ar:DurandR96}, it is shown that a 
spectrum of a first-order formula involving any number of unary function symbols is 
also the spectrum of a formula using only one binary relation. This can be 
generalized to any arity.

\begin{proposition}[\cite{ar:DurandR96,phd:Durand96}]
For every integer $i\geq 1$, $\FSpectra{i}{}\subseteq\RSpectra{i+1}{1}$
\end{proposition}

The converse is not known not be true. All in once, the following chain of 
inclusions holds.

\[
\RSpectra{i}{} \subseteq \FSpectra{i}{1} \subseteq \FSpectra{i}{} \subseteq \RSpectra{i+1}{1}\subseteq \RSpectra{i+1}{}
\]

It seems difficult to prove the converse of any of the inclusions given above. 
Then, a natural way to study the expressive power of languages relatively to 
spectrum definition is to reduce them through functional (here polynomial) 
transformation. 

\begin{definition}
If $f : \N \rightarrow \N$ and $S$ is a set of integers, then 
$f(S)=\{f(n) : n\in S\}$ and $S^i=\{n^i : n\in S\}$.
\end{definition}

Let ${\cC_1}$ and ${\cC_2}$ be  two classes of spectra and 
$f:\N\rightarrow \N$ be a function. A natural question is then the 
following:\medskip

\textit{Let $S$ be a spectrum in ${\cC_1}$ and $f$ be a $\mathbb{N}\rightarrow \mathbb{N}$ function,  is $f(S)$ a spectrum in ${\cC_2}$?}

\medskip

 In~\cite{ar:Fagin75a}, Fagin showed an interesting equivalence between spectra defined by different relational languages. Such results have been called transfer theorems since then.
 The proof  can be seen as an extension of the well-known interpretation method of Rabin~\cite{proc:Rabin64} with an additional constraint that describes how the domain size of the structure needs to change.

\begin{theorem}[Fagin 1975~\cite{ar:Fagin75a}]~\label{Fagin transfer theorem}
For every $i\geq 1$, $S\in \RSpectra{2i}{} \iff S^i\in \RSpectra{2}{}$. 
\end{theorem}

Since relational spectra of arity one are finite or cofinite the ``transfer" 
theorem above cannot be extended to $\RSpectra{1}{}$. Not too surprisingly, if 
function symbols are allowed, a similar and more uniform equivalence can be 
proved.

\begin{proposition}~\label{transfer theorem function}
For every $i\geq 1$, $S\in \FSpectra{i}{} \iff S^i\in \FSpectra{1}{}$.
\end{proposition}

The flexibility of unary functions as well as their expressive power are well 
emphasized by the following result which show that the image of any 
spectrum under any polynomial transformation is a spectrum involving unary 
functions only, provided the polynomial is "big enough". 

\begin{theorem}[Durand, Ranaivoson 1996~\cite{ar:DurandR96}]
\label{polynomial transformation}
Let $P(X)\in  \Q [X]$ of degree $m\geq i$ and with a strictly positive 
dominating coefficient. Then,

\[
S\in  \FSpectra{i}{} \Rightarrow \lceil P(S) \rceil =\{\lceil P(n) \rceil,\ n\in S\} \in \FSpectra{1}{}.
\]
\end{theorem}

\subsubsection{Spectra for one binary relation symbol}

An easy  consequence of Theorem~\ref{Fagin transfer theorem} is that for 
every spectrum $S$, there exists $i\in\N$, such that $S^i\in  \RSpectra{2}{1}$. 
This result underlines the great expressive power of sentences involving 
exactly one binary relation symbol. Up to  polynomial transformation, any 
spectrum is a spectrum of such a sentence.

Let $BIN \subseteq  \RSpectra{2}{1}$ be the set of spectra of a symmetric, 
irreflexive relation (simple graphs). The whole complexity of the spectrum problem is 
already contained in the apparently weaker question of whether $BIN$ is 
closed under complement.

\begin{theorem}[Fagin 1974, \cite{proc:Fagin74}]
$BIN$ is closed under complement iff
the complement of every first order spectrum
is also a spectrum.
\end{theorem}

In fact, one can prove the following stronger result: For all $X \in BIN$ the complement
$\N - X \in \RSpectra{2}{1}$ iff the complement of every first order 
spectrum is also a spectrum.

\begin{openquestion}[Fagin 1974, \cite{proc:Fagin74}]
\label{fagin-one-binary}
Is every first order spectrum in $\RSpectra{2}{1}$?
\end{openquestion}

\subsubsection{Spectra for two unary functions and more}
 
Here again, Proposition~\ref{transfer theorem function} implies that 
for any spectrum $S$, there exists an integer $i$ such that $S^i$ 
is a spectrum involving unary functions only. 
This should be compared with the very weak  expressive power 
of sentences involving unary predicates only. 
We also know that one unary function leads to the very specific class 
of ultimately periodic sets.  
The question now is: how many unary function symbols are necessary 
to obtain an expressive (in the spectrum framework) fragment of first-order logic.

Recall that $\FSpectra{1}{i}$ denotes the set of first order spectra
using at most $i$ unary function symbols. Obviously, for all positive integer $i$,
$\FSpectra{1}{i}\subseteq \FSpectra{1}{i+1}$

The inclusion between the two first levels is strict, as shown in~\cite{ar:Loescher97}. 

\begin{theorem}[Loescher 1997 \cite{ar:Loescher97}]
The set $\{ n^2 : n \in \N \}$ belongs to the class $\FSpectra{1}{2}\backslash
\FSpectra{1}{1}$, hence the inclusion
$\FSpectra{1}{1} \subset 
\FSpectra{1}{2}$ is proper.
\end{theorem}

In fact, more can be proved on the expressive power of sentences with two unary functions. 

\begin{theorem}[Durand, Fagin and Loescher, 1998, \cite{proc:DurandFL97}]~\label{DFL98:transfer}
Given $k$ and a spectrum  $S$ in
$\FSpectra{1}{k}$.
Then
$kS=\{ kn : n \in S \} \in
\FSpectra{1}{2}$.
\end{theorem}

Combined with Proposition \ref{transfer theorem function} or with 
Theorem \ref{polynomial transformation}, 
this implies that there is a first order spectrum over two
unary function symbols which, written in unary, is ${\bNP}$-complete.
This also implies that there is a transfer result that maps every spectrum 
to a spectrum over two unary functions.
\\
\ \\

\subsubsection{Rudimentary sets and spectra of restricted vocabularies}
\label{subsubsection rud voc}
 
The relation between rudimentary sets  and spectra have been investigated.
It has been first observed that the set of 
primes is in $\FSpectra{1}{} $ (Grandjean 1988 \cite{ar:Grandjean88}). More 
generaly, it holds that:

\[\RUD \subseteq \FSpectra{1}{} \mbox{ (Olive 1996 \cite{proc:Olive97}).}\]

Due to the closure of $\RUD$ by polynomial transformation and to the existence 
of the above described transfer results for spectra, it is clear that, not
only, one can improve:

\[\RUD \subseteq \FSpectra{1}{2}\]
 
\noindent but also
 
\[\RUD = \SPEC \mbox{ if and only if } \RUD = \FSpectra{1}{2}\]
 
However, in view of the following surprising result, there are evidences that
none of these equalities hold. 
 
\begin{theorem}[Woods, 1981, \cite{phd:Woods81}]
If $\RUD = \SPEC$  then $\NP\neq \coNP$ and $\NE=\coNE$
\end{theorem}
 
Since the proof is hardly available and given in a different framework in
Woods's thesis, we sketch it below.
 
\begin{proof} 
Let $\LTH$ the linear time analog of the polynomial hierarchy $\PH$.
Celia Wrathall gave in~\cite{ar:Wrathall78} a precise complexity
characterization of rudimentary set by proving that $\LTH = \RUD$. 
 
The following facts are easy to prove.
 
\begin{enumerate}
\item  $\LTH \subseteq \DSpace{n}\subseteq \NE$
\item If $\LTH = \DSpace{n}$ then $\LTH$ collapses to some level $k$.
\item If $\LTH$ collapses to some level $k$ then 
$\PH$ collapses to some level $k'$.
\item $\LTH = \DSpace{n}$ implies $\PH = \PSPACE $
\end{enumerate}  

If $\RUD = \SPEC$ then, since $\SPEC=\NE$, it holds $\LTH = \DSpace{n} = \NE$.
Hence, both $\LTH$ and $\PH$ collapse. 

For the other consequence, suppose $\NP=\coNP $. In this
case,  the polynomial hierarchy collapses to $\NP$ i.e. $\NP=\PH$. 
If again $\RUD=\SPEC$ one knows than $\LTH =\DSpace{n}=\NE$ and $\PH=\PSPACE$.
Then, since $\NP\subseteq \NE$ one obtains $\PSPACE=\PH\subseteq
\DSpace{n}$ which contradict the well-known results $\DSpace{n}\subsetneq\PSPACE$. 
\end{proof}

Hence, although the equality $\RUD = \FSpectra{1}{2}$ might seem realistic at first
sight, its consequences makes it probably hard to prove.
\\ 
\ 
 
\subsection{The unary and the arity hierarchies}
\label{SEC unary and arity hierarchies}

The results of the preceding section show that for any spectrum $S$, there is a 
polynomial $P$ such that $P(S)$ is the spectrum of a first order sentence with 
two unary functions. This underlines the expressive power of this latter class 
of spectra \textit{up to polynomial transformation}. However, as we know  
\textit{equality} between particular classes of spectra defined, for example, 
by restriction on the number or the arity of the predicates in the languages is 
often an open problem. Taking a very particular case, it is even not known 
whether three unary functions "say" more than two as far as spectra are 
concerned. This leads to the following open problems about spectra of unary 
functions.

\begin{openquestion}[The unary hierarchy] 
\label{QUE unary hierarchy}
Is the following hierarchy proper:
\[
\FSpectra{1}{1}\subset 
\FSpectra{1}{2} \subseteq 
\FSpectra{1}{3} \subseteq  \ldots \subseteq
\FSpectra{1}{k} \subseteq  \ldots?
\]
\end{openquestion}

\begin{openquestion} 
\label{rud=fspectra}
Is $\RUD = \FSpectra{1}{2}$?
\end{openquestion}

\begin{openquestion} 
\label{rspectra=fspectra}
Is $\RSpectra{2}{1} = \FSpectra{1}{}$ or even $\RSpectra{2}{1} = \FSpectra{1}{2}$?
\end{openquestion}

Both positive or negative answers to these questions would have nontrivial 
consequences. For example, proving $\RUD \neq \FSpectra{1}{2}$ would separate 
the classes $\RUD$ and $\SPEC$. In the opposite, we already know that $\RUD = 
\FSpectra{1}{2}$ implies $\NP\neq \coNP$ and $\NE=\coNE$.

Proving that $\FSpectra{1}{k}\subsetneq \FSpectra{1}{k+1}$ for some integer $k$ 
would also separate $\SPEC$ from $\RUD$. Similarly, a collapse of the unary 
hierarchy to some level $k$ would have  strong consequences.  It is easily seen 
that,  testing if a number  $n$ (as input i.e. in unary) is in the spectrum of 
a first order sentence over $k$ unary functions can be decided by a
deterministic polynomial time RAM algorithm that uses $k\cdot n$ additional non 
deterministic steps. Since from Theorem~\ref{THE Grandjean},  
$\NRam{2^{n}}\subseteq  \FSpectra{1}{}$ then, the inclusion 
$\FSpectra{1}{}\subseteq \FSpectra{1}{k}$ would imply immediately the following 
"trade-off" result on nondeterministic RAM computations 
(see~\cite{proc:DurandFL97}): \textit{ for any arbitrary constant $c$, any 
nondeterministic RAM, which given a number $n$ as input, runs in time $c\cdot n$ 
can be simulated by a RAM which runs in $k\cdot n$ nondeterministic steps and 
in a polynomial number of deterministic steps}. For $c$ greater than $k$ (which 
is fixed) such a result is rather unexpected and would strongly modify our 
understanding of the relationships between determinism and nondeterminism.

\bigskip

Another natural question concerns the relative expressive power of 
first order sentences  defined by restriction on the arity of the 
symbols involved in the language. 
It is open whether any spectrum is the spectrum of a sentence 
over one binary relation only. 
The question may be refined as follows (See Fagin~\cite{ar:Fagin75a,ar:Fagin93}). 

\begin{openquestion}[The arity hierarchy] 
\label{arity}
Is the following hierarchy proper:
\[
\RSpectra{1}{}\subseteq 
\RSpectra{2}{} \subseteq 
\RSpectra{3}{} \subseteq  \ldots \subseteq
\RSpectra{k}{} \subseteq  \ldots?
\]
The same question could be asked for spectra over $i$-ary functions. 
\end{openquestion}

Although the above problem is still open, Fagin proved the following partial result.

\begin{theorem}[Fagin 1975~\cite{ar:Fagin75a}]
If $\RSpectra{k}{}=\RSpectra{k+1}{}$ for some integer $k$ then, the 
arity hierarchy collapses 
and $\RSpectra{k}{}=\RSpectra{m}{}$ for every $m \geq k$.
\end{theorem}

\subsubsection{Collapse of hierarchies and closure under functions}
 
Let $\cC$ be a class of spectra and $f$ be a $\mathbb{N}\rightarrow 
\mathbb{N}$ function. Class $\cC$ is \textit{closed} under $f$ if for any 
spectrum $S$ in $\cC$, $f(S)$ is in $\cC$. In the spirit of
Theorems~\ref{Fagin transfer theorem} and~\ref{polynomial transformation} and 
Proposition~\ref{transfer theorem function} one can  relate the 
collapse of hierarchies to the possible 
closure of class of spectra under some function. 
In~\cite{Hunter03}, such 
problems are studied and several related results are given.
 
\begin{theorem}\ 
\begin{enumerate}
\item (Hunter, \cite{Hunter03}) The arity hierarchy collapses to
$\RSpectra{2}{1}$ if and only if
$\RSpectra{2}{1}$ is closed under function $f:n\mapsto \lceil \sqrt{n}\rceil$.
\item The unary hierarchy collapses to $\FSpectra{1}{2}$ if and only if
$\FSpectra{1}{2}$ is
closed under function $f:n\mapsto \lceil \frac{n}{2}\rceil$.
\end{enumerate}
\end{theorem}
 
\begin{proof}[Proof of (ii)]
In~ \cite{Hunter03}, a result similar to (ii) is proved about the
number of binary predicates instead of the number of unary functions. 
It is not hard to see that his result extends to the case of (ii).
\end{proof}
\ 

\subsection{Higher order spectra}

In~\cite{ar:Lynch82a}, Lynch also proposes a generalization of the 
characterization given in Proposition~\ref{prop Lynch 82} to higher order 
spectra. It is shown that the polynomial time hierarchy corresponds to 
second-order logic. In other words, let  $\bPH_1 = \bPH \cap \UN $ then, a set 
$X \subseteq \N$ is a second order spectrum if and only if $\{ 1^n: n \in X\} 
\in \bPH_1$.

Consider $\Sol_k$ the class of second-order formulas where all the second order 
variables are of arity at most $k$. Let  $ 2^{X^k}= \{ 
2^{n^k}  \in \N : n \in X \} $ for $X \subseteq \N$. The following precise characterization can be 
obtained.

\begin{theorem}[More and Olive 1997]~\cite{ar:MoreO97}
A set $X$ is a spectrum of a sentence in $\Sol_k$
iff
$2^{X^k}$ is rudimentary.
\end{theorem} 

\subsection{Spectra of finite variable logic $\Fol^k$}

We denote by $\Fol^k$ first order logic with only $k$
distinct variables (bound or free),
and by $Spec\Fol^k$ the set of spectra of sentences of $\Fol^k$.
$\Fol^k$ has been studied extensively in finite
model theory, \cite{Otto95b,GraedelKolaitisVardi97}, but somehow the spectrum
problem was not adressed for $\Fol^k$ in the literature\footnote{
While the third author was lecturing in Chennai in January 2009 on
the spectrum problem,
Dr. S. P. Suresh asked about it.
The results below are the fruit of discussions
with Amaldev Manuel and Martin Grohe.
}.

With the same proof as for Proposition \ref{RES MSO finite}
one gets easily the following.

\begin{proposition}
The spectrum of a sentence in $\Fol^1$ is finite or cofinite.
\end{proposition}

Let $M$ be a set of natural numbers. 
Recall that a 
{\em gap} of $M$ is a pair of integers $g_1, g_2$ such that
$g_1, g_2 \in M$ but for each $n$ with $g_1 < n < g_2$ we have that
$n \not\in M$.
We define $Tow(k,n)$ inductively: $Tow(0,n)=n$ and $Tow(k+1,n) = 2^{Tow(k,n)}$.

By coding structures which model iterated powersets one gets immediately
the following.
\begin{proposition}
For every $k$ there is a $\phi \in FO^4$
such that the spectrum $sp(\phi)$ contains gaps of size
$Tow(k,n)$.
\end{proposition}
Four variables are used here to express extensionality of the membership relation,
and closure under unions with singletons.

Working a bit harder one can prove the following\footnote{M. Grohe, personal communication}.
\begin{theorem}[Grohe]
\begin{renumerate}
\item
For every Turing machine $T$ there is a sentence $\phi_T \in \Fol^3$ such that
$$sp_{\phi_T} = \{ t^2 : T \ \ \mbox{has an accepting run of length} \ \  t \}$$
\item
For every recursive function $f$ there is a sentence $\phi_f$ in $\Fol^3$
such that the gaps in the spectrum of $\phi_f$ grow faster than $f$.
\end{renumerate}
\end{theorem}
The key idea is to encode Turing machines on grids, hence the $t^2$ in the statement
of the Theorem.

In contrast to the above, it follows from the proof that $\Fol^2$ has the finite
model property and hence is decidable,
\cite{GraedelKolaitisVardi97}, that the gaps are bounded.
\begin{theorem}
For every $\phi \in FOL^2$ the spectrum $sp(\phi)$
has gaps of size at most $2^{poly(n)}$.
\end{theorem}
\begin{corollary}
The following inclusions are proper:
$$
Spec\Fol^1 \subset
Spec\Fol^2 \subset
Spec\Fol^3
$$
\end{corollary}

\begin{openquestion}[Finite Variable Hierarchy]
\label{oq:fivar}
Does the hierarchy $Sp\Fol^k$ collapse at level $3$?
\end{openquestion}


\section{The Ash conjectures}
\label{Ash}

The notion of $k$-equivalence of structures was  introduced by Fra\"{\i}ss\'e 
\cite{ar:Fraisse54} in 1954 and the game presentation is due to Ehrenfeucht 
\cite{ar:Ehrenfeucht61} in 1961.
 
\begin{definition} 
Let $\sigma$ be a vocabulary, let $\cA$ and $\cB$ be $\sigma$-structures and let 
$k$ be an integer. The two structures $\cA$ and $\cB$ are $k$-equivalent if and 
only if they satisfy the same $\sigma$-sentences with quantifier depth $\leq k$.
\end{definition} 
 
For a detailed presentation of $k$-equivalence, we refer the reader to 
\cite{bk:EbbinghausF95,bk:Hodges93}. For our purpose, the two most important 
features of the above notion are the following. For each finite vocabulary 
$\sigma$ and for each quantifier depth $k$, the number of $k$-equivalence 
classes of $\sigma$-structures is finite and we denote it by $M_{\sigma,k}$. 
For each finite vocabulary $\sigma$, for each quantifier depth $k$ and for each 
$k$-equivalence class of $\sigma$-structures $\cC$, there exists a 
$\sigma$-sentence $\Psi_{\cC}$ of quantifier depth $k$ such that, for all 
$\sigma$-structure $\cA$, we have ${\cA}\in{\cC}$ if and only if ${\cA}\models 
\Psi_{\cC}$.
 
In 1994, Ash \cite{ar:Ash94} introduces a counting function relative to the 
$k$-equivalence classes:
 
\begin{definition} 
Let $\sigma$ be a finite relational vocabulary, and let $k$ be a positive 
integer. Ash's function $N_{\sigma,k}$ counts, for each positive integer $n$, 
the number of $k$-equivalence classes of $\sigma$-structures of size $n$.
\end{definition} 
  
This function is obviously bounded by the total number of classes, $M_{\sigma, 
k}$, and Ash's conjecture deals with its asymptotic behavior.
 
\begin{openquestion}[Ash's constant conjecture] 
\label{ash-1}
Is it true that
for any finite relational 
vocabulary $\sigma$ and any positive integer $k$, the Ash function 
$N_{\sigma,k}$ is eventually constant?
\end{openquestion}   
 
A weaker version of his conjecture is also proposed:
 
\begin{openquestion}[Ash's periodic conjecture] 
\label{ash-2}
Is it true that
for any finite relational 
vocabulary $\sigma$ and any positive integer $k$, the Ash function 
$N_{\sigma,k}$ is eventually periodic?
\end{openquestion}  
 
Ash shows by a very neat proof that both conjectures imply the spectrum 
conjecture (i.e. $\SPEC=\coSPEC$). Let us present the idea of the proof. Fix 
$\sigma$ and $k$ and assume for simplicity that $N_{\sigma,k}(n)=a$ for all 
$n$. Take a quantifier depth $k$ first-order $\sigma$-sentence $\phi$. We 
exhibit a quantifier depth $k$ first-order $\sigma'$-sentence $\theta_a$ such 
that $\spec{\theta_a}={\mathbb N}^+\setminus \spec{\phi}$, where $\sigma'$ consists 
in $a$ disjoint copies of $\sigma$. Let ${\cC}_1,\ldots,{\cC}_{M_{\sigma,k}}$ 
be the classes of $k$-equivalence of $\sigma$-structures. All the structures in 
a given class ${\cC}_i$ agree on $\phi$, i.e. either $\forall {\cA}\in{\cC}_i\ 
{\cA}\models\phi$ or $\forall {\cA}\in{\cC}_i\ {\cA}\models\neg\phi$, because 
$\phi$ has quantifier depth $k$. Form the set $X_a$ consisting of all sets of 
$a$ distinct ${\cC}_i$s such that $\forall {\cA}\in{\cC}_i\ 
{\cA}\models\neg\phi$. Take $\theta_a\equiv\bigvee_{Y\in 
X_a}\bigwedge_{{\cC}\in Y}\Psi'_{\cC}$, where $\Psi'_{\cC}$ characterizes $\cC$ 
and is written using a distinct copy of $\sigma$. The sentence $\theta_a$ has 
quantifier depth $k$ and $n\in \spec{\theta_a}$ if and only if there are (at 
least) $a$ distinct $k$-equivalence classes containing a structure $\cA$ with 
size $n$ and ${\cA}\models\neg\phi$. Since the total number of $k$-equivalence 
classes containing a structure $\cA$ with size $n$ is exactly $a$, there is no 
class left for a model of $\phi$, and the anounced result follows.
 
These ideas of Ash's have not been exploited afterwards, 
and his paper has remained isolated until recently.
In 2006, Chateau and More \cite{ar:ChateauM06} published 
a second paper related to Ash's counting functions.
For all $i\in\NN^+$, note
$N_{\sigma, k}^{-1}(i) = \{ n \in \NN^+ | N_{\sigma, k}(n) = i \}$, 
the inverse image of the positive integer $i$ under the function 
$N_{\sigma, k}$. 
Both Ash's conjectures can be rephrased in terms of the 
sets $N_{\sigma, k}^{-1}(i)$ and are subsumed under the following condition: 
 
\begin{openquestion}[Ultra-weak Ash conjecture] 
\label{ash-3}
Is it true that
for any finite relational vocabulary $\sigma$, for any positive integer  
$k$ and for all $i\in\NN^+$, the set $N_{\sigma, k}^{-1}(i)$ is a 
spectrum?  
\end{openquestion} 

This last conjecture is proved to be a
necessary and sufficient condition for the complement of a spectrum to be a 
spectrum.  
 
\begin{theorem}\label{ultra} 
Let $\sigma$ be a finite relational vocabulary, and let $k$ be a 
positive integer.  
For all $i\in\NN^+$, the set $N_{\sigma, k}^{-1}(i)$ is a spectrum if 
and only if for every $\sigma$-sentence $\varphi$ of quantifier depth 
$\leq k$, the set $\NN^+\setminus \spec{\varphi}$ is a spectrum. 
\end{theorem}

All in all, the spectrum conjecture is true if and only if the ultra-weak Ash 
conjecture is true. 
 
Note that, in some cases, one gets interesting additional information about complements of 
spectra. Eg., if the Ash function $N_{\sigma,k}$ is eventually constant, then 
for every $\sigma$-sentence $\varphi$ of quantifier depth 
$\leq k$, the set $\NN^+ \backslash \spec{\varphi}$ is the spectrum 
of a sentence of the same quantifier depth as $\varphi$,  
over a vocabulary with the same arities as $\sigma$. 

In order to make some progress, 
particular cases of Ash conjectures may be considered,  
i.e., by restricting the  the sets of pairs $(\sigma,k)$. 
In his original paper, Ash \cite{ar:Ash94} already did so. 
For instance, he shows that if $\sigma$ is a unary vocabulary, then for all $k\geq 1$,  
Ash's function $N_{\sigma,k}$ is eventually constant.  
In \cite{ar:ChateauM06}, it is also proved that for all finite relational vocabulary 
$\sigma$, Ash's function 
$N_{\sigma,2}$ is eventually constant. 
However, these results are of very limited interest because unary vocabularies 
or quantifier depth two (see \cite{ar:Mortimer75}) only allow to define 
finite and cofinite spectra. 

In other cases, solving restricted versions of Ash conjectures 
happens to be as difficult as solving the full conjectures. 
Let $\mathcal{G}$ be a vocabulary restricted to a single binary relation symbol 
(i.e. the language of graphs).   
As already said, Fagin \cite{ar:Fagin75a} shows that, up to a polynomial padding, every 
spectrum is a spectrum in $\RSpectra{2}{1}$. 
Since spectra are closed under inverse images of polynomial, 
one obtains:

\begin{proposition}\label{bin} 
If for all $k\geq 1$ and for all $i\in \NN^+$, the set $N_{\mathcal{G}, 
k}^{-1}(i)$ is a spectrum, then the (full) spectrum conjecture holds.  
\end{proposition} 
 
It is less known that, as an easy corollary of a result of 
Grandjean in \cite{ar:Grandjean90b}, up to a polynomial padding, every spectrum 
is the spectrum of a quantifier depth $3$ sentence, using an unbounded number of binary relations.
Once again, it follows:

\begin{proposition}\label{trois}  
If for all binary vocabulary $\sigma$ and for all $i\in \NN^+$, the 
set $N_{\sigma, 3}^{-1}(i)$ is a spectrum, then the spectrum 
conjecture holds.   
\end{proposition} 
 
Since there is no known padding result that uses a finite set of pairs 
$(\sigma,k)$, these results are presently the best possible.
 
In order to make further progress in solving particular cases of Ash 
conjectures, Chateau and More introduce a new type of restriction, concerning the 
semantics of the vocabularies. 
 
\begin{definition} 
Let $\sigma$ be a finite 
relational vocabulary, and let $\cC$ be a class of finite $\sigma$-structures.  
For any positive integer $k$, the Ash function for the class 
${\cC}$, denoted by $N_{{\cC},k}$, counts the number 
($\leq M_{\sigma,k}$)  of non $k$-equivalent structures in ${\cC}$ of 
size $n$ for all $n\geq 1$.  
\end{definition} 
 
Let $\cB$ be the set of finite Boolean algebras. 
Clearly, we have: 
$N_{{\cB}, k}^{-1}(1)=\{2^{\alpha} \in \N \ | \ \alpha\geq 2\}$,  
$N_{{\cB}, k}^{-1}(0)=\NN^+\setminus\{2^{\alpha} \in \N \ | \ \alpha\geq 2\}$ and  
$N_{{\cB}, k}^{-1}(i)=\emptyset$ for $i\not\in\{0,1\}$. 
From this example, it follows that we cannot expect 
Ash's constant or periodic conjectures to hold for classes of structures. 
A natural question arises: ``Which functions can be Ash's 
function for some class of structures?''. 
Let $f:\NN^+\mapsto\NN$ a function bounded by some constant $M_f$ and 
computable in $E$. Then, it is proved that there exist a vocabulary $\sigma$, a 
$\sigma$-sentence $\varphi$ and a quantifier depth $K$ such that 
$f=N_{Mod_f(\varphi),K}$, where $Mod_f(\varphi)$ denotes the class of finite models of $\varphi$.

Now, let us turn to the ultra-weak Ash conjecture for classes of structures. 
Consider the class $Mod_f(\Psi)$ of finite models of a first-order sentence $\Psi$. 
 
\begin{openquestion}[Ultra-weak Ash conjecture for classes of structures] 
\label{ash-4}
Is it true that
for any finite relational vocabulary $\sigma$, for any first-order 
$\sigma$-sentence $\Psi$, for any positive integer $k$ and for all 
$i\in\NN$, the set $N_{Mod_f(\Psi), k}^{-1}(i)$ is a spectrum?
\end{openquestion} 
 
Only somehow expressive classes of structures are interesting, and in 
particular it is natural to require that $Mod_f(\Psi)$ contains at least one structure of  
size $n$, for all positive 
integer $n$. 

\begin{theorem}\label{theory} 
Let $\sigma$ be a finite relational vocabulary, let $\Psi$ be a 
first-order $\sigma$-sentence such that $Mod_f(\Psi)$ contains at least one structure of  
size $n$, for all positive 
integer $n$ and let $k$ be a positive 
integer. For all $i\in\NN^+$, the set $N_{Mod_f(\Psi), k}^{-1}(i)$ is a 
spectrum if and only if for every $\sigma$-sentence $\varphi$ of 
quantifier depth $\leq k$ which implies $\Psi$ (i.e. every model 
of $\varphi$ is also a model of $\Psi$), the 
set $\NN^+\setminus \spec{\varphi}$ is a spectrum.  
\end{theorem} 
 
In some particular cases, we obtain more information about complements  
of spectra. 
For instance, if $\Psi$ expresses that all binary relations in 
$\sigma$ are functional (which can be done using a quantifier depth $3$ 
sentence), and if $k\geq 3$, then all binary relations in the sentence 
defining ${\mathbb N}^+\setminus \spec{{\varphi}}$ are functional too.  

Let us turn to classes of structures for which the study of Ash's functions 
is as difficult as the general case. 
More precisely, let us consider the following classes of structures:
let 
$\mathcal{SG}$ be the class of simple graphs; 
let 
$2\mathcal{E}$ be the class of two equivalence relations and  
let
$2\mathcal{F}$ be the class of two functions.
 
\begin{proposition}
\label{difficulttheories} 
Let ${\cC}\in\{\mathcal{SG},2\mathcal{E},2\mathcal{F}\}$. If 
for all positive integer $k$ and for all $i\in\NN^+$,  
the set $N_{{\cC}, k}^{-1}(i)$ is a spectrum, then the spectrum 
conjecture holds. 
\end{proposition} 
 
Once again, this is a consequence of various padding results.
Fagin actually proves in \cite{ar:Fagin75a}, 
that simple graphs allow a polynomial padding for all spectra.
It is also the case for two unary functions, as proved by Durand, Fagin and Loescher in 
\cite{proc:DurandFL97}. 
The last result concerning two equivalence relations is proved in 
\cite{phd:Chateau03,ar:ErshovLTT65}.

\newif\iffuse
\fusefalse
\section{Approach IV: Structures of bounded width}
\label{semantic}
\subsection{From restricted vocabularies to bounded tree-width}
In this section we look again at spectra of sentences
with one unary function and a fixed set of unary predicates.
The finite structures which have only one unary function
consist of disjoint unions of components of the form
given in Figure \ref{twfig1}.
\begin{figure}[h]

\includegraphics[height=2cm]{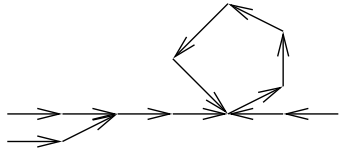}
\caption{Models of one unary function}~\label{twfig1}
\end{figure}
They look like directed forests where the roots are replaced by
a directed cycle.
The unary predicates are just colours attached to the nodes.
The similarity of labeled graphs to labeled trees can be measured
by the notion of {\em tree-width}, and 
in fact, these structures have tree width at most $2$.
Inspired by Theorem \ref{th:gur-she}, E. Fischer and J.A. Makowsky
\cite{ar:FischerM04} generalized Theorem \ref{th:gur-she}
by replacing the restriction on the vocabulary by a
purely model theoretic condition involving the width of a relational
structure.
In this section we discuss their results.
\\ 

\subsubsection{Tree-width}
In the eighties the notion of tree-width of a graph became a central
focus of research in graph theory through the monumental
work of Robertson and Seymour on graph minor closed classes of graphs,
and its algorithmic consequences \cite{ar:RobertsonSeymour86}.
The literature is very rich, but good references and orientation may be found
in \cite{bk:Diestel96,ar:Bodlaender93,ar:Bodlaender98}.
Tree-width is a parameter that measures to what extent
a graph is similar to a tree. Additional unary predicates
do not affect the tree-width. Tree-width of directed graphs
is defined as the tree-width of the underlying undirected 
graph\footnote{
In \cite{ar:JRST99} a different definition is given,
which attempts to capture the specific situation of directed
graphs. But the original definition is the one which is used
when dealing with hypergraphs and general relational structures.
}.
\begin{definition}[Tree-width]
\label{def:1}
A {$k$-tree decomposition} of a graph $G = (V,E)$ is a pair 
$(\{ X_i \mid i \in I\}, T = (I,F))$ 
with $\{X_i \mid i \in I\}$ 
a family of subsets of $V$, one for each node of $T$, and $T$ a tree
such that
\begin{enumerate}
\item ${\bigcup}_{i \in I} X_i = V$.
\item for all edges $(v, w) \in E$ there exists an $i \in I$ with $v \in X_i$ and $w \in X_i$.
\item for all $i,j,k \in I$: if $j$ is on the path from $i$ to $k$ in $T$, then $X_i \cap X_k \subseteq X_j$ in other words, the subset $\{t\mid v\in X_t\}$ is connected for all $v$.
\item for all $i \in I$, $|X_i| \leq k+1$.
\end{enumerate}
A graph $G$ is of {\em tree-width at most $k$} if there exists a $k$-tree
decomposition of $G$.
A class of graphs $K$ is a $TW(k)$-class iff all its members have tree
width at most $k$.
\end{definition}
Given a graph $G$ and $k \in \mathbb{N}$ there are polynomial time,
even linear time, algorithms, which
determine whether $G$ has tree-width $k$, and if the answer is yes,
produce a tree decomposition, cf. \cite{ar:Bodlaender98}. However, if $k$ is part of the input, the problem is $\NP$-complete \cite{ar:ArnborgCorneilProskurowski}

\begin{figure}[h]
\fbox{\epsfig{file=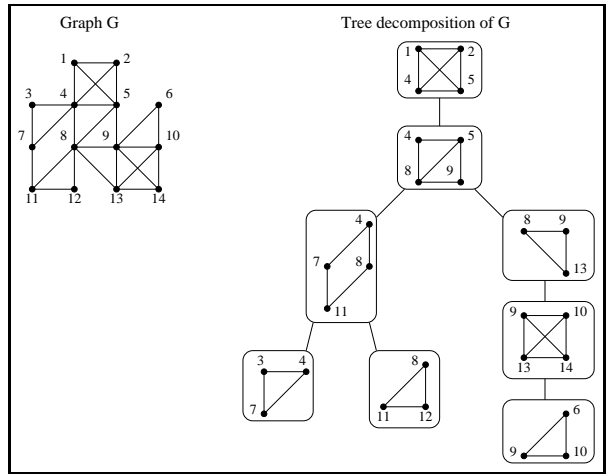,height=6cm}}
\caption{A graph and one of its tree-decompositions}
\end{figure}

Trees have tree-width $1$.
The clique $K_n$ has tree-width $n-1$.
Furthermore, for fixed $k$, the class of finite graphs of 
tree-width at most $k$ denoted by $TW(k)$, 
is $\Msol$-definable.
\begin{example}
The following graph classes are of tree-width at most $k$:
\begin{enumerate}
\item
Planar graphs of radius $r$ with $k=3r$.
\item
Chordal graphs with maximal clique of size $c$ with $k=c-1$.
\item
Interval graphs with maximal clique of size $c$ with $k=c-1$.
\end{enumerate}
\end{example}

\begin{example}
The following graph classes have unbounded tree-width
and are all $\Msol$-definable.
\begin{enumerate}
\item
All planar graphs  and
the class of all planar grids $G_{m,n}$. 
\\
Note that if $n \leq n_0$ for some fixed $n_0 \in \N$, then
the tree-width of the grids $G_{m,n}, n \leq n_0$, is bounded by $2n_0$.
\item
The regular graphs of degree $r, r \geq3 $ have unbounded tree-width.
\end{enumerate}
\end{example}
Tree-width for labeled graphs can be generalized to
arbitrary relational structures in a straightforward way.
Clause
(ii) in the above definition is replaced by
\begin{description}
\item[(ii-rel)] 
For each $r$-ary relation $R$,  if $\bar{v} \in R$, 
there exists an $i \in I$ with $\bar{v} \in X_i^r$. 
\end{description}
This was first used in \cite{ar:FederVardi99}.

It is now natural to ask, whether Theorem \ref{th:gur-she}
can be generalized to arbitrary vocabularies, provided one
restricts the spectrum to structures of fixed tree-width $k$.
Indeed, E. Fischer and J. Makowsky \cite{ar:FischerM04} have established
the following:

\begin{theorem}[E. Fischer and J.A. Makowsky (2004)]
\label{FM:1}
\ \\ 
Let $\phi$ be an $\Msol$ sentence
and $k \in \N$.
Assume that all the models of $\phi$ are in $TW(k)$.
Then $spec(\phi)$ is ultimately periodic.
\end{theorem}

\subsection{Extending the logic}

First one observes that the logic $\Msol$
can be extended by modular counting quantifiers 
$C_{k,m}$, 
where $C_{k,m} x\ \phi(x)$ is interpreted as
``there are, modulo $m$, exactly $k$ elements satisfying $\phi(x)$''.
We denote the extension of $\Msol$
obtained by adding, for all $k,m \in \N$ the quantifiers
$C_{k,m}$, 
by $\Cmsol$.

Typical graph theoretic concepts expressible in $\Fol$
are
the presence or absence (up to isomorphism) of a fixed (induced) subgraph $H$,
and
fixed lower or upper bounds on the degree of the vertices
(hence also $r$-regularity).
Typical graph theoretic concepts expressible in $\Msol$ but not in $\Fol$
are
connectivity, $k$-connectivity, reachability,
$k$-colorability (of the vertices), and
the presence or absence of a fixed (topological) minor.
The latter includes planarity,
and more generally, graphs of a fixed genus $g$.
Typical graph theoretic concepts expressible in $\Cmsol$ but not in $\Msol$
are
the existence of an eulerian circuit (path),
the size of a connected component being a multiple of $k$,
and
the number of connected components is a multiple of $k$.
All the non-definability statements above
can be proved using Ehren\-feucht-Fra\"{\i}ss\'e games.
The definability statements are straightforward.

Second, one observes that very little of the definition of
tree-width is used in the proof. The techniques used in the proof
of Theorem \ref{FM:1} can be adapted to other notions of
width of relational structures, such as {\em clique-width},
which was introduced first in
\cite{ar:CourcelleEngelfrietRozenberg93}
and studied more systematically in
\cite{ar:CourcelleOlariu00},
and to {\em patch-width}, introduced in 
\cite{ar:FischerM04}.

\subsection{Ingredients of the proof of Theorem \ref{FM:1}}
To prove Theorem \ref{FM:1}, and its generalizations below,
the authors use three tools:

\begin{enumerate}
\item
A generalization of the Feferman-Vaught Theorem for $k$-sums
of labeled graphs to the logic
$\Cmsol$, due to B. Courcelle, \cite{ar:courcelle90}, and further refined
by J.A. Makowsky in \cite{ar:MakowskyTARSKI}.
\item
A reduction of the problem to spectra of labeled trees by a technique
first used by B. Courcelle in \cite{ar:CourcelleIC95} in his study
of graph grammars.
\item
An adaptation of the Pumping Lemma for labeled trees, cf.
\cite{bk:hb-forla3-1}.
\end{enumerate}
The proof of Theorem \ref{FM:1} is quite general. 
Its proof can be adapted
to stronger logics and other notions of width of relational structures.
However, the details are rather technical.

\subsection{Clique-width}

A $k$-coloured $\tau$-structures is a 
$\tau_k = \tau \cup \{P_1, \ldots, P_k\}$-structure
where $P_i, i \leq k$ are unary predicate symbols
the interpretation of which are disjoint (but can be empty).

\begin{definition}
Let $\mathfrak{A}$ be a $k$-coloured $\tau$-structure.
\begin{enumerate}
\item
(Adding hyper-edges)
Let $R \in \tau$ be an $r$-ary relation symbol.
\\
$\eta_{R, P_{j_1}, \ldots , P_{j_r}}(\mathfrak{A})$ denotes the
$k$-coloured $\tau$ structure $\mathfrak{B}$ 
with the same universe as $\mathfrak{A}$, 
and for each $S \in \tau_k$, $S \neq R$ 
the interpretation is also unchanged.
Only for $R$ we put 
$$
R^B = R^A \cup \{ \bar{a} \in A^r : a_i \in P_{j_i}^A \}.
$$
We call the operation $\eta$ {\em hyper edge creation},
or simply {\em edge creation} in the case of directed graphs.
In the case of undirected graphs we 
denote by $\eta_{P_{j_1}, P_{j_2}}$
the operation of adding the corresponding undirected edges.
\item
(Recolouring)
$\rho_{i,j}(\mathfrak{A})$ denotes the
$k$-coloured $\tau$ structure $\mathfrak{B}$ 
with the same universe as $\mathfrak{A}$, 
and all the relations unchanged but for $P_i^A$ and $P_j^A$.
We put
$$
P_i^B = \emptyset \mbox{ and }
P_j^B = P_j^A \cup P_i^A.
$$
We call this operation {\em recolouring}.
\item
(modification via quantifier free translation)
More generally, for $S \in \tau_k$ of arity $r$ and $B(x_1, \ldots , x_r)$
a quantifier free $\tau_k$-formula,  
$\delta_{S,B}(\mathfrak{A})$ denotes the
$k$-coloured $\tau$ structure $\mathfrak{B}$ 
with the same universe as $\mathfrak{A}$, 
and for each $S' \in \tau_k$, $S' \neq S$ 
the interpretation is also unchanged.
Only for $S$ we put 
$$
S^B =  \{ \bar{a} \in A^r : \bar{a} \in B^A \}.
$$
where $B^A$ denotes the interpretation of $B$ in $\mathfrak{A}$.
\end{enumerate}
Note that the
operations of type $\rho$ and $\eta$ are special cases of the operation
of type $\delta$.
\end{definition}

\begin{definition}[Clique-Width, \cite{ar:CourcelleOlariu00,ar:MakowskyTARSKI}]
\label{def:cw}
\ 
\begin{enumerate}
\item
Here $\tau = \{R_E\}$ consist of the symbol for the edge relation.
Given a graph $G=(V,E)$, the \em{clique-width} of $G$
(\em{cwd(G)}) is the minimal number of colours required to obtain
the given graph as an $\{R_E\}$-reduct from
a $k$-coloured graph constructed inductively 
from coloured singletons  and closure under the following
operations:
\begin {enumerate}
\item disjoint union ($\sqcup$)
\item recolouring ($\rho_{i \to j}$)
\item edge creation ($\eta_{E, P_i, P_j}$)
\end {enumerate}
\item
For $\tau$ containing more than one binary relation symbol,
we replace the edge creation by the corresponding hyper edge creation
$\eta_{R, P_{j_1}, \ldots , P_{j_r}}$ for each $R \in \tau$.
\item
A class of $\tau$-structures is a $CW(k)$-class if all its
members have clique-width at most $k$.
\end{enumerate}
\end{definition}
%
If $\tau$ contains a unary predicate symbol $U$,
the interpretation of $U$ is not affected by the operations
recoloring or edge creation.
Only the disjoint union affects it.

A description of a graph or a structure using these operations is called a
{\em clique-width parse term} (or {\em parse term}, if no
confusion arises). 
Every structure of size $n$ has clique-width at most $n$.
The simplest class of graphs of unbounded
tree-width but of clique-width at most $2$ are the cliques. Given
a graph $G$ and $k \in \mathbb{N}$, determining whether $G$ has
clique-width $k$ is in $\mathbf{NP}$. A polynomial time algorithm
was presented for $k \leq 3$ in \cite{ar:Corneil-etal99}.

It was shown in
\cite{ar:FellowsRosamondRoticsSzeider2005a,pr:FellowsRosamondRoticsSzeider2006}
that for fixed $k \geq 4$ the problem is $\mathbf{NP}$-complete.
The recognition problem for clique-width of relational structures
has not been studied so far even for $k=2$.
The relationship between tree-width and clique-width was studied
in \cite{ar:CourcelleOlariu00,ar:GliksonMakowsky03,pr:AdlerAdler08}.

\begin{theorem}[Courcelle and Olariu (2000)]
Let $K$ be a  $TW(k)$-class of graphs.
Then $K$ is a $CW(m)$-class of graphs with
$m \leq 2^{k+1}+1$.
\end{theorem}
\begin{theorem}[Adler and Adler (2008)]
For every non-negative integer $n$ there is a structure $\mathfrak{A}$ with
only one ternary relation symbol such that $\mathfrak{A} \in TW(2)$
and $\mathfrak{A} \not\in CW(n)$.
\end{theorem}
The following examples are from
\cite{ar:MakowskyRotics99,ar:GolumbicRotics01}.
\begin{example}[Classes of clique-width at most $k$]
\ 
\begin{enumerate}
\item
The cographs with $k=2$.
\item
The distance-hereditary graphs with $k=3$.
\item
The cycles $C_n$ with $k=4$.
\item
The complement graphs $\bar{C_n}$ of the cycles $C_n$ with $k=4$.
\end{enumerate}
The cycles $C_n$ have tree-width at most $2$,
but the other examples have unbounded tree-width.
\end{example}

\begin{example}[Classes of unbounded clique-width] 
\ 
\begin{enumerate}
\item 
The class of all finite graphs.
\item
The class of unit interval graphs.
\item
The class of permutation graphs.
\item
The regular graphs of degree $4$ have unbounded clique-width.
\item The class grids $Grid$, consisting of the graphs $Grid_{n \times n}$.
\end{enumerate}
\end{example}
For more non-trivial examples, cf. 
\cite{ar:MakowskyRotics99,ar:GolumbicRotics01}.
In contrast to $TW(k)$, we do not know whether the class of all
$CW(k)$-graphs is $\Msol$-definable.

To find more examples it is useful to note, cf. \cite{ar:MakowskyMarino01c}:
\begin{proposition}
If a graph is of clique-width at most $k$ and $G'$ is an induced subgraph
of $G$, then the clique-width of $G'$ is at most $k$.
\end{proposition}

In \cite{ar:FischerM04} the following is shown:
\begin{theorem}[E. Fischer and J.A. Makowsky (2004)]
\label{FM:Main}
\ \\
Let $\phi \in \CMSOL{\tau}$ be such that all
its finite models have clique-width at most $k$.
Then there are  $m_0, n_0 \in \N$
such that if $\phi$ has a model of size $n \geq n_0$
then $\phi$ has also a model of size $n+m_0$.
\end{theorem}
From this we get immediately a further generalization
of Theorem \ref{FM:1}.
\begin{corollary}
Let $\phi \in \CMSOL{\tau}$ be such that all
its finite models have clique-width at most $k$.
Then $spec(\phi)$ is ultimately periodic.
\end{corollary}

\subsubsection{Patch-width}
Here is a further generalization of clique-width for which
our theorem still works. 
The choice of operation is discussed 
in detail in \cite{ar:CourcelleMakowsky00}.

\begin{definition}
Given a  $\tau$-structure $\mathfrak{A}$, the \em{patch-width} of $G$
(\em{pwd(G)}) is the minimal number of colours required to obtain
$\mathfrak{S}$
as a $\{\tau\}$-reduct from
a $k$-coloured $\tau$-structure inductively 
from fixed finite number of $\tau_k$-structures and closure under the following
operations:
\begin {enumerate}
\item disjoint union ($\sqcup$),
\item recoloring ($\rho_{i \to j}$) and
\item modifications ($\delta_{S,B}$).
\end {enumerate}
A class of $\tau$-structures is a $PW_{\tau}(k)$-class if all its
members have patch-width at most $k$.
\end{definition}
A description of a $\tau$-structure using these operations is called a
{\em patch term}. 

\begin{example}
\ 
\begin{enumerate}
\item
In \cite{ar:CourcelleOlariu00} it is shown that
if a graph $G$ has clique-width at most $k$ then its complement graph $\bar{G}$
has clique-width at most $2k$.
However, its patch-width is also $k$ as $\bar{G}$ can be
obtained from $G$ by $\delta_{E, \neg E}$.
\item
The clique $K_n$ 
has clique-width $2$.
However if we consider graphs as structures on a two-sorted universe (respectively for vertices and edges), then $K_n$ has clique-width $c(n)$
and patch-width $p(n)$ where $c(n)$ and $p(n)$ are functions
which tend to infinity.
This will easily follow from Theorem \ref{FM:Parikh-1}.
For the clique-width of $K_n$ 
as a two-sorted-structure 
this was already shown in \cite{th:rotics}.
\end{enumerate}
\end{example}

\begin{remark}
In \cite{ar:CourcelleMakowsky00} it is shown that
a class of graphs of patch-width at most $k$ is 
of clique-width at most $f(k)$ for some  function $f$.
It is shown in \cite{ar:FischerMakowskyPATCH}
that this is not true
for relational structures in general.
\end{remark}


In the definition of patch-width we allowed only unary predicates
as auxiliary predicates (colours). 
We could also allow $r$-ary predicates and speak
of $r$-ary patch-width.
The theorems where bounded patch-width is required
are also true for this more general case.
The relative strength of clique-width and the various forms of patch-width
are discussed in \cite{ar:FischerMakowskyPATCH}.

In \cite{ar:FischerM04} the following is shown:
\begin{theorem}[E. Fischer and J.A. Makowsky (2004)]
\label{FM:Main-p}
\ \\
Let $\phi \in \CMSOL{\tau}$ be such that all
its finite models have patch-width at most $k$.
Then there are  $m_0, n_0 \in \N$
such that if $\phi$ has a model of size $n \geq n_0$
then $\phi$ has also a model of size $n+m_0$.
\end{theorem}
From this we get yet another generalization of Theorem \ref{FM:1}. 
\begin{corollary}
Let $\phi \in \CMSOL{\tau}$ be such that all
its finite models have patch-width at most $k$.
Then $spec(\phi)$ is ultimately periodic.
\end{corollary}

More recent work on spectra and patch-width may be found
in \cite{pr:Shelah04,pr:DoronShelah06}.
\ \\

\subsection{Classes of unbounded patch-width}
Theorem \ref{FM:Main-p} gives a new method to show that
certain classes $K$ of graphs have unbounded tree-width, clique-width or
patch-width.

To see this we look at the class $Grid$ of all grids $Grid_{n \times n}$.
They are known to have unbounded tree-width, cf. \cite{bk:Diestel96},
and in fact, every minor closed class of graphs of unbounded
tree-width contains these grids.
They were shown to have unbounded clique-width in \cite{ar:GolumbicRotics01}.
However, for patch-width these arguments do not work.
On the other hand $Grid$ is $\Msol$-definable, and its spectrum
consists of the numbers $n^2$, so by Theorems 
\ref{FM:Main} and
\ref{FM:Main-p}, the unboundedness follows directly.

In particular, as this is also true for every $K' \supseteq K$,
the class of all graphs is of unbounded 
patch-width.

Without Theorem \ref{FM:Main-p}, there was only a conditional
proof of unbounded patch-width available.
It depends on the assumption that the polynomial hierarchy
$\Sigma^{\mathrm{P}}$ does not collapse to $\bNP$.
The argument then procceds as follows:
\begin{enumerate}
\item
Checking patch-width at most $k$ of
a structure $\mathfrak{A}$,
for $k$ fixed,
is in $\bNP$.  Given a structure $\mathfrak{A}$, 
one just has to guess a patch-term
of size polynomial in the size of $\mathfrak{A}$.
\item
Using the results of \cite{ar:MakowskyTARSKI}
one gets that checking a $\CMSOL{\tau}$-property $\phi$
on the class $PW_{\tau}(k)$ is in $\bNP$, whereas, by
\cite{ar:MakPnueli96}, there are 
$\Sigma_n^{\mathrm{P}}$-hard problems
definable in $\Msol$ for every level
$\Sigma_n^{\mathrm{P}}$ of the polynomial hierarchy.
\item
Hence, if the polynomial hierarchy does not collapse to $\bNP$,
the class of all $\tau$-structures is of unbounded
patch-width, provided $\tau$ is large enough.
\end{enumerate}

\begin{openquestion}
\label{patch-complexity}
What is the complexity of checking whether a $\tau$-structure
$\mathfrak{A}$ has patch-width 
at most $k$,
for a fixed $k$?
\end{openquestion}

\subsection{Parikh's Theorem}
R. Parikh's celebrated theorem, first proved in \cite{ar:Parikh66},
counts the number of occurences of letters in $k$-letter words of context-free
languages. For a given word $w$, the numbers of these occurences
is denoted by a vector  $n(w) \in \N^k$, and the theorem states
\begin{theorem}[Parikh 1966]
\label{TH:Parikh}
For a context-free language $L$, the set
$Par(L)= \{ n(w) \in \N^k : w \in L \}$
is semilinear.
\end{theorem}
A set $X \subseteq \N^s$ is  {\em linear
in $\N^s$}
iff there is vector $\bar{a} \in \N^s$
and a matrix $M \in \N^{s \times r}$ such that
$$
X = A_{\bar{a}, \bar{M}}
= \{ \bar{b} \in \N^s : \mbox{ there is } \bar{u} \in \N^r \mbox{ with }
\bar{b} = \bar{a} + M\cdot \bar{u} \}
$$
Singletons are
linear sets
with $M = 0$.
If $M \neq 0$ the series is {\em nontrivial}.
$X \subseteq \N^s$ is {\em semilinear
in $\N^s$} iff
$X$ is a finite union of linear sets
$A_i \subseteq \N^s$.
The terminology is from \cite{ar:Parikh66}, and has since become
standard terminology in formal language theory.
We note that for unary languages, Parikh's Theorem looks at the spectrum
of context-free languages.

B. Courcelle has generalized 
Theorem \ref{TH:Parikh}
further to context-free
vertex replacement graph grammars,
\cite{ar:CourcelleIC95}.
We want to generalize Theorem \ref{TH:Parikh} to spectra.
Rather than counting occurences of letters, 
we look at many-sorted
structures and the sizes of the different sorts,
which we call many-sorted spectra.
In \cite{ar:FischerM04} the following theorem is proved:

\begin{theorem}[E. Fischer and J.A. Makowsky (2006)]
\label{FM:Parikh-1}
\ \\
Let $K$ be a class of
$\Cmsol$-definable many-sorted
relational structures which are of patch-width at most $k$.
Then
the many-sorted spectrum of $K$ forms a semilinear set.
\end{theorem}

\appendix
\section{A review of some hardly accessible references}~\label{section technical}

This section contains a detailed presentation of the material of Subsection 
\ref{subsection recursion}. 
Note that the proofs sketched here do not necessarily correspond to the original proofs.

\subsection{Asser's paper}\label{subsection asser}
In chronological ordering, the first paper related to spectra is \cite{ar:Asser55}, 
in German, due to Asser in 1955. 
Though it does not use the name ``spectrum'', nor refer to Scholz in its title or in the text, 
the long introduction clarifies the context in which the concept of spectrum was born. The author addresses the general question of classes of cardinal numbers (not only natural numbers) so-called ``representable'' by a sentence of first-order logic with equality, both in the framework of satisfiability theory and validity theory. Here a first-order sentence $\varphi$ represents a given (finite or infinite) cardinality $m$ regarding satisfiablity if there is a structure which domain has cardinality $m$ which is a model of $\varphi$ (i.e. for finite $m$, it means $m\in \spec{\varphi}$), and regarding validity, if $\varphi$ holds in every structure with cardinality $m$. Asser first notices that $\varphi$ represents $m$ regarding satisfiability if and only if $\neg\varphi$ does not represent $m$ regarding validity, so that validity reduces to satisfiability via complement. Then, he remarks that, from  L{\"o}wenheim-Skolem Theorem \cite{ar:Loewenheim15, ar:Skolem20}, the representation question in satisfiability theory for infinite cardinalities is trivial: the first-order sentence $\varphi$ either has no infinite model (and in this case it has finite models in finitely many finite cardinalities only) or has models in every infinite cardinality. Hence, the problem actually is about exactly which sets of natural numbers are the set of cardinalities of finite models of first-order sentences, i.e. what we would call spectra. In a footnote, one reads ``this question was also asked by Scholz as a problem in \cite{ar:Scholz52}''. 

With this background, Asser's aim is to give a purely arithmetical characterization of spectra. This is done via an arithmetical encoding of finite structures, first-order sentences and satisfiability.
Let us make precise Asser's construction.

Note that in the sequel ``characteristic functions'' 
(of sets or relations) are not taken in the usual way: a unary function $f$ is 
said to be the characteristic function of the set of integers $n$ such that 
$f(n)=0$. It is only a technical matter to come back to the usual definition 
with little machinery, for instance use $\chi(n)=1\dot{-}f(n)$ (so-called {\em 
modified substraction} i.e. $x\dot{-}y=x-y$ if $x\geq y$ and $0$ otherwise). 
Using this alternative definition, characteristic functions are not required to 
be 0-1-valued.
\medskip 

W.l.o.g., let $\varphi$ be a sentence in relational Skolem normal form, i.e. 
$\varphi\equiv\forall x_1\ldots\forall x_r\exists x_{r+1}\ldots\exists 
x_s\psi(x_1,\ldots x_s)$, where $\psi(x_1,\ldots x_s)$ is a Boolean combination 
of atomic formulas $R_i^{(a_i)}(x_{j_1},\ldots x_{j_{a_i}})$ with $i=1,\ldots,t$ 
and of atoms $x_{l_1}=x_{l_2}$. Assume that $\psi$ contains $u$ different atoms
of type 
$R_i^{(a_i)}(x_{j_1},\ldots x_{j_{a_i}})$ and $v$ different atoms of type 
$x_{l_1}=x_{l_2}$. Let $\Psi:\{0,1\}^{u+v}\longrightarrow\{0,1\}$ be the 
Boolean function associated to the propositional version of $\psi$ (using the 
convention that $0$ encodes true and $1$ encodes false). 

Denote by 
$Bit_k(y,z_1,\ldots,z_k,n)$ the binary digit of $y$ of rank 
$\sum_{l=1}^kz_l\cdot n^{l-1}$, assuming $y<2^{n^k},z_1<n,\ldots,z_k<n$.
 Encode the $k$-ary relation 
$R$ on the domain $\{0,\ldots,n-1\}$ by the number $y<2^{n^k}$ such that 
$Bit_k(y,z_1,\ldots,z_k,n)=0$ if and only if $R(z_1,\ldots,z_k)$ holds. 
Let 
$\delta(z_1,z_2)=0$ if $z_1=z_2$ and $1$ otherwise.
Let $\Psi^*(y_1,\ldots,y_t,x_1,\ldots,x_s,n)$ be obtained from $\Psi$ by
replacing each atom $R_i^{(a_i)}(x_{j_1},\ldots x_{j_{a_i}})$ by 
$Bit_{a_{i}}(y_{i},x_{j_1},\ldots,x_{j_{a_{i}}},n)$
and every atoms $x_{l_1}=x_{l_2}$ by $\delta(x_{l_1},x_{l_2})$.
\raus{
$\Psi^*(y_1,\ldots,y_t,x_1,\ldots,x_s,n)=\Psi(digit_{a_{i_1}}(y_{i_1},x_{j^1_1},\ldots,x_{j^1_{a_{i_1}}},n),$ 
$\ldots,Bit_{a_{i_u}}(y_{i_u},x_{j^u_1},\ldots,x_{j^u_{a_{i_u}}},n),\delta(x_{l^1_1},x_{l^1_2}),\ldots,\delta(x_{l^v_1},x_{l^v_2}))$. 
}
The first-order quantifiers $\forall x_1\ldots\forall x_r\exists 
x_{r+1}\ldots\exists x_s$ are dealt with by defining 

\[\Psi^{**}(y_1,
\ldots,y_t,n)=\sum_{x_1=0}^{n-1} 
\ldots\sum_{x_r=0}^{n-1}\prod_{x_{r+1}=0}^{n-1} \ldots\prod_{x_{s}=0}^{n-1} 
\Psi^*(\overline{y},\overline{x},n).\]

 Note the non-standard use of $\sum$ 
for $\forall$ and $\prod$ for $\exists$, due to the fact that $0$ encodes true 
and $1$ encodes false. Finally the characteristic function of the spectrum of 
the sentence $\varphi\equiv\forall x_1\ldots\forall x_r\exists 
x_{r+1}\ldots\exists x_s\psi(x_1,\ldots x_s)$ is 
$\chi(n)=\prod_{y_{1}=0}^{2^{n^{a_1}}-1} 
\ldots\prod_{y_{t}=0}^{2^{n^{a_t}}-1}\Psi^{**}(y_1,\ldots,y_t,n)$. This 
construction is clearly elementary. Conversely, it is also easy to verify that 
any function defined as $\chi(n)=\prod_{y_{1}=0}^{2^{n^{a_1}}-1} 
\ldots\prod_{y_{t}=0}^{2^{n^{a_t}}-1}\Psi^{**}(y_1,\ldots,y_t,n)$, where 
$\Psi^{**}$ is obtained from some Boolean function $\Psi$ by the same type of 
construction, is the characteristic function of the spectrum of the 
corresponding first-order sentence. Hence we have the following result.

\begin{theorem}\label{asser1}
A set $\mathcal S$ is a spectrum iff its characteristic function $\chi$ has the form 
$\chi(n)=\prod_{y_{1}=0}^{2^{n^{a_1}}-1} \ldots\prod_{y_{t}=0}^{2^{n^{a_t}}-1}\Psi^{**}(y_1,\ldots,y_t,n)$, where $\Psi^{**}$ is obtained from some Boolean function $\Psi$ by the above construction.
\end{theorem}

Note that Asser judges his result ``non satisfactory'', in particular because this paraphrastic characterization is of no help in proving that a given set is or not a spectrum, or in providing any concrete spectrum.
However, Asser's characterization is enough to prove Theorem \ref{theorem asser}, that we restate here for sake of self-containment.

\begin{quote}
{\sc Theorem}  \ref{theorem asser}\ \ $\SPEC\subsetneq {\mathcal E}_{\star}^3$
\end{quote}

The inclusion follows from the fact that Theorem \ref{asser1} provides 
elementary characteristic functions for spectra. The properness is obtained by 
diagonalization.

As a conclusion, Asser asks some questions, that have essentially remained open 
up to now. 
First, he asked for a recursive characterization of spectra. He 
notes that there are actually two different problems. 
The first one 
asks for 
a recursively defined class of functions, i.e. 
a class of functions defined via some basis functions 
and closure under some functional operations, such that the unary functions in 
this class are exactly the characteristic functions of spectra. 
Second, he
asks for a recursively defined class of functions, but now such 
that the unary functions in it enumerate exactly the spectra, i.e. a set $S$ is 
a spectrum if and only if $S=f({\mathbb N})$ for some $f$ in the class. Note 
that this is not the most commonly admitted meaning for enumeration, because 
the enumeration functions are usually required to be strictly increasing, which 
is not the case here. 

Next, Asser 
refers to ``work in 
progress'' that proves that a large class of unary functions are 
characteristic functions of spectra, among which the following arithmetically 
defined sets: prime numbers, multiples of a given integer $k$, powers of a 
given $k$, $k$-th powers, composite numbers.a

Finally, the third and most famous open question proposed in this paper 
is usually known as Asser's Problem (Open Question \ref{open-assers-problem}) and 
asks whether spectra are closed under complement.

\subsection{Mostowski's paper}\label{subsection mostowski}

A paper  almost simultaneous with Asser's is  \cite{ar:Mostowski56}, due to A. Mostowski in 
1956. It also adresses recursive characterization of spectra, and 
explicitly uses the name ``spectrum''. It is noticed that ``The results of 
Asser overlap in part with results which I have found in 1953 while attempting 
(unsuccessfully) to solve Scholz's problem (cf. Roczniki Polskiego Towarzystwa 
Matematycznego, series I, vol. 1 (1955), p.427). I shall give here proofs of my 
results which do not overlap with Asser's.''~\footnote{Thanks to J. Tomasik, we have 
seen  a translation of the Polish reference. It is the abstract 
of a seminar given by Andrzej Mostowski on October, 16. 1953. In addition to the 
following material, it is also stated that spectra form a strict subclass of 
primitive recursive sets, a result which indeed overlaps with Asser's.
}. 

Here, A. Mostowski defines a class of functions denoted by $K$ as follows. 

\begin{definition}
\label{classK}
The class $K$ is the least class
\begin{itemize}
\item[--] containing the functions $Z_k,U^i_k,S,C$ respectively defined by:
\begin{itemize}
\item[-] $Z_k(x_1,\ldots,x_k,n)=0$
\item[-]  $U^i_k(x_1,\ldots,x_k,n)=min(x_i,n)$,  for $i=1,\ldots,k$
\item[-] $S(x,n)=min(x+1,n)$
\item[-] $C(x)=n$
\end{itemize}
\item[--] closed under composition: $f(x_1,\ldots, 
x_{j-1},g(y_1,\ldots,y_p,n),x_{j+1},\ldots,x_k,n)$
 
\item[--] closed under recursion: 
$\left\{\begin{array}{l}f(0,\vec{x},n)=g(\vec{x},n)\\ 
f(x+1,\vec{x},n)=min(h(x,f(x,\vec{x},n),\vec{x},n),n)\end{array}\right.$
 
\end{itemize}

\end{definition}

The basis functions $Z_k$, $U^i_k$ and $S$ are intended as the classical zero, projections and successor functions, but the special variable $n$ always bounds their values. The function $C$ is intended as a maximum function. The functional operations composition and recursion are also bounded by $n$.
The main result of Mostowski's paper is the following theorem.

\begin{theorem}
For any unary function $f\in K$, the set $\{n+1\ |\ f(n)=0\}$ is a spectrum.
\end{theorem}

Let us give an idea of the proof via an example.
Consider the functions $f$, $g$ and $h$ defined as follows: 

\begin{itemize}
  \item[-] $f(x,n)=1$ if $x=0$ and $f(x,n)=0$ otherwise. 

I.e.
$\left\{\begin{array}{l}f(0,n)=S(Z(C(n),n),n)\\ 
f(x+1,n)=min(Z(f(x,n),n),n)\end{array}\right.$
  \item[-] $g(x,n)=0$ if $x$ is even and $g(x,n)=1$ otherwise. 

I.e.  
$\left\{\begin{array}{l}g(0,n)=Z(C(n),n)
\\ g(x+1,n)=min(f(g(x,n),n),n)\end{array}\right.$ 
\item[-] $h(n)=g(C(n),n)$.

\end{itemize}

\noindent Clearly we have $h\in K$ and $h(n)=0$ if and only if $n$ is even. 
Let us derive from the definition of $h$ a sentence $\psi$ in the vocabulary 
$$\sigma=\{\leq, min,max, Succ^{(2)}, R_f^{(3)}, R_g^{(3)}, R_h^{(2)}\}$$ 
such that $\psi$ has a model with $n+1$ elements if and only if $h(n)=0$ 
(i.e. $\spec{\psi}$ is the set of odd numbers). 
The key point of the construction is that the functions in the class $K$ can be 
interpreted as functions on finite structures 
(eg. from $\{0,\ldots,n\}^k$ to $\{0,\ldots,n\}$) 
without loss of generality, because of the special variable $n$ that bounds all their values.

The sentence $\psi$ first expresses the fact that $\leq$ is a linear ordering, 
$min$ and $max$ are its first and last elements and $Succ$ its successor 
relation. Then, $\psi$ describes the behavior of the predicates $R_f$, $R_g$ 
and $R_h$ corresponding to the graphs of the functions $f$, $g$ and $h$. For 
instance, $R_g$ obeys the conjunction of the following sentences:

\begin{itemize}
  \item[-] $g$ is functional in its first variable:
   
$\forall x,y \left(y=max\longrightarrow\exists !z R_g(x,y,z)\right)$

\item[-] the second 
variable in $g$ is always $n$:

$\forall x,y,z \left(R_g(x,y,z)\longrightarrow y=max\right)$

\item[-] description of the base case of the definition of $g$:

$\forall x,y,z \left[\left(R_g(x,y,z)\wedge x=min\right) \longrightarrow 
z=min\right]$

\item[-] description of the recursive recursive case of the definition of $g$:

$\forall x,y,z [(R_g(x,y,z)\wedge \neg x=min)$ 

$\longrightarrow \exists t,u (Succ(t,x)\wedge R_g(t,y,u)\wedge 
R_f(u,y,z))]$ 
\end{itemize}

\noindent Our goal is then achieved by adding to $\psi$ the following condition: 

$\forall x,y\left(R_h(x,y)\longrightarrow (x=max \wedge y=min)\right)$

\noindent Finally, it is clear that 
$\spec{\psi}$
is the set of odd numbers as required.

Mostowski asks if the converse is true, i.e. 
\begin{openquestion}
\label{mostowski-K}
Is every 
spectrum representable as $\{n+1\ |\ f(n)=0\}$ for some function $f\in K$? 
\end{openquestion}
No answer is known up to now.

As a conclusion, new examples of spectra are presented: the set of integers having the form $n!$ for some $n$, and the set $\{n\ |\ n^2+1  \mbox{ is prime }\}$. 
Also, Mostowski asks whether Fermat's prime numbers, i.e. primes of the form $2^{2^n}+1$, form a spectrum. This question can be understood in two different ways, as noticed by Bennett: which one of the sets $A=\{p\ |\ p\mbox{ is prime and }p=2^{2^n}+1\mbox{ for some integer }n\}$ and $B=\{n\ |\ 2^{2^n}+1\mbox{ is prime }\}$ is intended ? Using rudimentary relations, the set $A$ is easily proved to be a spectrum, whereas it is still not known for the set $B$.
\medskip

Finally, let us remark that it is ordinarily considered that what Mostowski 
proved is that the unary relations in ${\mathcal E}^2_{\star}$ are spectra. This is not exactly the case, 
but the legend is most probably due to the fact that Bennett attributes this result to Mostowski. 
However, Bennett also notes that, even if it is easy to prove that 
$K\subseteq{\mathcal E}^2$, it is not clear that the bounded version of 
any function in ${\mathcal E}^2$ 
(i.e. $f_{b}(x_1,\ldots,x_k,n)=min(f(x_1,\ldots,x_k),n)$) is in $K$. 
Mostowski's construction crucially relies on the fact that the functions in $K$ 
are bounded by their last variable, and does not generalize to functions in 
${\mathcal E}^2$. 
In contrast, it is not difficult to verify that the bounded versions of 
addition and multiplication are in $K$, 
and consequently that the rudimentary relations have their characteristic functions in $K$.
Whatever, it is true that the unary relations in ${\mathcal E}^2_{\star}$ 
are indeed spectra, see Corollary \ref{bennett-corr}.

\subsection{Bennett's thesis}\label{subsection bennett}
This is a huge work titled ``On spectra'' \cite{phd:Bennett62}, but which also deals with a lot of other subjects. Bennett's thesis is unpublished, and only available via library services.
It is one of the remarkable early texts anticipating later developments
in finite model theory, definability theory and complexity theory.
It contains a characterization (and various definitions) of 
rudimentary sets and already relates spectra to space bounded Turing machines, thus catching a glimpse of many of the results
concerning spectra that were formulated and proved 
in more modern language after 1970.

Not only first-order spectra are considered by Bennett, but also spectra of higher order logics, and not only sets, but also many-sorted sets, all in all spectra of the whole theory of types. This full generality makes the notations quite clumsy. The use of many-sorted structures corresponds to relations of arity greater than one, and the use of higher order logics provides more complicated relations.

We shall limit ourselves with the cases of one-sorted (i.e. ordinary) spectra 
of orders one and two.
Note that the first item of Theorem \ref{TH Bennett} is also stated as Theorem \ref{th spec srud} in Subsection \ref{subsection recursion}.
\medskip

\begin{theorem}[Bennett, 1962~\cite{phd:Bennett62}]~\label{TH Bennett}
\ 
\begin{enumerate}
\item~\label{Bennett1} A set $S\subseteq{\mathbb N}$ is a first-order spectrum iff it can be defined by a formula of the form $\exists y\!\leq\! 2^{x^j} R(x,y)$ for some $j\geq 1$, where $R$ is strictly rudimentary.
\item~\label{Bennett2} A set $S\subseteq{\mathbb N}$ is a second-order spectrum iff it can be defined by a formula of the form $\exists y\!\leq\! 2^{x^j} R(x,y)$ for some $j\geq 1$, where $R$ is rudimentary.
\end{enumerate}
\end{theorem}

Spectra of higher order are characterized by similar features: spectra of 
order $2n$ correspond to rudimentary relations prefixed by an existential 
quantifier bounded by an iterated exponential $2^{\!^{.^{.^{2^{x^j}}}}}$ with 
height $n$, and spectra of order $2n-1$ correspond to strictly rudimentary 
relations prefixed by an existential quantifier bounded by an iterated 
exponential with height $n$. Spectra of sentences over a 
$d$-sorted universe have the same types of characterizations, using $\exists 
y\leq 2^{\max(x_1,\ldots,x_d)^j}\ R(x_1,\ldots,x_d,y)$. Finally, the spectra of 
the whole type theory are characterized as the elementary relations.


\medskip

Bennett also introduces several other subrudimentary classes,  
respectively called ``strongly'', ``positive'' and ``extended'' 
rudimentary relations, which yield a bunch of slightly different 
characterizations of spectra, which may witness various unsuccessful attempts 
to design a truly satisfactory characterization. 
In this survey, we shall limit ourselves to $\RUD$ and $\SRUD$.

\medskip

Some consequences of the characterization theorem (not all of them are immediate):

\begin{corollary}\label{bennett-corr}
\ 
\begin{enumerate}
\item For each $n\geq 1$, the class of spectra of order $n$ is closed under $\wedge$, $\vee$, bounded quantifications, substitution of rudimentary functions, explicit transformations and finite modifications. 
\item For each $n\geq 1$, the class of spectra of order $2n$ is closed under $\neg$.
\item 
\label{bennett-corr-iii}
The class of first-order spectra contains the rudimentary relations and ${\mathcal E}_{\star}^2$.
\item The class of second-order spectra strictly contains the rudimentary relations.
\item For each $n\geq 1$, spectra of order $n$ form a subset of spectra of order $n+1$ and a strict subset of spectra of order $n+2$.
\end{enumerate}
\end{corollary}

We propose below a proof of Bennett's theorem.
\medskip

\begin{proof}[Proof of Theorem~\ref{TH Bennett}]
\noindent (\ref{Bennett2}) We first present the second-order case, because it has less technical difficulties.
\medskip

\noindent - First inclusion: 
$\{\spec{\varphi}\ 
|\ \varphi\in \Sol\}\subseteq\{\exists y\leq 2^{x^j}\ R(x,y)\ 
|\ j\geq 1\hbox{ and }R\in \RUD\}$ 
i.e. $\varphi$ has a model with $x$ elements iff $\exists y\leq 2^{x^j}\ R(x,y)$ is true.

W.l.o.g. we may assume that $\varphi$ has no  first-order or second-order 
free variable (just quantify existentially in case there are any). Assume the 
second-order variables appearing in $\varphi$ have arities strictly less than 
$j$. Then we take $y=2^{x^j}$. We encode a second-order variable $Z$ with arity 
$a<j$ by the number $z<2^{x^a}<y$ in the usual way. Hence, every second-order 
quantification $QZ^{(a)}$ in $\varphi$ is translated into the first-order 
bounded quantification $Qz<2^{x^a}<y$. Recall that $Bit(a,b)$ is true iff the bit of rank $b$ of $a$ is $1$. Now, every atomic formula 
$Z(z_1,\ldots,z_a)$ is translated into $Bit(z,z_1+z_2\cdot x+\ldots+z_a\cdot 
x^{a-1})$. 
Every first-order quantification $qz$ in $\varphi$ is translated into the 
bounded quantification $qz<x$. The atomic formulas $z=z'$ in $\varphi$ remain 
unchanged. Let $\varphi'$ denote the obtained formula. Finally, let $R\equiv 
(y=2^{x^j})\wedge \varphi'$.
\medskip

\noindent - Second inclusion: 
$\{\spec{\varphi}\ 
|\ \varphi\in SO\}\supseteq\{\exists y\leq 2^{x^j}\ R(x,y)\ 
|\ j\geq 1\hbox{ and }R\in \RUD\}$ 
i.e. $\exists y\leq 2^{x^j}\ R(x,y)$ is true iff $\varphi$ has a model with $x$ elements.

First, we use three existentially quantified relations, namely $\leq^{(2j)}$ 
which is bound to be a linear ordering over the $j$-tuples of vertices, 
$+^{(3j)}$ which is bound to be the associated addition and $\times^{(3j)}$ 
which is bound to be the associated multiplication. Let us denote by 
$Arithm(\leq,+,\times)$ the first-order sentence expressing this requirement. 
Note that we may now use for free any usual arithmetic predicate on numbers 
bounded by $x^j$ (written in $x$-ary notation, i.e. seen as $j$-tuples of 
integers in $\{0,\ldots,x-1\}$). Next, all variables in $R$, including $x$ and 
$y$, are encoded by $j$-ary second-order variables in $\varphi$ in the usual 
way. For instance if $x=\sum_{l=0}^p2^{i_l}$, we let 
$X=\{(i_0,0,\ldots,0),\ldots,(i_p,0,\ldots,0)\}$.

W.l.o.g. we may assume that all the atomic formulas in $R$ are of type $u\cdot v=w$ 
(concatenation), which we translate into 

\[
\begin{array}{l}
Concat(U,V,W)\equiv\\
\quad \exists \overline{t} V(\overline{t}) 
\wedge\forall\overline{z}(V(\overline{z})\longrightarrow\overline{z}\leq\overline{t}) 
\wedge\forall\overline{z}(U(\overline{z})\longrightarrow\overline{z}\leq(\overline{max}-\overline{t}))\\
\quad\wedge\forall \overline{z}  \left(W(\overline{z})\longleftrightarrow 
\left(\left(\overline{z}\leq\overline{t}\wedge V(\overline{z})\right)\vee 
\left(\overline{z}>\overline{t}\wedge 
U(\overline{z}-\overline{t})\right)\right)\right). 
\end{array}
\]

Note that this sentence 
would be cleaner in dyadic than it is in binary, but the whole encoding would 
also be more complicated because two unary relations are needed to encode an 
integer in dyadic (the set of $1$s and the set of $2$s) because its length is fixed.

In order to translate the bounded quantifications in $R$, we also need the 
following first-order sentence, which expresses the fact that the integers $u$ and $v$ respectively encoded by $U$ and $V$ are such that $u<v$.

\[
\begin{array}{l}
Smaller(U,V)\equiv\\
\exists\overline{z}\left(V(\overline{z})\wedge\neg 
U(\overline{z})\wedge\forall\overline{z'}>\overline{z}\neg 
U(\overline{z'})\right)\\
\vee \exists\overline{z}\left(V(\overline{z})\wedge 
U(\overline{z})\wedge\forall\overline{z'}>\overline{z}\left(\neg 
V(\overline{z'})\wedge\neg U(\overline{z'})\right)\wedge 
\exists\overline{z'}<\overline{z} \left(V(\overline{z'})\wedge\neg 
U(\overline{z'})\right)\right).
\end{array}
\]

Now, let $R'$ be obtained from $R$ by applying the following rules: every bounded first-order quantification $\forall z<z'\ldots$ is translated into the second order quantification $\forall Z^{(j)} Smaller(Z,Z')\longrightarrow \ldots$, and accordingly for $\exists z<z'\ldots$; and every atomic formula $u\cdot v=w$ is translated into $Concat(U,V,W)$. 

It remains to express that $X$ encodes the size $x$ of the domain, which is done 
using the binary notation of $(max,0,\ldots,0)$, which represents the largest 
element of the domain. More precisely, we have $max+1=x$, which translates in 
binary as follows:

\[
\begin{array}{l}
Dom(X)\equiv\\
\forall\overline{z}
(((X(\overline{z})\wedge \forall\overline{z'}<z \neg X(\overline{z'})) \longrightarrow \neg Bit((max,0,\ldots,0),\overline{z}))\\
\wedge ((\exists\overline{z'}<z X(\overline{z'})) \longrightarrow (X(\overline{z}) \longleftrightarrow Bit((max,0,\ldots,0),\overline{z})))\\
\wedge ((\exists\overline{z'}>\overline{z} (X(\overline{z'})\wedge \forall\overline{z''}<\overline{z'} \neg X(\overline{z''}))) \longrightarrow Bit((max,0,\ldots,0),\overline{z})) )\\ 
\end{array}
\]

Finally, $\varphi$ is $\exists 
\leq^{(2j)}\exists+^{(3j)}\exists\times^{(3j)}\exists Y^{(j)}\exists X^{(j)} 
(Arithm(\leq,+,\times)\wedge Dom(X)\wedge R')$.

\medskip

\noindent (\ref{Bennett1}) Next we turn to the first-order case. We consider the proof of the second-order 
case and show how it has to be modified in order to fit to the first-order 
case. Note that the proof is now more tricky, and we use dyadic notation 
because we have to be more precise.
\medskip

\noindent - First inclusion: 
$\{\spec{\varphi}\ |\ \varphi\in FO\}
\subseteq\{\exists y\leq 2^{x^j}\ R(x,y)\ |\ j\geq 1\hbox{ and }R\in \SRUD\}$

The main difference concerning $\varphi$ is that it contains no second-order quantifications. Concerning $R$, we have to deal with two differences: bounded quantifications are now replaced by part-of quantifications ($\forall z_1\harp z_2$ and $\exists z_1\harp z_2$) on the one hand and we have to use concatenation instead of arithmetic on the other hand. 

However, $\varphi$ does contain free second-order variables, say 
$Z_1^{(a_1)},\ldots,Z_k^{(a_k)}$, which we do not encode in the usual way 
because $\SRUD$ does not allow to use arithmetical predicates, hence the $Bit$ 
predicate is not available. Instead, we assume for now that the alphabet is 
$\{1,2,\ast,\star,\bullet\}$ and we first define a provisional predicate $R'(x,y)$. We shall 
explain later how to get rid of the extra symbols $\ast$, $\star$ and $\bullet$ 
to obtain the expected $R(x,y)$.
\medskip

We use the following encoding: if 
$Z=\{(x_1^1,\ldots,x_a^1),\ldots,(x_1^p,\ldots,x_a^p)\}$, with $p\leq x^a$, 
then let $z=\star\ast x_1^1*\ldots*x_a^1\ast\star\ldots\star\ast 
x_1^p*\ldots*x_a^p\ast\star$. Note that we have $|z|\leq x^a\cdot a\cdot 
(|x|+2)$.

Let us define $x_0=\ast x\ast(x-1)\ast\ldots\ast1\ast$, i.e. the dyadic 
representation of $x_0$ is the concatenation of the dyadic representations of 
all integers in $\{1,\ldots,x\}$, separated by $\ast$s. Note that $|x_0|\leq 
x\cdot(|x+2|)<x^2$. Finally, let $y=\bullet z_1\bullet\ldots\bullet z_k\bullet 
x_0\bullet $. Clearly we have $y\leq 2^{x^j}$ for some $j\geq 1$.

Now, $R'(x,y)$ will begin with $\exists z_1\harp  y\ \ldots\ \exists 
z_k\harp  y\ \exists x_0\harp  y\ ((y=\bullet 
z_1\bullet\ldots\bullet z_k\bullet x_0\bullet) \wedge \neg (\bullet 
\harp  z_1)\wedge\ldots\wedge \neg(\bullet \harp  z_k) \wedge 
\neg(\bullet \harp  x_0) \wedge \ldots)$, in order to retrieve the 
significant parts of $y$.

We use $x_0$ to replace every first-order quantification $\forall u \ldots$ 
appearing in $\varphi$ by a part-of quantification $\forall u \harp  
x_0\left(Int(u,x_0)\longrightarrow\ldots\right)$ in $R'$, and similarly for 
$\exists u \ldots$, where $Int(u,x_0)$ means that $u$ is a maximal non-empty 
string of $1$s and $2$s in $x_0$. The most technical part of the proof is to 
write a strictly rudimentary formula $Dom(x_0,x)$ which is true iff $x_0$ has 
the expected form, but for sake of brevity, we do not explicit this formula. 
In particular, note that we now consider the domain as $\{1,\ldots,x\}$ instead of 
$\{0,\ldots,x-1\}$ as we did previously. Finally it is not difficult to write a 
formula $Verif(x_0,z)$ expressing the fact that $z$ has the expected form 
$\star\ast x_1^1*\ldots*x_a^1\ast\star\ldots\star\ast 
x_1^p*\ldots*x_a^p\ast\star$. Namely, take

\[
\begin{array}{l}
Verif(x_0,z)\equiv \exists  u\harp  z\ \\ (\star u\star \harp  
z) \wedge \forall u\harp  z (((\star u\star \harp  z)\wedge \neg 
(\star \harp  u)\wedge u\neq\epsilon)\longrightarrow \exists 
v_1\harp  u\ldots\exists v_a\harp  u\\
(Int(v_1,x_0)\wedge\ldots\wedge Int(v_a,x_0)\wedge u=\ast v_1\ast\ldots\ast 
v_a\ast))\\
\wedge \forall u_1,u_2,\alpha,\beta,\gamma\harp  z(((z=\alpha\star 
u_1\star\beta\star u_2\star\gamma\vee z=\alpha\star u_1\star u_2\star\gamma)\\
\wedge \neg(\star \harp  u_1)\wedge\neg (\star\harp  
u_2))\longrightarrow u_1\neq u_2).
\end{array}
\]

There are two types of atomic formulas in $\varphi$: equalities $z_1=z_2$ and 
atoms $Z(z_1,\ldots,z_a)$. Equalities remain unchanged and $Z(z_1,\ldots,z_a)$ 
is changed into $\star\ast z_1\ast\ldots\ast z_a\ast\star \harp  z$. 
These operations lead to the strictly rudimentary formula $\varphi'$.

Finally, $R'(x,y)$ is $\exists z_1\harp  y \ldots \exists z_k\harp  y \exists x_0\harp  y ((y=\bullet z_1\bullet\ldots\bullet z_k\bullet x_0\bullet) \wedge \neg (\bullet \harp  z_1)\wedge\ldots\wedge \neg(\bullet \harp  z_k) \wedge \neg (\bullet \harp  x_0)
\wedge Dom(x_0,x) \wedge  Verif(x_0,z_1)\wedge\ldots\wedge Verif(x_0,z_k) \wedge \varphi')$.
\medskip

To obtain $R$, it remains to get rid of the alphabet $\{1,2,\ast,\star,\bullet\}$. 
Let $\ast$ be a string of $1$s which is not a subword of $x, x-1,\ldots, 2$ and $1$. 
For instance, $\ast$ could be of length $|x|+1$. Let $\star=2\ast2$ and $\bullet=22\ast22$. 
The final length of $y$ is polynomially longer than it used to be, 
which remains acceptable. 
Finally, take 
$R(x,y)\equiv \exists\ast \harp  y\ 
\exists\star \harp  y\ 
\exists\bullet \harp  y\ 
((\forall u\harp  \ast\ (u=1))\wedge \ast\neq\epsilon\wedge \star=2\ast2\wedge 
\bullet=22\ast22\wedge R')$. 
Note that strictly rudimentary relations do not define predicates referring 
to the length of integers, so that $\ast$ cannot be bound to be some 
specific word like $1^{|x|+1}$.

\medskip 

\noindent - Second inclusion: 
$\{\spec{\varphi}\ |\ \varphi\in FO\}
\supseteq\{\exists y\leq 2^{x^j}\ R(x,y)\ |\ j\geq 1\hbox{ and }R\in \SRUD\}$

The main difference with the second-order case concerning $\varphi$ is that it only contains first-order quantifications. However, we are still free to choose as many free second-order variables as we may need. In particular, we still use usual arithmetic predicates on the ($j$-tuples of) elements of the domain, and the previous first-order sentence $Arithm(\leq,+,\times)$ is still required to hold for this purpose. In addition, we introduce the second-order variables $X_1, X_2$ and $Y_1,Y_2$, both of arity $j$, respectively representing the set of positions where $x$ and $y$ have $1$s and $2$s and no other second-order variables are introduced. Let $Word(X_1,X_2)$ be the sentence expressing the fact that $X_1$ and $X_2$ (and similarly $Y_1,Y_2$) do represent a dyadic word, namely $$\begin{array}{l}Word(X_1,X_2)\equiv\forall\overline{z}\neg(X_1(\overline{z})\wedge X_2(\overline{z}))\\
\wedge\exists\overline{z}\forall\overline{t}((\overline{t}>\overline{z}\longrightarrow(\neg X_1(\overline{t})\wedge\neg X_2(\overline{t})))\wedge(\overline{t}\leq\overline{z}\longrightarrow(X_1(\overline{t})\vee X_2(\overline{t})))).\end{array}$$
Concerning $R$, we may assume w.l.o.g. that it only contains part-of quantifications $qz\harp  x$ and $qz\harp  y$ and no $qz\harp  z'$ for $z'\not\in\{x,y\}$.

The main trick is that a part-of quantification $\exists z\harp  y\ldots$ (for instance) will be replaced by $2j$ first-order quantifications $\exists\overline{z_1}\exists\overline{z_2} \left(\overline{z_1}\leq\overline{z_2}\wedge\ldots\right)$, where $\overline{z_1}$ and $\overline{z_2}$ encode the positions where $z$ begins and ends, as a subword of $y$.

We have to translate the atomic formulas $u\cdot v=w$. W.l.o.g.\ we may rewrite $R$ 
in an equivalent formula by replacing everywhere $u\cdot v=w$ with $(u\cdot 
v= w\wedge u\harp  y\wedge v\harp  y\wedge 
w\harp  y)\vee(u\cdot v= w\wedge u\harp  y\wedge 
v\harp  y\wedge w\harp  x)\vee\ldots\vee(u\cdot v= 
w\wedge u\harp  x\wedge v\harp  x\wedge w\harp  x) 
$. Hence, there are $8$ slightly different cases to be taken care of. We limit 
ourselves with the case $u\cdot v\cdot w\wedge u\harp  y\wedge 
v\harp  y\wedge w\harp  y$. The corresponding formula 
$Concat_{yyy}(\overline{u_1},\overline{u_2},\overline{v_1},\overline{v_2},\overline{w_1},\overline{w_2})$ 
is as follows: 

\[
\begin{array}{l}
\overline{w_2}=\overline{w_1}+\overline{u_2}-\overline{u_1}+\overline{v_2}-\overline{v_1}\\
\wedge \forall\overline{z} 
(\overline{w_1}\leq\overline{z}<\overline{w_1}+\overline{u_2}-\overline{u_1}\longrightarrow\\
((Y_1(\overline{z})\longleftrightarrow Y_1(\overline{z}-\overline{w_1}))\wedge 
(Y_2(\overline{z})\longleftrightarrow Y_2(\overline{z}-\overline{w_1}))))\\
\wedge \forall \overline{z} (\overline{w_1}+\overline{u_2}-\overline{u_1}\leq\overline{z}<\overline{w_2}\longrightarrow\\ 
((Y_1(\overline{z})\longleftrightarrow Y_1(\overline{z}-\overline{w_1}-\overline{u_2}+\overline{u_1}))\wedge
(Y_2(\overline{z})\longleftrightarrow Y_2(\overline{z}-\overline{w_1}-\overline{u_2}+\overline{u_1})))).
\end{array}
\]

Let us denote by $R'$ the obtained sentence.

The last remaining part is to write out a sentence $Dom'(X_1,X_2)$ expressing 
the fact that $X_1,X_2$ encodes (in dyadic) the cardinality of the domain, i.e. 
the successor of the $j$-tuple $(max,0,\ldots,0)$. This is a bit more technical 
than the sentence $Dom(X)$ we used for the binary notation and we do not 
spell it out here. Finally, take $\varphi\equiv Arithm\wedge Word(X_1,X_2)\wedge 
Word(Y_1,Y_2)\wedge Dom(X_1,X_2)\wedge R'$.
\end{proof}

\medskip

\subsubsection*{Connections with complexity classes}

At the beginning of complexity theory, the usual compexity classes 
such as the polynomial hierarchy had not emerged yet. 
So the classes used by Bennett are not standard ones. 
He considers two hierarchies based on space-bounded deterministic 
Turing machines defined in a recursive fashion: 
the base class is of type $FDSpace(f(n))$, 
and the next class has a space bound which is a function in the previous class.

Let us denote by $({\mathcal R}^i)_{i\geq 1}$ the first hierarchy, 
introduced in Ritchie's 1963 paper \cite{ar:Ritchie63}, 
which comes from from his Ph.D. thesis \cite{phd:Ritchie60}. 

\begin{definition}[Ritchie's classes]\ 
\begin{itemize}
\item[-] Let ${\mathcal R}^1$ be the class of functions computable by some (deterministic) Turing machine in space bounded by $b\cdot\max(\overrightarrow{x})$ on input $\overrightarrow{x}$, where $b\geq 1$ is some integer fixed for each machine, i.e. ${\mathcal R}^1=FDSpace({\mathcal O}(2^n))$ in modern notation.
\item[-] For each $i\geq 1$, let us denote by ${\mathcal R}^{i+1}$ the class of functions computable by a Turing machine in space bounded by $B(\overrightarrow{x})$, where $B$ is some function in ${\mathcal R}^i$, fixed for each machine. 
\item[-] For each $i\geq 1$, let us denote by ${\mathcal R}^i_{\star}$ the class of relations whose characteristic functions are in ${\mathcal R}^i$.
\end{itemize}
\end{definition}
It is proved in \cite{ar:Ritchie63} that this hierarchy $({\mathcal R}^i_{\star})_{i\geq 1}$ is strict and that its union corresponds to elementary relations. 
\medskip

Using the same pattern, Bennett introduces a second hierarchy, that we denote by $({\mathcal B}^i)_{i\geq 1}$. 

\begin{definition}[Bennett's classes]\ 
\begin{itemize}
\item[-] Let ${\mathcal B}^1$ be the class of functions computable by some (deterministic) Turing machine in space bounded by $P(\overrightarrow{x})$ on input $\overrightarrow{x}$, where $P$ is some arithmetical polynomial fixed for each machine, i.e. ${\mathcal B}^1=FDSpace((2^{{\mathcal O}(n)}))$ in modern notation.
\item[-] For each $i\geq 1$, let us denote by ${\mathcal B}^{i+1}$ the class of functions computable by a Turing machine in space bounded by $B(\overrightarrow{x})$, where $B$ is some function in ${\mathcal B}^i$, fixed for each machine. 
\item[-]For each $i\geq 1$, let us denote by ${\mathcal B}^i_{\star}$ the class of relations whose characteristic functions are in ${\mathcal B}^i$.
\end{itemize}
\end{definition}

Bennett shows that Ritchie's classes ${\mathcal R}^i_{\star}$ come in between spectra of various orders, but not in a very nice way. In contrast, he proves nice closure properties and an exact intercalation between the classes of spectra of consecutive orders for the classes ${\mathcal B}^i_{\star}$. However, all these classes are too big to be informative concerning relationship between first-order spectra and complexity classes.

In order to state the next theorem, let us denote by ${\mathcal S}^i$ the class of (many-sorted) spectra of formulas of order $i$. For instance $\SPEC$ is the class of unary relations in ${\mathcal S}^1$.

\begin{theorem}\ 
\begin{enumerate}
\item\label{turing-ritchie} ${\mathcal R}^1_{\star}\subseteq {\mathcal S}^3$ and for each $i\geq 2$, ${\mathcal S}^{2i-2}\subseteq {\mathcal R}^i_{\star}\subseteq {\mathcal S}^{2i+1}$. Moreover, for no $i,j\geq 1$ does ${\mathcal R}^i_{\star}={\mathcal S}^j$.
\item\label{turing-bennett} For each $i\geq 1$, ${\mathcal S}^{2i}\subseteq {\mathcal B}^i_{\star}\subseteq {\mathcal S}^{2i+1}$ (equality or strictness is unknown) and ${\mathcal R}^i_{\star}\subsetneq {\mathcal B}^i_{\star}\subsetneq {\mathcal R}^{i+1}_{\star}$. Moreover, ${\mathcal B}^i_{\star}$ is closed with respect to union, intersection, bounded quantifications, substitution of rudimentary functions, explicit transformations and finite modifications.
\end{enumerate}
\end{theorem}

The proof of item (\ref{turing-ritchie}) is based on recursive characterizations of the classes ${\mathcal R}^i_{\star}$, whereas item (\ref{turing-bennett}) is stated without proof.

\subsection{Mo's paper}~\label{se:Mo}

There is a late paper on the recursive aspect of spectra, namely \cite{ar:Mo91}, due to the Chinese logician Mo Shaokui in 1991, only available in Chinese (see the author's English abstract in Mathematics Abstracts of Zentralblatt \cite{misc:Mo-Zbl}).
With the help of Zhu Ping \cite{misc:zhuping}, we have been able to state Mo's result, and we propose a proof  sketch.
\raus{
\begin{quote}
There are two Scholz problems. First, to find a necessary and
sufficient condition for a set of natural numbers to be a finite spectrum
(i.e. the set of satisfiable numbers of a formula in first-order logic);
second, whether the complement of a finite spectrum is necessarily also a
finite spectrum. In this paper it is shown that a set M of natural numbers
is a finite spectrum if and only if $x\in M$ is an ${\mathcal E}\sp 2$ (or
${\mathcal E}\sp 1$ or ${\mathcal E}\sp 0)$ predicate (in the Grzegorczyk hierarchy)
prefixed by existential quantifiers of the form $\exists f<v\uparrow
(v\uparrow h)$, where $\uparrow$ denotes the power function, $h$ is some
concrete natural number, and the bound variables $f$ should be the first
argument of the function $dig\sb v$, where $dig\sb v(x,y)$ denotes the $y$-th
digit (in the $v$-scale) of the natural number $x$. Therefore, the first Scholz
problem is solved. It is also shown that if all the functions in ${\mathcal E}\sp 0$ can be enumerated by a function in ${\mathcal E}\sp 2$, then the
complement of a certain finite spectrum cannot be any finite spectrum.
Hence, under such a condition, the answer to the second Scholz problem is
negative.
\end{quote}
}

\begin{definition}
Let $x,x_1,x_2,\ldots$ and $y,y_1,y_2,\ldots$ be two disjoint sets of variables.
Let $\textsc{Mo}$ be the smallest class of predicates over integers 
containing the relations $x_1+x_2=x_3$, $x_1\times x_2=x_3$ 
(both for variables of type $x$ only) and $Bit(y,x)$ 
(where the first variable is of type $y$, the second of type $x$) 
and closed under Boolean operations and (polynomially) 
bounded quantifications for variables of type $x$ only. 
\end{definition}

Note that a predicate in $\textsc{Mo}$ has two types of variables, 
which do not play similar roles, and that $\textsc{Mo}$ extends 
the rudimentary relations by the use of $Bit(y,x)$ atoms, 
which are not definable because $y$ variables are not allowed 
in the atomic formulas for addition and multiplication. 

\begin{theorem}
$\{\spec{\varphi}\ |\ \varphi\in FO\}=\{\exists y_1\leq 2^{x^{j_1}}\ldots\exists y_k\leq 2^{x^{j_k}}\\ 
R(x,y_1,\ldots,y_k)\ |\ k, j_1,\ldots, j_k\geq 1\hbox{ and }R\in \textsc{Mo}\}$
\end{theorem}
\begin{proof}
It is a slightly modified version of the proof of Bennett's theorem for second-order spectra. 
\medskip

- First inclusion: 
$\varphi$ has a model with $x$ elements iff 
$\exists y_1\leq 2^{x^{j_1}}\ldots\exists y_k\leq 2^{x^{j_k}}\ R(x,y_1,\ldots,y_k)$ is true.

We encode a predicate symbol $Y$ with arity $j$ by the number $y<2^{x^j}$ in the usual way. Hence, every atomic formula $Y(x_1,\ldots,x_j)$ is translated into $\exists x'<x^j (Bit(y,x')\wedge x'=x_1+x_2\cdot x+\ldots+x_j\cdot x^{j-1})$. Every first-order quantification $qx_i$ in $\varphi$ is translated into the bounded quantification $qx_i<x$. The atomic formulas $x_1=x_2$ in $\varphi$ remain unchanged. Let $R$ denote the obtained formula with free variables $x,y_1,\ldots,y_k$.
\medskip

- Second inclusion: $\exists y_1\leq 2^{x^{j_1}}\ldots\exists y_k\leq 2^{x^{j_k}}
R(x,y_1,\ldots,y_k)$ is true iff $\varphi$ has a model with $x$ elements.

First, we use three predicate symbols, namely $\leq^{(2)}$ which is bound to be a linear ordering on the vertices, $+^{(3)}$ which is bound to be the associated addition and $\times^{(3)}$ which is bound to be the associated multiplication. Let us denote by $Arithm_1(\leq,+,\times)$ the first-order sentence expressing this requirement. Note that we may now use for free any usual arithmetic predicate on numbers bounded by $x$. 

Next, every free variable of type $y$ in $R$ and bounded by $2^{x^j}$ is translated into a predicate symbol $Y$ of arity $j$.

W.l.o.g., we may assume that all the bounded quantifications in $R$ are of type $qx'<x^i$ for some $i\geq 1$. The bounded quantification $qx'<x^i$ in $R$ is simply translated into $qx'_1\ldots qx'_i$ and $x'$ is represented by the $i$-tuple $(x'_1,\ldots,x'_i)$.

There are three types of atomic formulas in $R$. Let us first consider formulas $x_1+x_2=x_3$ and $x_1\times x_2=x_3$. Assume we have $x_1<x^{i_1}\leq x^j, x_2<x^{i_2}\leq x^j,x_3<x^{i_3}\leq x^j$, with $j=\max(i_1,i_2,i_3)$. The variables $x_1,x_2,x_3$ correspond to the tuples $(x_1^1,\ldots,x_1^{j}),(x_2^1,\ldots,x_2^{j}),(x_3^1,\ldots,x_3^{j})$ (padding with as many $0$s as necessary). This includes the case $x<x^2$ so that $x$ corresponds to $(0,1,0,\ldots,0)$. Then $x_1+x_2=x_3$ is changed into  $Add_j(x_1^1,\ldots,x_1^{j},x_2^1,\ldots,x_2^{j},x_3^1,\ldots,x_3^{j})$ and $x_1\times x_2=x_3$ is changed into $Mult_j(x_1^1,\ldots,x_1^{j},x_2^1,\ldots,x_2^{j},x_3^1,\ldots,x_3^{j})$, where the formulas $Add_j$ and $Mult_j$ express addition and multiplication on $j$-tuples in $x$-ary notation. The case of atomic formulas $Bit(y,x')$ is dealt with similarly. Assume we have $y<2^{x^j}$, and $x'<x^i$ then there are three possibilities. If $i< j$, then $Bit(y,x')$ is changed into $Y(x'_1,\ldots,x'_i,0,\ldots,0)$. If $i=j$, then $Bit(y,x')$ is changed into $Y(x'_1,\ldots,x'_j)$. If $i>j$, then $Bit(y,x')$ is changed into $Y(x'_1,\ldots,x'_j)\wedge x_{j+1}=0\wedge\ldots\wedge x_i=0$. Similarly, $Bit(y,x)$ (which may also occur in $R$ because $x$ is a free variable of type $x$) translates into $Y(0,1,0,\ldots,0)$ if $Y$ has arity $2$ at least and into $0\neq 0$ (false) if $Y$ is unary. Let us denote by $R'$ the first-order sentence thus obtained.

Finally, $\varphi$ is $Arithm_1(\leq,+,\times)\wedge R'$.
\end{proof}

Note that, in order to uniformize the proofs with that of Bennett's theorem and help comparison, we have slightly modified the original statement in two points. First, Mo uses functional vocabularies, which yields bounds of type $x^{x^j}$ for $y$ type variables and the use of atoms $Digit_x(y,x')=x''$ (meaning ``the digit of rank $x'$ of $y$ in $x$-ary notation is $x''$'') instead of $Bit(y,x)$. Second, the relation $R$ is originally described using Grzegorczyk's classes ${\mathcal E}^0_{\star}$, ${\mathcal E}^1_{\star}$ or ${\mathcal E}^2_{\star}$ instead of $\RUD$.

Finally, concerning Asser's problem (so-called second Scholz problem here), the author's abstract \cite{misc:Mo-Zbl} asserts that: 
\begin{quote}
It is also shown that if all the functions in ${\mathcal E}\sp 0$ can be enumerated by a function in ${\mathcal E}\sp 2$, then the
complement of a certain finite spectrum cannot be any finite spectrum.
Hence, under such a condition, the answer to the second Scholz problem is
negative.
\end{quote}
Hence, the conditional negative solution proposed here seems to be linked to some separation of ${\mathcal E}\sp 0$ and ${\mathcal E}\sp 2$ via diagonalization, which seems unlikely (the classical proof of separation of ${\mathcal E}\sp i$ and ${\mathcal E}^{i+1}$ uses the bound on the growth of the functions in ${\mathcal E}\sp i$). 
\raus{
Note that it is not clear whether ${\mathcal E}^i$ or ${\mathcal E}^i_{\star}$ are intended, because both expressions ``${\mathcal E}^i$ predicate'' and ``${\mathcal E}^i$ function'' are used. A separation of ${\mathcal E}\sp 0_{\star}$ and ${\mathcal E}\sp 2_{\star}$ via diagonalization is even more unlikely.
}

\section{A compendium of questions and conjectures}
\ \\

In this appendix we list, for convenience,
all the Open Questions (OQ)and stated in our survey.

\ 

\subsubsection*{From Section \ref{se:intro}}
\ \\

\begin{description}
\item[OQ \ref{scholz-problem}]
(Scholz)
Characterize the sets of natural numbers
that are first order spectra.

\item[OQ \ref{open-assers-problem}]
(Asser)
Is the complement of a first order spectrum
a first order spectrum?

\item[OQ \ref{asser-msol}]
Is the complement of a spectrum of an $\Msol$-sentence
again
a spectrum of an $\Msol$-sentence?

\item[OQ \ref{Fagin-binary}]
(Fagin)
Is every first order
spectrum the spectrum of a first order sentence of
one binary relation symbol?

\item[OQ \ref{Fagin-simple}]
(Fagin)
Is every first order
spectrum the spectrum of a first order sentence
over simple graphs?

\item[OQ \ref{Fagin-planar}]
(Fagin)
Is every first order
spectrum the spectrum of a first order sentence
over planar graphs?
\end{description}

\ 

\subsubsection*{From Section \ref{se:counting}}
\ \\

We recall a few definitions:
Let $M \subseteq \N^+$, and let
$m_1, m_2, \ldots  $ an enumeration of $M$ ordered
by the size of its elements.
$\chi_M(n)$ is the characteristic function of $M$.
$\eta_M(n)$ is the enumeration function of $M$, i.e.,
$ \eta_M(n) = m_n$ if it exists, and
$\eta_M(n)= 0$ otherwise.
Finally, $\delta_M(n)= \eta_M(n+1)-\eta_M(n)$.

\begin{description}
\item[OQ \ref{counting}]
Which strictly increasing sequences of positive integers,
are enumerating functions of spectra?
For instance, how fast can they grow?

\item[OQ \ref{gaps}]
If $M$ is a spectrum how can $\delta_M(n)$ behave?

\item[OQ \ref{littlewood}]
Let 
$\pi(n)$ be the
counting function of the primes, and let
$li(n)$ be its
approximation by the integral logarithm. 
Define
\\
$
\pi^+ =\{ n : \pi(n) - li(n) > 0 \}
$
and
$
\pi^- =\{ n : \pi(n) - li(n) \geq 0 \}
$.
\\
Are the sets
$\pi^+$ and $\pi^-$ spectra?

\item [OQ \ref{lcounting}]
Let $\phi$ a first order sentence, and
$f_{\phi}$
be the associated
labeled counting function that is monotonically increasing.
Is there a first order sentence $\psi$ such that
for all $n$ we have
$ f_{\phi}(n) = \eta_{\psi}(n) $?

\item[OQ \ref{algebraic}]
Are there any irrational algebraic reals which are spectral?
\item[OQ \ref{automatic}]
Is every automatic real a spectral real?
\end{description}
The binary string complexity of a real in binary presentation
is the function $p_r(m)$ which counts, for each $m$
the number of distinct binary words $w$ of length $m$ occurring in $r$.
\begin{description}
\item[OQ \ref{spectral-lim}]
Does the binary string complexity $p_r(m)$ of
a spectral real $r$ satisfy
$$
\liminf_{m \rightarrow \infty} \frac{p_r(m)}{2^m} < 1
$$
or even
$$
\liminf_{m \rightarrow \infty} \frac{p_r(m)}{2^m} = 1?
$$
\item[OQ \ref{e2-reals}]
Are the $b$-adic $\mathcal{E}^2$-computable reals
$\mathcal{E}^2$-Cauchy computable?
\item[OQ \ref{low}]
Is the inclusion $\mathcal{F}_{low} \subseteq \mathcal{E}^2$ proper?
\end{description}

\ 

\subsubsection*{From Section \ref{recursion}}
\ \\

\begin{description}
\item[OQ \ref{e0e1e2}]
Are the inclusions in
$
{\mathcal E}_{\star}^0 \subseteq {\mathcal E}_{\star}^1 \subseteq {\mathcal E}_{\star}^2
$
proper?

\item[OQ \ref{rud=e0}]
Additionally,
is the inclusion 
$\RUD \subseteq
{\mathcal E}^0_{\star} 
$ proper?

\item[OQ \ref{e2=spec}]
Is the inclusion in
${\mathcal E}_{\star}^2\subseteq\SPEC$
proper?

\item[OQ \ref{rud=spec}]
Is the inclusion $\RUD \subseteq \SPEC$ proper?
\end{description}
\ 

\subsubsection*{From Section \ref{complexity}}
\ \\

\begin{description}
\item[OQ \ref{p=np}]
\ 
\begin{renumerate}
\item
Are any of the inclusions
\\
$\L \subseteq \NL$, $\LINSPACE \subseteq \NLINSPACE$,
$\classP \subseteq \NP$ and $\E \subseteq \NE$ proper?
\item
Do any of the equalities  $\NP = \coNP$ and $\NE = \coNE$ hold?
\end{renumerate}

\item[OQ \ref{lth}]
Are the inclusions
$\L \subseteq \RUD = \LTH  \subseteq \LINSPACE$ proper?

\item[OQ \ref{asser-cat}]
Is every
spectrum the spectrum of a categorical sentence ?

\item[OQ \ref{compl-spec}]
Is there a universal (complete) spectrum $S_0$ and a suitable notion of
reduction such that every spectrum $S$ is reducible to $S_0$?

\item[OQ \ref{lynch}]
Is the inclusion
$d\geq 1$, $\NTime{2^{d\cdot n}}\subseteq \RSpectra{d}{}(+)$ proper?

\item[OQ \ref{fspec}]
Is there a characterization as a complexity class of the classes
$\FSpectra{d}{}$ for all $d\geq 1$ ?

\end{description}
\ 

\subsubsection*{From Section \ref{vocabularies}}
\ \\

\begin{description}
\item[OQ \ref{fagin-one-binary}]
Is every first order spectrum in $\RSpectra{2}{1}$?

\item[OQ \ref{QUE unary hierarchy}]
Is the following hierarchy proper:
\[
\FSpectra{1}{1}\subset
\FSpectra{1}{2} \subseteq
\FSpectra{1}{3} \subseteq  \ldots \subseteq
\FSpectra{1}{k} \subseteq  \ldots?
\]

\item[OQ \ref{rud=fspectra}]
Is $\RUD = \FSpectra{1}{2}$?

\item[OQ \ref{rspectra=fspectra}]
Is $\RSpectra{2}{1} = \FSpectra{1}{}$ or even $\RSpectra{2}{1} = \FSpectra{1}{2}$?

\item[OQ \ref{arity}]
Is the following hierarchy proper:
\[
\RSpectra{1}{}\subseteq
\RSpectra{2}{} \subseteq
\RSpectra{3}{} \subseteq  \ldots \subseteq
\RSpectra{k}{} \subseteq  \ldots?
\]
The same question may be asked for spectra over $i$-ary functions.
\item[OQ \ref{oq:fivar}]
Does the hierarchy $Sp\Fol^k$ collapse at level $3$?
\end{description}
\ 

\subsubsection*{From Section \ref{Ash}}
\ \\

\begin{description}
\item[OQ \ref{ash-1}]
(Ash's constant conjecture)\\
Is it true that
for any finite relational
vocabulary $\sigma$ and any positive integer $k$, the Ash function
$N_{\sigma,k}$ is eventually constant?

\item[OQ \ref{ash-2}]
(Ash's periodic conjecture)\\
Is it true that
for any finite relational
vocabulary $\sigma$ and any positive integer $k$, the Ash function
$N_{\sigma,k}$ is eventually periodic?

\item[OQ \ref{ash-3}]
(Ultra-weak Ash conjecture)\\
Is it true that
for any finite relational vocabulary $\sigma$, for any positive integer
$k$ and for all $i\in\NN^+$, the set $N_{\sigma, k}^{-1}(i)$ is a
spectrum?

\item[OQ \ref{ash-4}]
(Ultra-weak Ash conjecture for classes of structures)\\
Is it true that
for any finite relational vocabulary $\sigma$, for any first-order
$\sigma$-sentence $\Psi$, for any positive integer $k$ and for all
$i\in\NN$, the set $N_{Mod_f(\Psi), k}^{-1}(i)$ is a spectrum?

\end{description}
\ 

\subsubsection*{From Section \ref{semantic}}
\ \\

\begin{description}
\item[OQ \ref{patch-complexity}]
What is the complexity of checking whether a $\tau$-structure
$\mathfrak{A}$ has patch-width
at most $k$,
for a fixed $k$?

\end{description}

\subsubsection*{From Section \ref{semantic}}
\ \\

\begin{description}
\item[OQ\ref{mostowski-K}]
Is every
spectrum representable as $\{n+1\ |\ f(n)=0\}$ for some function $f\in K$?
\end{description}
Recall that $K$ is the class of functions defined in Section \ref{section technical}
Definition \ref{classK}.

\def\cprime{$'$}

\end{document}